\documentclass[twoside,11pt]{amsart}
\usepackage[all]{xy}
\CompileMatrices
\usepackage{amsfonts}
\usepackage{amssymb}
\usepackage{amsthm}
\usepackage{amsmath}
\usepackage{ifthen}
\usepackage{url}
\usepackage[usenames]{color}

 \newtheorem{thm}{Theorem}[section]
 \newtheorem{prop}[thm]{Proposition}
 \newtheorem{cor}[thm]{Corollary}
 \newtheorem{lem}[thm]{Lemma}

\theoremstyle{definition}
\newtheorem{defn}[thm]{Definition}

\theoremstyle{remark}

  1

\renewcommand{\L}{\ifmmode {\mathcal{L}}\else$\mathcal{L}$\ \fi}

\newcommand{\bbC}{\ifmmode {\mathbb{C}}\else$\mathbb{C}$\ \fi}
\newcommand{\bbR}{\ifmmode {\mathbb{R}}\else$\mathbb{R}$\ \fi}
\newcommand{\R}{\mbox{$\Bbb R$}}

\newcommand{\be}{\begin{equation}}
\newcommand{\ee}{\end{equation}}

\newcommand{\fpbar}{\ifmmode {\overline{\mathbb{F}_p}}\else$\mathbb{F}_p$\ \fi}

\newcommand{\fp}{\ifmmode {\mathbb{F}_p}\else$\mathbb{F}_p$\ \fi}
\newcommand{\zp}{\ifmmode {\mathbb{Z}_p}\else$\mathbb{Z}_p$\ \fi}
\newcommand{\transpose}[1]{\text{$^t\!#1$}}

\newcommand{\Z}{\mathbb{Z}}

\newcommand{\M}{\ifmmode {\frak M}\else${\frak M}$ \fi}
\newcommand{\m}{\ifmmode {\frak m}\else$\frak m$ \fi}
\newcommand{\mh}{\ifmmode {\frak m}(H)\else${\frak m}(H)$ \fi}
\newcommand{\p}{\ifmmode {\frak p}\else${\frak p}$\ \fi}
\renewcommand{\P}{\ifmmode {\frak P}\else${\frak P}$\ \fi}
\newcommand{\e}{\ifmmode {\mathcal{E}}\else$\mathcal{E}$ \fi}
\newcommand{\G}{\ifmmode {\mathcal{G}}\else${\mathcal{G}}$\ \fi}
\newcommand{\A}{\ifmmode {\mathcal{A}}\else${\mathcal{ A}}$\ \fi}

\newcommand{\Qp}{\ifmmode {{\Bbb Q}_p}\else${\Bbb Q}_p$\ \fi}
\newcommand{\qp}{\ifmmode {{\Bbb Q}_p}\else${\Bbb Q}_p$\ \fi}
\newcommand{\ql}{\ifmmode {{\Bbb Q}_l}\else${\Bbb Q}_l$\ \fi}
\newcommand{\Q}{\ifmmode {\Bbb Q}\else${\Bbb Q}$\ \fi}
\newcommand{\q}{\ifmmode {\Bbb Q}\else${\Bbb Q}$\ \fi}

\def\sectionnam{\@empty}
\def\subsectionnam{\@empty}

\begin{document}

\title[The torsion-congruences for $p$-adic $L$-functions of unitary groups] 
 {Non-abelian $p$-adic $L$-functions and Eisenstein series of unitary groups; the CM method.}

\author{Thanasis Bouganis}
\address{Universit\"{a}t Heidelberg\\ Mathematisches Institut\\
Im Neuenheimer Feld 288\\ 69120 Heidelberg, Germany.}
\email{bouganis@mathi.uni-heidelberg.de}
\thanks{The author acknowledges support by the ERC}


\maketitle

In this work we prove the so-called ``torsion congruences'' between
abelian $p$-adic $L$-functions that are related to automorphic
representations of definite unitary groups. These congruences play a
central role in the non-commutative Iwasawa theory as it became
clear in the works of Kakde, Ritter and Weiss on the non-abelian
Main Conjecture for the Tate motive. We tackle these congruences for
a general definite unitary group of $n$ variables and we obtain more
explicit results in the special cases of $n=1$ and $n=2$. In both of
these cases we also explain their implications for some particular
``motives'', as for example elliptic curves with complex
multiplication. We use two approaches; both of them rely on the
doubling-method developed by Shimura, Garrett, Piateski-Shapiro and
Rallis. The first one, which is presented in this paper, is based on
the fact that the special values under consideration can be obtained
as values of Siegel-type Eisenstein series of $U(n,n)$ on CM points.
The second one, which is presented in \cite{Bouganis4}, is based on
the existence of particular Klingen-type Eisenstein series of
$U(n+1,1)$ whose constant term under a proper Fourier-Jacobi
expansion involves the $L$-values that we consider. In
\cite{Bouganis5} we apply the methods developed in this work in
order to tackle also the so-called M\"{o}bious-Wall congruences and
we also discuss the possibility to extend our methods to handle
special values of indefinite unitary groups.
\newline

\section{Introduction}

In \cite{CFKSV,FK} a vast generalization of the Main Conjecture of
the classical (abelian) Iwasawa theory to a non-abelian setting was
proposed. As in the classical theory, the non-abelian Main
Conjecture predicts a deep relation between an analytic object (a
non-abelian $p$-adic $L$-function) and an algebraic object (a Selmer
group or complex over a non-abelian $p$-adic Lie extension).
However, the evidences for this non-abelian Main Conjecture are
still very modest. One of the central difficulties of the theory
seems to be the construction of non-abelian $p$-adic $L$-functions.
Actually, the only known results in this direction are mainly
restricted to the Tate motive, initially for particular totally real
$p$-adic Lie extensions (see \cite{Hara,Kakde,Kato,RW}) and later
for a large family of totally real $p$-adic Lie extension as it is
shown by Ritter and Weiss in \cite{RW,RW2} and Kakde \cite{Kakde2}.

For other motives besides the Tate motive not much is known. For
elliptic curves there are some evidences for the existence of such
non-abelian $p$-adic $L$-functions offered in
\cite{Bouganis2,DelbWard} and also some computational evidences
offered in \cite{DelbWard2,Dokchitsers}. Also, there is some recent
progress, achieved in \cite{BV}, for elliptic curves with complex
multiplication defined over $\mathbb{Q}$ with repsect the $p$-adic
Lie extension obtained by adjoing to $\mathbb{Q}$ the $p$-power
torsion points of the elliptic curve.

The main aim of this work, as well as its companion work
\cite{Bouganis4}, is to tackle the question of the existence of
non-abelian $p$-adic $L$-functions for ``motives'', whose classical
$L$-functions can be studied through $L$-functions of automorphic
representations of definite unitary groups. In this work we will
prove the so called ``torsion congruences'' (to be explained below)
for these motives. In a second part of this work \cite{Bouganis5} we
use our approach to tackle also the so called M\"{o}bious-Wall
congruences (as for example are described in \cite{RW2}). Without
going into details, we simply mention here that these results allow
one to conclude the existence of the non-abelian $p$-adic
$L$-function in the $K_1(\widehat{\Lambda(G)_{S^*}})$. The stronger
result, that the non-abelian $p$-adic $L$-function actually lies in
$K_1(\Lambda(G)_{S^*})$, as is conjectured in \cite{CFKSV}, needs,
with the present knowledge, one to assume that the classical abelian
Main Conjecture holds for all the subfields of the $p$-adic Lie
extension that corresponds to $G$.

\textbf{The ``torsion-congruences'' for motives:} Let $p$ be an odd
prime number. We write $F$ for a totally real field and $F'$ for a
totally real Galois extension with $\Gamma:=Gal(F'/F)$ of order $p$.
We assume that the extension is unramified outside $p$. We write
$G_F:=Gal(F(p^\infty)/F)$ where $F(p^\infty)$ is the maximal abelian
extension of $F$ unramified outside $p$ (may be ramified at
infinity). We make the similar definition for $F'$. Our assumption
on the ramification of $F'/F$ implies that there exist a transfer
map $ver:G_F \rightarrow G_{F'}$, which induces also a map
$ver:\mathbb{Z}_p[[G_F]] \rightarrow \mathbb{Z}_p[[G_{F'}]]$,
between the Iwasawa algebras of $G_F$ and $G_{F'}$, both of them
taken with coefficients in $\mathbb{Z}_p$. Let us now consider a
motive $M/F$ (by which we really mean the usual realizations of it
and their compatibilities) defined over $F$ such that its $p$-adic
realization has coefficients in $\mathbb{Z}_p$. Then under some
assumptions on the critical values of $M$ and some ordinarity
assumptions at $p$ (to be made more specific later) it is
conjectured that there exists an element $\mu_F \in
\mathbb{Z}_p[[G_F]]$ that interpolates the critical values of $M/F$
twisted by characters of $G_F$. Similarly we write $\mu_{F'}$ for
the element in $\mathbb{Z}_p[[G_{F'}]]$ associated to $M/{F'}$, the
base change of $M/F$ to $F'$. Then the so-called torsion congruences
read
\[
ver(\mu_{F}) \equiv \mu_{F'} \mod{T},
\]
where $T$ is the trace ideal in $\mathbb{Z}_p[[G_F']]^\Gamma$
generated by the elements $\sum_{\gamma \in \Gamma} \alpha^\gamma$
with $\alpha \in \mathbb{Z}_p[[G_F']]$. These congruences have been
proved by Ritter and Weiss \cite{RW} for the Tate motive and under
some assumptions by the author \cite{Bouganis2} for $M/F$ equal to
the motive associated to an elliptic curve with complex
multiplication. In this work we prove these congruences for motives
that their $L$-functions can be studied by automorphic
representations of definite unitary groups.

\textbf{The general setting of this work:} We keep the notations
already introduced above. We now write $K$ for a totally imaginary
quadratic extension of $F$, that is $K$ is a CM field. On our prime
number $p$ we put the following ordinary assumption: all primes
above $p$ in $F$ are split in $K$. As before we consider a totally
real Galois extension $F'$ of $F$ of degree $p$ that is ramified
only at $p$. We write $K':=F'K$, a CM field with ${K'}^+=F'$. Now we
fix, once and for all, the embeddings $incl_\infty :
\overline{\mathbb{Q}} \hookrightarrow \mathbb{C}$ and $incl_p :
\overline{\mathbb{Q}} \hookrightarrow \mathbb{C}_p$. Next we fix,
with respect to the fixed embeddings $(incl_\infty,incl_p)$ an
ordinary CM type $\Sigma$ of $K$ and denote this pair by
$(\Sigma,K)$. We recall that $\Sigma$ is called ordinary (see
\cite{Katz2}) when the following condition is satisfied: ``whenever
$\sigma \in \Sigma$ and $\lambda \in \Sigma^\rho$ ($\rho$ is the
complex conjugation), the $p$-adic valuations induced from the
$p$-adic embeddings $incl_p \circ \sigma$ and $incl_p \circ \lambda$
are inequivalent''. We note that the splitting condition on $p$
implies the existence of such an ordinary CM type. We consider the
induced type $\Sigma'$ of $\Sigma$ to $K'$. That is, we fix a CM
type for $K'$ such that for every $\sigma \in \Sigma'$ we have that
its restriction $\sigma_{|K}$ to $K$ lies in $\Sigma$. We write
$(\Sigma',K')$ for this CM type, and we remark that this is also an
ordinary CM type. In addition to the splitting condition we also
impose the condition that the reflex field $E$ of $(K,\Sigma)$ has
the property that $E_w = \mathbb{Q}_p$, where $w$'s are the places
of $E$ corresponding to the embeddings $E \hookrightarrow
\overline{\mathbb{Q}} \hookrightarrow \mathbb{C}_p$. For example
this is the case if $p$ does not ramify in $F$ or if the type
$(K,\Sigma)$ is the lift of a type $(K_0,\Sigma_0)$ where $K_0$ is a
quadratic imaginary field, such that $K/K_0$ is a Galois extension
and $p$ splits in $K_0$.

Now we are ready to define the motives $M/F$ that appear in this
work. These are going to be of the form $M(\psi)/F \otimes M(\pi)/F$
where $M(\psi)/F$ and $M(\pi)/F$ are motives over $F$ to be defined
now. We start with $M(\psi)/F$ as first.

Let $\psi$ be a Hecke character of $K$ and assume that its infinite
type is $k\Sigma$ for some integer $k\geq 1$. We write $M(\psi)/F$
for the motive over $F$ that is obtained by ``Weil Restriction'' to
$F$ from the rank one motive over $K$ associated to $\psi$. In
particular we have that $L(M(\psi)/F,s)=L(\psi,s)$ or more generally
for a finite character $\chi$ of $G_F$ we have
\[
L(M(\psi) \otimes \chi,s)=L(\psi\tilde{\chi},s),
\]
where $\tilde{\chi}=\chi \circ N_{K/F}$, the base change of $\chi$
to $G(KF(p^\infty)/K)$. Now we consider the character $\psi':=\psi
\circ N_{K'/K}$, the base change of $\psi$ from $K$ to $K'$. It is a
Hecke character of infinite type $k\Sigma'$. Moreover we have that
$M(\psi')/F'$ is the base change of $M(\psi)/F$ to $F'$.

For the motives of the form $M(\pi)/F$ mentioned above we need to
introduce some notation from the theory of unitary groups. We
consider now a hermitian space $(W,\theta)$ over $K$, that means
that $W$ is a vector space over $K$, we write $n$ for its dimension,
and $\theta$ is a non-degenerate hermitian form on it. Moreover we
assume that the signature of the form $\theta_{\sigma}$ on the
complex vector space $W \otimes_{K,\sigma} \mathbb{C}$ is the same
for every embedding $\sigma : K \hookrightarrow \mathbb{C}$ in
$\Sigma$. In particular this implies that our hypothesis on the
splitting of the primes above $p$ in $K$ is the usual ordinary
condition; for a more general ordinarity condition the reader should
see \cite[page 8]{HLS2}. We write $U(\theta)$ for the corresponding
unitary group. That is, as a functor from the category of
$F$-algebras to the category of groups is defined by
\[
GU(\theta)(R):= \big\{ g \in GL_{K \otimes_F R}(W \otimes_F R) |
\theta(gx,gy) =\nu(g)\theta(x,y),\,\,\,x,y \in W \otimes_{F}
R,\,\,\,\nu(g) \in R^\times \big\},
\]
and
\[
U(\theta)(R)=\big\{ g \in GU(\theta)(R) | \nu(g)=1  \big\},
\]
We let $U(\theta')$ be the group $Res_{F'/F}U(W)/F'$, that is the
unitary group corresponding to $(W',\theta)$ where $W'= W \otimes_K
K'$. The $F$-rational points of $U(\theta')$ are the $F'$-rational
points of $U(\theta)$. We consider now a motive $M(\pi)/F$ over $F$
such that there exists an automorphic representation $\pi$ of some
unitary group $U(\theta)(\mathbb{A}_F)$ with the property that the
$L$-function $L(M(\pi)/K,s)$ of $M(\pi)/K$ over $K$ is equal to
$L(\pi,s)$. We now write $\pi'$ for the base change of $\pi$ to
$U(\theta')$. This exists in this general setting only conjecturally
by Langlands' functoriality conjectures but in the cases of interest
that we are going to consider later it is known to exist. Then we
have that
\[
L(M(\pi)/K',s)=L(\pi',s)=L(M(\pi')/K',s).
\]

Our aim in this work is to prove the torsion congruences for the
motive $M/F:=M(\psi)/F \otimes M(\pi)/F$ under the following three
assumptions
\begin{enumerate}
\item The $p$-adic realizations of $M(\pi)$ and $M(\psi)$ have
$\mathbb{Z}_p$-coefficients.
\item The motive $M(\pi)/F$ is associated to an automorphic
representation $\pi$ of a definite unitary group.
\item Let $n$ be the number of variables of the unitary group associated to
$M(\pi)/F$. Then for the weight of the character $\psi$ we have $k
\geq n$.
\end{enumerate}

Now we indicate some cases of special interest that are included in
the motives that we described above.

\textbf{The case $n=1$:} The main application in this case is
obtained with $\pi$ trivial. In this setting, our theorem proves the
``torsion congruences'' for elliptic curves with complex
multiplication, or in general for Hilbert modular forms of CM type.
Results in this direction have been also obtained in our previous
work \cite{Bouganis2} on these congruences. However we stress that
we do not only reobtain these results with our new methods but also
improve on the assumptions that we made there. Actually we obtain
the same result almost unconditionally. Finally we mention that in
\cite{Bouganis2} the main ingredient was the Eisenstein measure of
Katz as in \cite{Katz2} and is related to the automorphic theory of
the group $GL_2/F$, that is Hilbert modular forms. In this work we
use the automorphic theory of unitary groups and hence hermitian
modular forms.

\textbf{The case $n=2$:} Let us now discuss an application of the
case $n=2$. We consider a Hilbert cuspidal form $f$ of $F$, which is
assumed to be a normalized newform. We write $N_{f}$ for its
conductor. We assume that $N_f$ is square free and relative prime to
$p$. We now impose the following assumptions on $f$.
\begin{enumerate}
\item $f$ has a trivial Nebentypus.
\item There exists a finite set $S$ of finite places of $F$ such that
we have (i) $ord_v(N_f) \neq 0$ for all $v \in S$, (ii) for $v \in
S$ we have that $v$ is inert in $K$ and finally (iii) $\sharp S +
[F:\mathbb{Q}]$ is even.
\end{enumerate}
Let us write $D/F$ for the totally definite quaternion algebra that
we can associate to the set $S$, i.e. $D$ is ramified at all finite
places $v \in S$ and also at all infinite places. Note that our
assumptions imply that there exists an embedding $K \hookrightarrow
D$. If we write $\pi'$ for the cuspidal automorphic representation
of $GL_2(\mathbb{A}_F)$ associated to $f$ then our assumptions imply
that there exists a Jacquet-Langlands correspondence $\pi:=JL(\pi')$
to $D^\times(\mathbb{A}_F)$. As we will explain later there exists
an isomorphism
\[
(D^\times \times K^\times)/F^\times \cong GU(\theta)(F)
\]
for some totally definite two dimensional Hermitian form
$(W,\theta)$. In particular, the representation $\pi$ induces an
automorphic representation, by abuse of notation, $\pi$ on
$GU(\theta)$ and by restriction to $U(\theta)$. Moreover it is known
that
\[
L(\pi,s)=L(BC(\pi'),s),
\]
where $BC(\pi')$ is the base-change of $\pi'$ from
$GL_2(\mathbb{A}_F)$ to $GL_2(\mathbb{A}_K)$. In particular we may
pick $M(\pi)/F$ above to be the motive associated to the Hilbert
modular form $f$. This explains our interest in the case $n=2$. We
also remark here that Ming-Lun Hsieh in \cite{Hsieh2} has made
important progress with respect to the classical abelian Iwasawa
Main Conjecture of such motives, i.e. $M(\pi)/F \times M(\psi)/F$.

\textbf{The Main Theorem:} Now we are ready to state the main
theorem of this work.
\begin{thm} Assume that the prime $p$ is unramified in $F$ (but may ramify in
$F'$). Then we have,
\begin{enumerate}
\item For $n=1$: The torsion congruences hold true.
\item For $n=2$: We write $(\pi,\pi)$ for the standard normalized
Peterson inner product of $\pi$ (it will be made explicit later). If
$(\pi,\pi)$ has trivial valuation at $p$ then the torsion
congruences hold true.
\end{enumerate}
\end{thm}

Before we discuss the general strategy for proving the above theorem
we would like to remark that the condition, that $p$ is unramified
at $F$ is imposed because up to date the so-called $q$-expansion
principle (in its $p$-integral form) is not known for the group
$U(n,n)/F$ when $p$ is ramified in $F$.

\textbf{Strategy of the proof:} The key idea that this work and in
its continuation \cite{Bouganis4} is based is the following: Special
values of $L$ functions of unitary representations can be realized
with the help of the doubling method either (i) as values of
hermitian Siegel-type Eisenstein series on CM points of Hermitian
domains or (ii) as constant terms of hermitian Klingen-type
Eisenstein series for some proper Fourier-Jacobi expansion. We
explain briefly the two approaches. Approach A below is used in this
work and approach B is done in \cite{Bouganis4}.

\textbf{Approach A (Values of Eisenstein series on CM points):} In
this approach we consider Siegel-type Eisenstein series of the group
$U(n,n)$ with the property that their values at particular CM points
are equal to the special $L$-values that we want to study. The CM
points are obtained from the doubling method as indicated by the
embedding
\[
U(n,0) \times U(0,n) \hookrightarrow U(n,n)
\]
Then we make use of the fact that the CM pairs $(K,\Sigma)$ and
$(K',\Sigma')$ that we consider are closely related (i.e. the second
is induced from the first) which allows us to relate the various CM
points over $K$ and $K'$. Then we use the diagonal embedding,
induced from the embedding $K \hookrightarrow K'$, between the
symmetric space of $U(n,n)_{/F}$ and that of
$Res_{F'/F}U(n,n)_{/F'}$ to relate the Eisenstein series over the
different fields. This is also the idea that was used in
\cite{Bouganis2}.

\textbf{Approach B (Constant term of Fourier-Jacobi expansions):} In
this approach we obtain Klingen-type Eisenstein series of the group
$U(n+1,1)$ with the property that the constant term of their
Fourier-Jacobi expansion is related with the special values that we
want to study. Then again we use the embedding $K \hookrightarrow
K'$ to relate these Klingen-type Eisenstein series over the
different fields and hence also to obtain a relation between their
constant terms. The main difficulty here is that the Klingen-type
Eisenstein series have a rather complicated Fourier-Jacobi
expansion, which makes hard the direct study of the arithmetic
properties of these Eisenstein series. However the Klingen-type
Eisenstein series are obtained with the help of the pull-back method
from Siegel-type Eisenstein series of the group $U(n+1,n+1)$ using
the embedding
\[
U(n+1,1) \times U(0,n) \hookrightarrow U(n+1,n+1).
\]
The Siegel-type Eisenstein series have a much better understood
Fourier expansion, which turns out it suffices to study also the
Klingen-type Eisenstein series.

\textbf{Organization of the article:} This article is organized as
follows. The next section serves as an introduction to the theory of
hermitian forms, that is automorphic forms associated to unitary
groups both from the classical complex analytic point of view as
well as the arithmetic. Needless to say that nothing in that section
is new. In section 3 we introduce the Eisenstein measure studied by
Harris, Li and Skinner in \cite{HLS1,HLS2} plus some important input
from the work of Ming-Lun Hsieh \cite{Hsieh1,Hsieh2}. Also in this
section, up to some small modifications, there is not much new
material. In the next section we construct the measures $\mu_{F}$
and $\mu_{F'}$ that appear in the ``torsion-congruences''. These
measures are obtained by evaluating the Eisenstein measure of
Harris, Li and Skinner at particular CM points of $U(n,n)$. This
construction is implicit in the papers \cite{HLS1,HLS2} and it will
appear in full details in the forthcoming work of Eischen, Harris,
Li and Skinner \cite{EHLS}. For the needs of our work we will try to
provide here some parts of this construction restricting ourselves
only in the cases of interest. The main part of this work is in
section 5 where we prove congruences between Siegel-type Eisenstein
series. In section 6 we use these congruences between the Eisenstein
series to establish the ``torsion congruences'' for the various
motives that we make explicit in the introduction. Finally there is
an appendix where we simply reformulate a result of Ritter and Weiss
in \cite{RW}.

\textbf{Acknowledgments:} The author would like to thank Prof.
Coates and Prof. Venjakob for their interest in this work, which has
been a source of encouragement for the author.

\section{Automorphic forms on unitary groups and their $p$-adic counterparts.}

As we indicated in the introduction, in this section we simply
recall the definition and fix the notation of the key objects
(automorphic forms of unitary groups, Mumford Objects e.t.c.) that
we are going to use later. Our references are the two books of
Shimura \cite{Shimurabook1,Shimurabook2} and the papers
\cite{Eischen,HLS1,HLS2,Hsieh1,Hsieh2}, where all the material of
this section can be found. Actually we indicate separately, at each
paragraph, the references that we closely followed while writing
this section and hence the reader can find there more details if
he/she wishes.

Let $F$ be a field (local or global) of characteristic different
from two and we consider a couple $(K,\rho)$ of an $F$-algebra of
rank two and an $F$-linear automorphism of $K$. That is, $K$ is
either (i) a quadratic extension of $F$ and $\rho$ is the
non-trivial element of $Gal(K/F)$ or (ii) $K = F \times F$ and
$(x,y)^\rho=(y,x)$ for $(x,y) \in F \times F$. We will always
(except when we indicate otherwise) write $\mathfrak{g}$ for the
ring of integers of $F$ and $\mathfrak{r}$ for the ring of integers
of $K$ in case (i) and $\mathfrak{r}=\mathfrak{g} \times
\mathfrak{g}$ in case (ii).

Let now $V$ be a $K$-module isomorphic to $K_m^1$ and let
$\epsilon=\pm 1$. By an $\epsilon$-hermitian form on $V$ we mean an
$F$-linear map $\phi: V \times V \rightarrow K$ such that, (i)
$\phi(x,y)^\rho=\epsilon \phi(y,x)$ and (ii)
$\phi(ax,by)=a\phi(x,y)b^\rho$ for every $a,b \in K$. Assuming
$\phi$ is non-degenerate we define the algebraic group $GU(\phi)/F$
over $F$ as the algebraic group representing the functor from
$F$-algebras to groups:
\[
GU(\phi)(R):= \big\{ g \in GL_{K \otimes_F R}(V \otimes_F R) |
\phi(gx,gy) =\nu(g)  \phi(x,y),\,\,\,\nu(g) \in R^{\times} \big\},
\]
for an $F$-algebra $R$. Similarly we make the definition for
$U(\phi)/F$ by
\[
U(\phi)(R):= \big\{ g \in GL_{K \otimes_F R}(V \otimes_F R) |
\phi(gx,gy) =\phi(x,y) \big\}.
\]

\textbf{Complex analytic hermitian forms (see \cite[page
30]{Shimurabook2}):} We now pick $F =\mathbb{R}$ and $K =\mathbb{C}$
above and as $\rho$ the usual complex conjugation. We consider the
pair $(V,\phi)$ with $V=\mathbb{C}_n^1$ and with respect the
standard basis we write
\[
\phi=\left(
       \begin{array}{ccc}
         0 & 0 & -i1_r \\
         0 & \theta & 0 \\
         i1_r & 0 & 0 \\
       \end{array}
     \right)
\]
where $\theta \in GL_t(\mathbb{C})$ with $\theta^* = \theta > 0$.
That is, $-i\phi$ is a skew-Hermitian form and $\phi$ has signature
$(r+t,r)$ with $n=2r+t$. For the moment we assume that $r > 0$.

We now describe the (unbounded) symmetric spaces attached to this
unitary group as well as the operation of the unitary group on these
symmetric spaces. We put
\[
\mathfrak{Z}^{\phi}:=\mathfrak{Z}^r_\theta:=\left\{\left(
                               \begin{array}{c}
                                 x \\
                                 y \\
                               \end{array}
                             \right)
 \in \mathbb{C}_r^{r+t} | x \in \mathbb{C}_r^r, y \in
\mathbb{C}_r^t, i(x^*-x) > y^* \theta^{-1}y)\right\}
\]
For $t=0$, we have that $U(\phi)(\mathbb{R})$ is isomorphic to
$U(n,n)(\mathbb{R})$. We write $\mathbb{H}_n$ for its symmetric
space. We consider now an element $\alpha \in
G^\phi:=GU(\phi)(\mathbb{R})$ written as
\[
\alpha=\left(
         \begin{array}{ccc}
           a & b & c \\
           g & e & f \\
           h & l & d \\
         \end{array}
       \right)
\]
with $a,d \in \mathbb{C}_r^r$ and $e\in \mathbb{C}_t^t$. Then we
define an action of $G^\phi$ on
$\mathfrak{Z}:=\mathfrak{Z}^r_\theta$ by
\[
\alpha \left(
         \begin{array}{c}
           x \\
           y \\
         \end{array}
       \right) = \left(
                   \begin{array}{c}
                     x' \\
                     y' \\
                   \end{array}
                 \right)
\]
with $x' = (ax + by +c)(hx + ly + d)^{-1},\,\,\,y'=(gx+ey+f)(hx + ly
+ d)^{-1}$. Moreover we define the following factors of automorphy
\[
\lambda : G^\phi \times \mathfrak{Z} \rightarrow
GL_{r+t}(\mathbb{C}),\,\,\,\mu:G^\phi \times \mathfrak{Z}
\rightarrow GL_{r}(\mathbb{C})
\]
by
\[
\lambda(\alpha,z) = \left(
                     \begin{array}{cc}
                       \bar{h} x^t + \bar{d} & \bar{h}y^t+i\overline{l\theta} \\
                       -i\bar{\theta}^{-1}(\bar{g}x^t+\bar{f}) & -i\bar{\theta}^{-1}\bar{g}y^t + \bar{\theta}^{-1}\overline{e\theta} \\
                     \end{array}
                   \right)
\]
and
\[
\mu(\alpha,z)=hx+ly+d
\]
for $z=\left(
         \begin{array}{c}
           x \\
           y \\
         \end{array}
       \right) \in \mathfrak{Z}$. Finally we define
       $j(\alpha,z):=det(\mu(\alpha,z)$.

We now consider a CM-type $(K,\Sigma)$. We write $\mathbf{a}$ for
the set of archimedean places of $K$ determined by $\Sigma$ and we
define the set $\mathbf{b}:=\mathbf{a} \cup \rho \,\,\mathbf{a}$. We
fix a hermitian space $(V,\phi)$ over $K$ and consider the symmetric
space
\[
\mathbb{H}:=\mathbb{H}_{\phi}:=\prod_{v \in \mathbf{a}}
\mathfrak{Z}^{\phi_v}.
\]
The group $G_+:=G^{\phi}_{\mathbb{R}+}:=\prod_{v \in
\mathbf{a}}G^\phi_v:=\prod_{v \in \mathbf{a}}G^\phi$ operates on
$\mathbb{H}_{\phi}$ componentwise through the operation of each
component $G^\phi_v$ on $\mathfrak{Z}^{\phi_v}$ described above. We
write $(m_v,n_v)$ for $v \in \mathbf{a}$ for the type of $\phi_v$
and consider a finite dimensional rational representation on a
complex vector space $X$:
\[
\omega:\prod_{v \in \mathbf{a}} (GL_{m_v}(\mathbb{C}) \times
GL_{n_v}(\mathbb{C})) \rightarrow GL(X).
\]

For a function $f : \mathbb{H} \rightarrow X$ and $\alpha \in G_+$
we define the functions $f||_\omega \alpha : \mathbb{H} \rightarrow
X$ and $f|_\omega \alpha : \mathbb{H} \rightarrow X$ by
\[
(f||_\omega \alpha)(z):=\omega\left(\prod_{v \in \mathbf{a}}
(\mu(\alpha_v,z_v) \times \lambda(\alpha_v,z_v))\right)^{-1}
f(\alpha z),\,\,\,f|_\omega \alpha := f||_\omega
(\nu(\alpha)^{-1/2}\alpha).
\]

\begin{defn} Let $\Gamma$ be a congruence subgroup of $G$. Then a function $f:\mathbb{H} \rightarrow X$
is called a hermitian modular form for the congruence subgroup
$\Gamma$ of type $\omega$ if
\begin{enumerate}
\item $f$ is holomorphic,
\item $f|_\omega \gamma = f$ for all $\gamma \in \Gamma$,
\item $f$ is holomorphic at cusps.
\end{enumerate}
\end{defn}

In this work we will mainly be concerned with scalar weight
hermitian forms. That is, for an element $k \in
\mathbb{Z}^{\mathbf{b}}$ we may consider the representation
$\omega(x) := det(x)^k$. Then the above modular condition reads
\[
(f||_k\alpha)(z)=j_\alpha(z)^{-k}f(\alpha z),\,\,\,\forall \alpha\in
\Gamma,
\]
where $j_\alpha(z)^{-k}:=j(\alpha,z)^{-k}:=\prod_{v \in \mathbf{a}}
det\left(\mu(\alpha_v,z_v))^{-k_v}det(\lambda_v(\alpha_v,z_v))^{-k_{v\rho}}\right)$.
As in Shimura \cite[page 32]{Shimurabook2} one may write
\[
j_\alpha(z)^{-k}=\prod_{v \in
\mathbf{a}}det(\alpha_v)^{k_{v\rho}}det(\mu(\alpha_v,z_v))^{-k_v-k_{v\rho}}.
\]
In this work we always consider weights $k \in
\mathbb{Z}^{\mathbf{b}}$ of the form $k_{v\rho}=0$ and $k_v=k \in
\mathbb{Z}$ for all $v \in \mathbf{a}$. This space we will denote by
$M_k(\Gamma)$.

\textbf{Unitary automorphic forms (see \cite[page
80]{Shimurabook1}):} We write $G^\phi$ for $U(\phi)/F$ and
$G^\phi_{\mathbb{A}}:=G^\phi(\mathbb{A}_F):=\prod_{v \in
\mathbf{h}}G^\phi(F_v) \prod_{v \in \mathbf{a}}G^\phi(\mathbb{R})$
for the adelic points of the unitary group $G^\phi$. We define
$C_{\mathbf{a}}:=\{\alpha \in G^\phi_{\mathbf{a}} |
\alpha(\mathbf{i})=\mathbf{i}\}$. We say that
$\mathbf{f}:G^\phi_{\mathbb{A}} \rightarrow \mathbb{C}$ is an
(unitary) automorphic form of weight $k \in \mathbb{Z}^\mathbf{b}$
if there exist an open compact subgroup $D$ of $G^\phi_\mathbf{h}$
such that for all $\alpha \in G^\phi(F)$ and $w \in DC_{\mathbf{a}}$
we have
\[
\mathbf{f}(\alpha x w)=j(\alpha,\mathbf{i})^{-k}\mathbf{f}(x).
\]
The relation between classical hermitian forms and unitary
automorphic forms is as follows. We pick $q \in G^\phi_\mathbf{h}$
and define $\Gamma_q:=G^\phi(F) \cap qDq^{-1}$. Then the function
$f_q:\mathbb{H} \rightarrow \mathbb{C}$ defined by the rule
\[
\mathbf{f}(qy)=(f_q||_k y)(\mathbf{i}),\,\,\forall y \in
G^\phi_\mathbf{a}
\]
satisfies $f_q||_k \gamma=f_q$ for all $\gamma \in \Gamma_q$. We
call $\mathbf{f}$ a unitary automorphic form if the $f_q$'s are
hermitian modular forms for all $q \in G_{\mathbf{h}}$. We denote
this space by $\mathcal{M}_k(D)$. As it is well-known is we fix a
decomposition $G^\phi_\mathbb{A}=\coprod_{q\in \mathcal{B}}
G^\phi(F)q D$ for a finite set $\mathcal{B} \subset G_\mathbf{h}$
then the map $\mathbf{f} \mapsto f_q$ establishes a bijection
between $\mathcal{M}_k(D)$ and $\prod_{q\in
\mathcal{B}}M_k(\Gamma_q)$.

\textbf{Some special congruence subgroups (see
\cite{Shimurabook1,Shimurabook2}):} We now describe some congruences
subgroups that play an important role in this work. We start with
$G:=GU(n,n)$ and introduce some special open compact subgroups $D
\subset G_\mathbf{h}$. We consider two fractional ideals
$\mathfrak{a}$ and $\mathfrak{b}$ of $F$ such that
$\mathfrak{a}\mathfrak{b} \subseteq \mathfrak{g}$, the ring of
integers of $F$, and define using the notation of Shimura \cite[page
11]{Shimurabook2}
\[
D[\mathfrak{a},\mathfrak{b}]:=\left\{x=\left(
                                        \begin{array}{cc}
                                          a_x & b_x \\
                                          c_x & d_x \\
                                        \end{array}
                                      \right)
 \in G_\mathbf{h} | a_x \prec
\mathfrak{r}, b_x \prec \mathfrak{a}\mathfrak{r}, c_x \prec
\mathfrak{b}\mathfrak{r}, d_x \prec \mathfrak{r} \right\},
\]
where we recall $\mathfrak{r}$ is the ring of integers of $K$. One
usually picks either $\mathfrak{a}=\mathfrak{b}=\mathfrak{g}$ or
$\mathfrak{a}=\mathfrak{b}^{-1}=\mathfrak{d}^{-1}$ with
$\mathfrak{d}$ the different ideal of $F$ over $\mathbb{Q}$. As it
is explained in \cite[page 73]{Shimurabook1} we may pick the finite
set $\mathcal{B} \subset G_\mathbf{h}$ consisting of elements of the
form $diag[\hat{r},r]$ for $r \in GL_n(\mathbb{A}_{K,\mathbf{h}})$.
Actually we may even pick the elements $r$ to be of the form $\left(
                                                   \begin{array}{cc}
                                                     t & 0 \\
                                                     0 & 1_{n-1} \\
                                                   \end{array}
                                                 \right)$ with $t
                                                 \in
                                                 \mathbb{A}_{K,\mathbf{h}}$.
We moreover remark the following computation,
\[
\left(
  \begin{array}{cc}
    \hat{r} & 0 \\
    0 & r \\
  \end{array}
\right)\left(
         \begin{array}{cc}
           a & b \\
           c & d \\
         \end{array}
       \right)\left(
                \begin{array}{cc}
                  \hat{r} & 0 \\
                  0 & r \\
                \end{array}
              \right)^{-1}=\left(
                             \begin{array}{cc}
                               \hat{r}a\hat{r}^{-1} & \hat{r}br^{-1} \\
                               rc\hat{r}^{-1} & rdr^{-1} \\
                             \end{array}
                           \right)=\left(
                             \begin{array}{cc}
                               \hat{r}a r^{*} & \hat{r}br^{-1} \\
                               rcr^{*} & rdr^{-1} \\
                             \end{array}
                           \right).
\]
Finally, for an integral ideal $\mathfrak{c}$ of $\mathfrak{g}$, we
introduce the notations
\[
\Gamma_0(\mathfrak{b},\mathfrak{c}):=G_1 \cap
D[\mathfrak{b}^{-1},\mathfrak{b}\mathfrak{c}],\,\,\,\Gamma_1(\mathfrak{b},\mathfrak{c}):=\left\{\gamma
\in \Gamma_0(\mathfrak{c}) | a_\gamma - 1_n \in
\mathfrak{r}\mathfrak{c} \right\}
\]
and we write
$\Gamma_0(\mathfrak{c}):=\Gamma_0(\mathfrak{g},\mathfrak{c})$.

Now we pick an $n$-dimensional hermitian space $(V,\theta)$ over $K$
with $\theta$ positive definite. Let us write $M$ for the
$\mathfrak{g}$-maximal $\mathfrak{r}$-lattice in $V$ and for an
ideal $\mathfrak{c}$ of $\mathfrak{g}$ we now define a congruence
subgroup $D(\mathfrak{c})^\theta \subset G^\theta_{\mathbf{h}}$. We
first define
\[
C:=\{\alpha \in G^\theta_{\mathbf{h}} \,\,|\,\, M\alpha
=M\},\,\,\,\tilde{M}:=\{x \in V\,\,|\,\, \theta(x,M) \subset
\mathfrak{d}^{-1}_{K/F}\}
\]
and then
\[
D^\theta(\mathfrak{c}):=\{\gamma \in C \,\,|\,\, \tilde{M}_v
(\gamma_v-1) \subset \mathfrak{c}_v M_v,\,\,\forall v |
\mathfrak{c}\}
\]

Following Shimura \cite[page 87]{Shimurabook1} we define an element
$\sigma \in GL_n(K)_{\mathbf{h}}$ such that $M'\sigma=M$ where
$M'=\sum_{i=1}^{n}\mathfrak{r}e_i$ for some fixed basis $\{e_i\}$ of
$V$. Then if we write $\theta':=\sigma \theta \sigma$ then for every
finite place of $F$ we have that
\[
\gamma \in D^\theta(\mathfrak{c})_v \Leftrightarrow
{\theta'}^{-1}_v(\sigma \gamma \sigma^{-1} -1)_v \prec
\mathfrak{c}_v \mathfrak{d}_{K/F,v}
\]

\textbf{Families of polarized abelian varieties over $\mathbb{C}$
(see \cite[page 22]{Shimurabook2}):} We consider the following data
$\mathcal{P}:= \{A,\lambda,\imath,\{t_i\}_{i=1}^r\}$ where
\begin{enumerate}
\item $A$ is a complex abelian variety of dimension $d$.
\item $\lambda: A \rightarrow A^{\vee} $ is a polarization of $A$
\item $\imath : K \hookrightarrow End_{\mathbb{Q}}(A)$ a ring
endomorphism, where $K$ is a CM field such that $\imath{K}$ is
stable under the Rosati involution $\alpha \mapsto \alpha'$ of
$End_{\mathbb{Q}}(A)$ determined by $\lambda$.
\item The $t_i$'s are points of $A$ of finite order.
\end{enumerate}

We fix an analytic coordinate system of $A$, that is we fix an
isomorphism $\xi : \mathbb{C}^d/\Lambda \cong A(\mathbb{C})$, where
$\Lambda$ a lattice in $\mathbb{C}^d$. We define a ring injection
$\Psi: K \hookrightarrow \mathbb{C}^d_d$ such that $\imath(a)\xi(u)
= \xi (\Psi(a)u)$ for $a \in K$ and $u \in \mathbb{C}^d$. Then
$\Psi$ gives the structure of a $K$-vector space to the $\mathbb{Q}$
liner span $\mathbb{Q}\Lambda \subset \mathbb{C}^d$, and hance
$\mathbb{Q}\Lambda$ is isomorphic to $K_r^1$ for some $r$ such that
$2d=r[K:\mathbb{Q}]$. We can find an $\mathbb{R}$-linear isomorphism
$q:(K_{\mathbf{a}})_r^1 \rightarrow \mathbb{C}^d$ such that
$q(ax)=\Psi(a)q(x)$ for $a \in K$ and $x \in K_{r}^1$. We define
$L:=q^{-1}(\Lambda)$.

As it is explained in Shimura \cite{ShimuraCM,Shimurabook2}, if we
write $E(\cdot,\cdot)$ for the Riemann form of $A$ determined by the
polarization $\lambda$, then there is an element $\mathcal{T} \in
GL_r(K)$ such that $\mathcal{T}^* = - \mathcal{T}$ and
\[
E(q(x),q(y)) = Tr_{K/\mathbb{Q}}(x \mathcal{T}y^*),\,\,\,x,y \in
K_r^1
\]
Defining $u_i := q^{-1}(t_i)$ we have constructed the PEL-datum
\[
\Omega:= \left\{K,\Psi,L,\mathcal{T},\{u_i\}_{i=1}^s \right\}
\]

Now we fix a CM-type $\Sigma:=\{\tau_{v}\}_{v \in \mathbf{a}}$ of
$K$ and write $m_v$ resp. $n_v$ for the multiplicity of $\tau_v$
resp $\rho \tau_v$ in $\Psi$. Then we have that $m_v + r_v = r$ for
all $v \in \mathbf{a}$ and we can decompose $\mathbb{C}^d$ into a
direct sum $\oplus_{v \in \mathbf{a}}V_v$ so that each $V_v$ is
isomorphic to $\mathbb{C}^r$ and $\Psi(a)$ acts on $V_v$ as
$diag[\bar{a}_v 1_{m_v},a_v 1_{n_v}]$ for each $a\in K$. With this
definitions of $m_v$ and $n_v$ we have that the hermitian form
$i\mathcal{T}_v$ has signature $(m_v,n_v)$ for every $v \in
\mathbf{a}$.

\textbf{Lattices and polarizations (see \cite{HLS2} and \cite[page 8
]{Hsieh1}):} Even though for this paper we need only to consider the
case of unitary groups isomorphic to $U(n,n)$ we present the more
general case $U(m,n)$ since we will need it in \cite{Bouganis4}. We
fix two nonnegative integers $m \geq n \geq 0$. We consider a $K$
vector space $W$ of dimension $m-n$ endowed with a skew-Hermitian
form $\theta$. We fix a basis $\{w_1,\ldots, w_{m-n}\}$ of $W$ such
that $\theta(w_i,w_j)=a_i\delta_{i,j}$. Moreover we assume that $i
\sigma(a_i)> 0$ for all $\sigma \in \Sigma$ and $\sigma_p(a_i)$ is a
$p$-adic unit for all $\sigma_p \in \Sigma_p$. We let
$I^{X}=\oplus_{i=1}^{n}K x_i$ and $I^{Y}=\oplus_{i=1}^{n}K y_i$ and
we consider the skew-Hermitian space $(V,\theta_{m,n})$ defined by
$V:=I^Y \oplus W \oplus I^X$ and
\[
\theta_{m,n}:=\left(
                \begin{array}{ccc}
                   &  & -1_n \\
                   & \theta &  \\
                  1_n &  &  \\
                \end{array}
              \right)
\]

We now pick some particular lattices in the above defined hermitian
spaces. We recall that we write $\mathfrak{r}$ for the ring of
integers of $K$ and $\mathfrak{g}$ for the ring of integers of $F$.
We define $X:=\mathfrak{r} x_1 \oplus \cdots \oplus \mathfrak{r}
x_n$, an $\mathfrak{r}$ lattice in $I^X$ and
$Y:=\mathfrak{d}^{-1}_{K/F} y_1 \oplus \cdots \oplus
\mathfrak{d}^{-1}_{K/F} y_n$, an $\mathfrak{r}$-lattice in $I^Y$. We
choose a $\mathfrak{r}$-lattice $L$ in $W$ that is
$\mathfrak{g}$-maximal with respect to the Hermitian form $\theta$
(see \cite[page 26]{Shimurabook1} for the definition of maximal
lattices). Then we define the $\mathfrak{r}$-lattice $M$ in $V$ as
$M:=Y \oplus L \oplus X$.

We now let $M_p:=M \otimes_{\mathfrak{r}} \mathfrak{r}_p$. We
consider the following sublattices of $M_p$,
\[
M^{-1}:= Y_{\Sigma_p} \oplus L_{\Sigma_p} \oplus Y_{\Sigma_p^\rho}
\]
and
\[
M^{0}:= X_{\Sigma_p^\rho} \oplus L_{\Sigma_p^\rho} \oplus
X_{\Sigma_p},
\]
where for a set $S$ of places of $K$ and an $\mathfrak{r}$ ideal $L$
we write $L_S:=L \otimes_{\mathfrak{r}} \prod_{v \in S}
\mathfrak{r}_v$.

The sublattices $(M^0,M^{-1})$ have the following properties: (i)
they are maximal isotropic submodules of $M_p$ and they are dual to
each other with respect to the alternating form
$(\cdot,\cdot)_{m,n}$ defined by as
$(v,v')_{m,n}:=Tr_{K/\mathbb{Q}}(\theta_{m,n}(v,v'))$ and (ii) we
have that $rank M^{-1}_{\Sigma_p}= rank M^{0}_{\Sigma_p^\rho}=m$ and
$rank M^{-1}_{\Sigma_p^\rho}= rank M^{0}_{\Sigma_p}=n$. Such a pair
is usually called a polarization of $M_p$. As it is explained in
\cite[page 9]{HLS2} the existence of such a polarization is
equivalent to the ordinary condition that we have imposed on $p$.

\textbf{Shimura varieties for unitary groups (see \cite[pages 8-9
]{Hsieh1} and \cite{Hsieh2}):} Let $G:=GU(\theta_{m,n})/F$ and $M$
be as above our fixed lattice. For a finite place $v \in \mathbf{h}$
we set
\[
D_v:=\left\{g \in G(F_v) : M_v g = M_v \right\}
\]
and $D := \prod_{v \in \mathbf{h}}D_v$. Our ordinary assumption
allow us to identify for every $v$ above $p$
\[
G(F_p) \stackrel{\sim}{\rightarrow} \prod_{v \in \Sigma_p}
GL_{m+n}(F_v) \times F_v^\times
\]
and the maximality assumption of the lattice $M$ gives
\[
G(\mathfrak{g}_p)\stackrel{\sim}{\rightarrow} \prod_{v \in \Sigma_p}
GL_{m+n}(\mathfrak{g}_v) \times \mathfrak{g}_v^\times
\]
That is, for every $v | p$ we have
\[
D_v = GL(M_{\Sigma_{p}}) \times \mathfrak{g}_v^\times
\stackrel{\sim}{\rightarrow} GL_{m+n}(\mathfrak{g}_v) \times
\mathfrak{g}_v^\times.
\]

We fix an integer $N$ relative prime to $p$ and define an open
compact subgroup $K$ with the property
\[
K(N) \subseteq \left\{g \in D : M(g-1) \subseteq N M\right\}.
\]

We define for $r \geq 0$ the groups
\[
K^r(N):=\left\{g \in K(N) : g_v \equiv \left(
                                    \begin{array}{cc}
                                      1_m & * \\
                                      0 & 1_n \\
                                    \end{array}
                                  \right)\mod{p^r},\,\,\forall v | p
\right\}.
\]

Let us write $E$ for the reflex field of our fixed type $(K,\Phi)$.
We write $O_E$ for its ring of integers and we consider the ring
$R:=O_E \otimes \mathbb{Z}_{(p)}$. Let $\mathcal{S}$ denote a finite
set of rational primes. We write $U \subseteq D$ for an open-compact
subgroup of $G(\mathbb{A}_\mathbf{h})$. Let $S$ be a connected,
locally noetherian $R$-scheme and $\bar{s}$ a geometric point of
$S$. An $S$-quadruple
$(A,\bar{\lambda},\imath,\bar{\eta}^{(\mathcal{S})})$ of level $U$
consists of the following data:
\begin{enumerate}
\item $A$ is an abelian scheme of dimension $(m+n)d$ over $S$, where
$d=[K:\mathbb{Q}]$,

\item $\bar{\lambda}$ is a class of polarizations $O_{(\mathcal{S}),+}\lambda$, where $\lambda$ is a prime to $\mathcal{S}$ polarization of $A$ over
$S$,

\item $\imath : \mathfrak{r} \hookrightarrow End_S(A) \otimes_\mathbb{Z}
\mathbb{Z}_{(\mathcal{S})}$ compatible with the Rosati involution
induced by $\lambda$.

\item $\bar{\eta}^{(U)}=U\eta^{(\mathcal{S})}$, with
$\eta^{(\mathcal{S})}: M \otimes \hat{\mathbb{Z}}^{(\mathcal{S})}
\stackrel{\sim}{\rightarrow} T^{(\mathcal{S})}(A_{\bar{s}})$ an
$\mathfrak{r}$-linear $\pi_1(S,\bar{s})$-invariant isomorphism and
$\mathcal{T}^{(\mathcal{S})}(A_{\bar{s}}):=\varprojlim_{(N,\mathcal{S})=1}
A[N](k(s))$

\item We write $V^{(\mathcal{S})}(A_s):=\mathcal{T}(A_s)
\otimes_{\mathbb{Z}} \mathbb{A}^{(\mathcal{S})}$. Then the numerical
structure induces an isomorphism $\eta^{(\mathcal{S})}:M \otimes
\mathbb{A}^{(\mathcal{S})} \stackrel{\sim}{\rightarrow}
V^{(\mathcal{S})}(A_{\bar{s}})$. We obtain a skew-hermitian form
$e^\eta$ on $V^{(\mathcal{S})}(A_{\bar{s}})$ by
$e^\eta(x,x'):=\theta_{m,n}(\eta^{-1}(x),\eta^{-1}(x'))$. Then, if
we write $e^\lambda$ for skew-hermitian on
$V^{(\mathcal{S})}(A_{\bar{s}})$ induced by the polarization, we
require that $e^\lambda=u e^\eta$ for $u \in
\mathbb{A}_{\mathbf{h}}^{(\mathcal{S})}$.

\item We have that
\[
det(X-\imath(b)|Lie(A))=\prod_{\sigma \in \Sigma}(X-(\sigma \circ
c)(b))^m (X-\sigma(b))^n \in O_S[X], \,\,\forall b \in\mathfrak{r}
\]
\end{enumerate}

We will consider mainly two situations for $\mathcal{S}$. Namely,
the case where $\mathcal{S}=\emptyset$ and $\mathcal{S}=\{p\}$. In
the first case, it follows by the theory of Shimura and Deligne that
the functor $\mathcal{F}_U$ from the category of schemes over $E$ to
the category of sets defined as
\[
\mathcal{F}_U
(S)=\left\{\underline{A}=(A,\bar{\lambda},\imath,\bar{\eta})/S
\right\}/\cong
\]
is representable by a quasi-projective scheme $Sh_G(U)$ defined over
$E$. In the other case, it is know by the theory of Kottwitz that if
we pick $U=K(N)$ neat and such that $U_p=D_p$, then the functor
$\mathcal{F}^{(p)}_U$ from the category of schemes over $R$ to sets
\[
\mathcal{F}^{(p)}_{K}
(S)=\left\{\underline{A}=(A,\bar{\lambda},\imath,\bar{\eta}^{(p)})/S
\right\}/\cong
\]
is represented by a quasi-projective scheme $Sh^{(p)}_G(K)/R$.

\textbf{Algebraic hermitian modular forms (see \cite[page
19]{Eischen}\footnote{The paging is from the arXiv version of the
paper.}):} Let $(V,\phi)$ be a hermitian form with $\phi_{\sigma}$
of signature $(m_{\sigma},n_{\sigma})$ for $\sigma \in \Sigma$. We
fix an $\mathfrak{r}$-algebra $R_0$ and we consider an algebraic
representation $(V,\rho)$ of $\prod_{\sigma \in
\Sigma}GL_{m_{\sigma}} \times GL_{n_{\sigma}}$ defined over $R_0$,
that is we have for every $R_0$-algebra $R$ a homomorphism
\[
\rho_{R}: \prod_{\sigma \in \Sigma} GL_{m_{\sigma}}(R) \times
GL_{n_{\sigma}}(R) \rightarrow GL(V_R),\,\,V_R:=V \otimes_{R_0} R
\]
that commutes with extensions of scalars $R \rightarrow R'$ of
$R_0$-algebras.

For each datum $\underline{A}/R=\{A,\lambda,\imath,\alpha\}/R$
defined over $R$, for an $R_0$-algebra $R$ we define modules
\[
\mathcal{E}_{\underline{A}/R}^{+}=\prod_{\sigma \in
\Sigma}Isom_R(R^{m_{\sigma}},e_{\sigma}\underline{\omega}_{A/R}),
\]
\[
\mathcal{E}_{\underline{A}/R}^{-}=\prod_{\sigma \in
\Sigma}Isom_R(R^{n_{\sigma}},e_{\rho\sigma}\underline{\omega}_{A/R}),
\]
and
\[
\mathcal{E}_{\underline{A}/R}=\mathcal{E}_{\underline{A}/R}^{+}
\oplus \mathcal{E}_{\underline{A}/R}^{-}.
\]
Here for $\sigma \in \Sigma \coprod \Sigma^{\rho}$ we write
$e_\sigma \in K \otimes_{\mathbb{Q}} \bar{\mathbb{Q}}$ for the
corresponding orthogonal idempotent related to the decomposition $K
\otimes_{\mathbb{Q}} \bar{\mathbb{Q}}´= \otimes_{\sigma \in \Sigma
\coprod \Sigma^{\rho}}\bar{\mathbb{Q}}$. We note that to give an
element $\varepsilon \in \mathcal{E}_{\underline{A}/R}^{\pm}$ is
equivalent to fixing a basis for
$\underline{\omega}_{A/R}^+:=\prod_{\sigma \in
\Sigma}e_{\sigma}\underline{\omega}_{A/R}$ and
$\underline{\omega}_{A/R}^-:=\prod_{\sigma \in
\Sigma}e_{\rho\sigma}\underline{\omega}_{A/R}$. In particular we
have that the group $\prod_{\sigma \in \Sigma}GL_{m_{\sigma}}(R)$
(resp $\prod_{\sigma \in \Sigma}GL_{n_{\sigma}}(R))$ acts on
$\mathcal{E}_{\underline{A}/R}^{+}$ (resp.
$\mathcal{E}_{\underline{A}/R}^-$) by
\[
\alpha \cdot \varepsilon(v):= \varepsilon(\alpha^t v),\,\,\, \alpha
\in \prod_{\sigma \in \Sigma}GL_{m_{\sigma}}(R), \varepsilon \in
\mathcal{E}_{\underline{A}/R}^{+}, v \in \prod_{\sigma \in
\Sigma}R^{m_\sigma}
\]
and hence the group $\prod_{\sigma \in \Sigma} GL_{m_{\sigma}}(R)
\times GL_{n_{\sigma}}(R)$ on $\mathcal{E}_{\underline{A}/R}$.

\begin{defn} A hermitian modular form of weight $\rho$ and level $\alpha$, defined over
$R_0$ is a function $f$ on the set of pairs
$(\underline{A},\varepsilon)/R$ with values in $V/R$ such that the
following hold:
\begin{enumerate}
\item The element $f(\underline{A},\varepsilon)$ depends only on the
$R$-isomorphism class of $(\underline{A},\varepsilon)$.
\item The function $f$ is compatible with base change $R \rightarrow
R'$ of $R_0$-algebras, that is
\[
f(\underline{A} \times_{R} R',\varepsilon \otimes_{R} R') =
f(\underline{A},\varepsilon) \otimes _{R} 1 \in V_{R'}
\]
\item For each $(\underline{A},\varepsilon)$ over $R$ and $\alpha
\in \prod_{\sigma \in \Sigma} GL_{m_{\sigma}}(R) \times
GL_{n_{\sigma}}(R)$ we have
\[
f(\underline{A},\alpha
\varepsilon)=\rho(\alpha^t)^{-1}f(\underline{A},\varepsilon)
\]
\end{enumerate}
\end{defn}

\textbf{$p$-adic hermitian modular forms (see
\cite{HLS1,HLS2,Hsieh1,Hsieh2}):} We consider the functor
$\mathcal{F}^{(p)}_{K^n}$ from $R$ schemes to sets given by
\[
\mathcal{F}^{(p)}_{K^n} (S)=\left\{(\underline{A},j_n)
=(A,\bar{\lambda},\imath,\bar{\eta}^{(p)},j_n)/S \right\}/\cong,
\]
where $j_n$ an $\mathfrak{r}$-linear embedding
\[
j_n:M^0 \otimes \mu_{p^n} \hookrightarrow A[p^n].
\]
This functor is representable by a scheme $Ig_G(K^n)/R$ (this is
what in \cite[page 26]{HLS2} is denoted by $Ig_{2,n}$). We now
consider the strict ideal class group of $F$, that is $Cl_F^+(K)=F_+
\setminus \mathbb{A}_{F,f}/\nu(K)$, and pick representatives of it
in $\mathbb{A}^{(p)}_{F,f}$. From each such representative $c$ we
consider the functor $\mathcal{F}^{(p)}_{K^n,c}$ from $R$-schemes to
sets
\[
\mathcal{F}^{(p)}_{K^n,c} (S)=\left\{(\underline{A},j_n)
=(A,\lambda,\imath,\bar{\eta}^{(p)},j_n)/S :
\lambda\,\,is\,\,a\,\,c-polarization \right\}/\cong,
\]
This functor is representable by a scheme $Ig_{G_1}(K^n,c)/R$ and we
have that
\[
\coprod_{c} Ig_{G_1}(K^n,c) \cong Ig_G(K^n)
\]

We now take the ring $R$ above to be a $p$-adic ring, that is
$R=\varprojlim_{k}R/p^k R$. We write $R_k:=R/p^kR$. Now we fix a
toroidal compactification $\overline{S_G(K)}/R_k$ of $S_G(K)/R_k$
and write $T_{0,k}:=\overline{S_G(K)}[1/E]/R_k$, for the ordinary
locus of it. Here $E$ is a lift of the Hasse invariant from $R_1$ to
$R_k$ (see \cite[page 30]{HLS2}). For a positive integer $\ell$ we
set $T_{\ell,k}:=Ig_G(K^\ell)/R_k$. There exist finite \'{e}tale
maps $\pi_{\ell',\ell,k}:T_{\ell',k} \rightarrow T_{\ell,k}$ and we
define $T_{\infty,k}:=\varprojlim_\ell
T_{\ell,k}=Ig_G(K^\infty)/R_k$. Then $T_{\infty,k}$ is galois over
$T_{0,k}$ with Galois group isomorphic to
$Aut_{\mathfrak{r}_p}(M^0)$. For $\ell,k \in \mathbb{N}$ we define
the spaces
\[
V_{\ell,k}:=H^0(T_{\ell,k},\mathcal{O}_{T_{\ell,k}}),
\]
and then $V_{\infty,k}:=\varinjlim_\ell V_{\ell,k}$ and
$V:=\varprojlim_k V_{\infty,k}$. We call $V_p(G,K):=V^N$, the space
of $p$-adic modular forms of level $K$. Here $N \leq
GL_{m+n}(\mathfrak{r}_p)$ is the upper-triangular unipotent radical
of $GL_{m+n}(\mathfrak{r}_p)$.

\textbf{Algebraic and $p$-adic hermitian modular forms (see
\cite{HLS1,HLS2,Hsieh1,Hsieh2}):} Now we assume that we are given a
$p^\infty$ arithmetic structure of an abelian variety $A$ with CM by
$\mathfrak{r}$ of type $(m,n)$ defined over a $p$-adic ring $R$.
That is we have compatible $\mathfrak{r}$-linear embeddings
\[
j_n:M^0 \otimes \mu_{p^n} \hookrightarrow A[p^n],
\]
for all $n \geq 0$. That is, we assume an embedding
\[
j_\infty:M^0 \otimes \mu_{p^\infty} \hookrightarrow A[p^\infty].
\]
In turn we get an isomorphism
\[
j:M^0 \otimes \hat{\mathbb{G}}_m \stackrel{\sim}{\rightarrow}
\hat{A}.
\]
Identifying $Lie(A)=Lie(\hat{A})$ we obtain an isomorphism
\[
j:M^0 \otimes_R R \stackrel{\sim}{\rightarrow} Lie(A),
\]
which induces also the isomorphisms
\[
j_{+}:M_{\Sigma_p}^0 \otimes_R R \stackrel{\sim}{\rightarrow}
e_{\Sigma_p}Lie(A),\,\,\,j_{-}:M_{\Sigma_p^c}^0 \otimes_R R
\stackrel{\sim}{\rightarrow} e_{\Sigma_p^c}Lie(A).
\]
As $\underline{\omega}_{A/R}=Hom(Lie(A),R)$ we obtain isomorphisms
\[
\omega(j)_+:M_{\Sigma_p}^0 \otimes_R R \stackrel{\sim}{\rightarrow}
e_{\Sigma_p}\underline{\omega}_{A/R},\,\,\,\omega(j)_-:M_{\Sigma_p}^0
\otimes_R R \stackrel{\sim}{\rightarrow}
e_{\Sigma^c_p}\underline{\omega}_{A/R}
\]
and then
\[
\omega(j):=\omega(j)_+ \oplus \omega(j)_- : M^0 \otimes_R R
\stackrel{\sim}{\rightarrow} \underline{\omega}_{A/R}.
\]
In particular we obtained an element $\omega(j) \in
\mathcal{E}_{\underline{A}/R}$. This construction allows us to
consider every algebraic hermitian modular form also as a $p$-adic
hermitian modular form. Indeed for $f$ a hermitian form we define
its $p$-adic avatar $\hat{f}$ by
\[
\hat{f}(\underline{A},j) := f(\underline{A},\omega(j)).
\]

\textbf{Mumford Objects\footnote{The author has been informed by
Ellen Eischen that this terminology is not standard and the term has
been used for first time in this context in her Ph.D thesis.} for
$GU(n,n)/F$ and $q$-expansions (see \cite[pages 34-38]{Eischen}):}
The familiar $q$-expansion with respect some given cusp of an
elliptic modular forms has an algebraic interpretation as the
evaluation of the modular forms on the so-called Tate curve that
corresponds to the selected cusp. Our next goal is to introduce the
analogues of the Tate curve for the unitary groups $GU(n,n)$. We
start by considering a Hermitian space $(V,\phi)$ over the CM field
$K$ and we assume that
\[
\phi=\left(
       \begin{array}{cc}
         0 & -1_n \\
         1_n & 0 \\
       \end{array}
     \right)
\]
That is, $G^\phi=GU(n,n)$. We now fix maximal isotropic spaces $W$
and $W'$ with $W \cong W' \cong K^n$ of $\phi$ and we have a
decomposition
\[
V = W \oplus W'
\]
We consider the standard $\mathfrak{g}$-maximal lattice of $\phi$ in
$V$ defined as $\Lambda:=\sum_{i=1}^n \mathfrak{r} e_i \oplus
\sum_{i=1}^n \mathfrak{d}^{-1}_{K/F} f_{i}$ for the standard basis
of $(V,\phi)$ i.e. $\phi(e_i,e_j)=\phi(f_i,f_j)=0$ and
$\phi(e_i,f_j)=-\delta_{i,j}$ and $\mathfrak{d}_{K/F}$ the relative
different of $K$ over $F$. We now define the lattices
\[
L:=W \cap \Lambda,\,\,and\,\,L':=W' \cap \Lambda.
\]
Note the choice of the pair $(L,L')$ is equivalent to the choice of
a polarization as explained above. We write $P$ for the stabilizer
of $W'$ in $G^\phi$ and $N_P$ for its unipotent radical. Then $N_P$
consists of matrices of the form
\[
\left(
                                                               \begin{array}{cc}
                                                                 1_n & B \\
                                                                 0 & 1_n \\
                                                               \end{array}
                                                             \right)
                                                             \]
where $B \in S_n$. For a congruence subgroup $\Gamma$ of $G^\phi$ we
define $H:=\Gamma \cap N_P$. Then
\[
H= \left(
                                                               \begin{array}{cc}
                                                                 1_n & M \\
                                                                 0 & 1_n \\
                                                               \end{array}
                                                             \right)
\]
where $M$ is a lattice in $S_n$. Writing $M^{\vee}$ for the dual
lattice of $M$, i.e.
\[
M^{\vee}:=\left\{x \in S_n| Tr_{F/\mathbb{Q}}(tr(xM)) \subset
\mathbb{Z} \right\}
\]
we define $H^{\vee}:=\left(
                       \begin{array}{cc}
                         1_n & M^{\vee} \\
                         0 & 1_n \\
                       \end{array}
                     \right)$.

For a ring $R$ we define the ring of formal power series
\[
R((q,H^{\vee}_{\geq 0})):=\left\{\sum_{h \in H^{\vee}}a_hq^h | a_h
\in R,\,\,a_h =0,\,\,if\,\,h << 0 \right\}
\]

Over the ring $R((q,H^{\vee}_{\geq 0}))$ we now define a
$\mathbb{Z}$-liner morphism $q:L \rightarrow (L') \otimes
\mathbb{G}_m$ as follows.
\[
L \rightarrow Hom_{\mathbb{Z}}(H,L') \cong H^{\vee} \otimes L'
\rightarrow L' \otimes \mathbb{G}_m
\]
where the first map is given by $\ell \mapsto (h \mapsto h(\ell))$
and the last one is given by $h' \mapsto q^{h'} \in
\mathbb{G}_{m}(R((q,H^{\vee}_{\geq 0})))$. The Mumford object
corresponding to the cusp $(L,L')$ is given by the algebraification
of the rigid analytic quotient
\[
Mum_{(L,L')}(q):= (L' \otimes \mathbb{G}_m)/ q(L)
\]
The PEL structure of $Mum_{L,L'}(q)$ is given as follows. We have a
canonical endomorphism $\imath_{can} : \mathfrak{r} \hookrightarrow
End_{R((q,H^{\vee}_{\geq 0}))}(Mum_{L,L'}(q))$ given by $\alpha
\mapsto (\ell \mapsto \alpha \cdot \ell)$ for $\alpha \in
\mathfrak{r}$ and $\ell \in L$.

We now define the canonical polarization $\lambda_{can}$ of
$Mum_{(L,L')}(q)$ as follows. For dual abelian variety
\[
Mum_{(L,L')}(q)^{\vee}=Mum_{({L'}^\vee,{L}^\vee)}=(L^\vee \otimes
\mathbb{G}_m)/q({L'}^\vee)
\]
there exist an isogeny
\[
 \lambda_{can} : Mum_{(L,L')}(q) \rightarrow Mum_{(L,L')}(q)^{\vee}
\]
induced by the isomorphism $L^\vee \otimes_{\mathfrak{r}} K \cong L'
\otimes_{\mathfrak{r}} K$. The level structure is induced by the
embedding $\alpha_{N} : L' \otimes \mu_{N} \hookrightarrow L'
\otimes \mathbb{G}_m$. Finally we have a canonical differential
$\omega_{can}$. This is defined by dualizing the isomorphism
\[
Lie(Mum_{(L,L')})(q) \cong Lie(L' \otimes \mathbb{G}_m = L' \otimes
R((q,H^{\vee}_{\geq 0}))
\]
That is we obtain an isomorphism
\[
\omega_{can}:R((q,H^{\vee}_{\geq 0}))^n \cong {L'}^\vee \otimes
R((q,H^{\vee}_{\geq 0})) \cong
\omega_{Mum_{(L,L')}(q)/R((q,H^{\vee}_{\geq 0}))}
\]

\textbf{Cusps:} Now we study the (0-genus) cusps of the group
$GU(n,n)$ with respect particular congruences subgroups. We will see
that to each of these we can associate an arithmetic data
$(Mum_{L,L'}(q),\lambda_{can},\imath_{can},\omega_{can})$. As above
we write $P$ for the standard parabolic of $G:=GU(n,n)$ and $N_P$
for its unipotent radical. Then a Levi part of $P$ can be identified
with $GL_n(K)$ by the embedding $d \mapsto diag[\hat{d},d]$ for $d
\in GL_n(K)$. Then the set of cusps of $G$ with respect the open
locally compact subgroup $K_0$ is given by
\[
C_0(K):=GL_n(K) \times N_P(F) \setminus
G(\mathbb{A}_{F,\mathbf{h}})/K_0.
\]
It is well-known, see \cite[lemma 9.8]{Shimurabook1} that we can
choose a decomposition
\[
G(\mathbb{A}_{F,\mathbf{h}})=\coprod_{j=1}^m G(F)g_j K_0,
\]
and if we pick $g \in C_0(K)$ and write it as $g=\gamma g_i k$ with
respect to the above decomposition then the Mumford object
associated to the cusp $g$ is given by $Mum_{L_g,L'_g}(q)$ where
$L_g:=L g_i \cap V$ and $L'_g:=L g_i \cap V$.

\textbf{The complex analytic point of view (see \cite[page
37]{Eischen}):} Now we would like to describe the complex points of
the Mumford object $Mum_{L,L'}(q)$. Recall that the associated to
$GU(n,n)$ symmetric space is
$\mathbb{H}_F:=\mathbb{H}_{(n,n)}^{[F:\mathbb{Q}]}$ where
\[
\mathbb{H}_{(n,n)}:=\left\{z \in M_n(\mathbb{C}) | i(z^* - z)
> 0 \right\}.
\]
We note that if we write $S$ for the set of hermitian matrices over
$K$ then $\mathbb{H}_{(n,n)}= S + i S_{+}$ and hence also
$\mathbb{H}_F=S_{\mathbf{a}} + \mathbf{i} S_{\mathbf{a}+}$. Given a
$\tau \in \mathbb{H}_F$ we consider the lattice $L_{\tau} \in
\mathbb{C}^{2n}$ generated by $L' \otimes_{\mathfrak{r}} 1 \in W'
\otimes_{K} \mathbb{C}$ and $\tau L \in W' \otimes_K \mathbb{C}$.
Then using the exponential map $\mathbf{exp}$ we obtain
$\mathbf{exp}(L_{\tau}) \subset W' \otimes_K \mathbb{C}^\times$.
Then we have that if fix the indeterminate parameter $q$ as
$q=exp_{\mathbf{a}}(2\pi i tr(\tau))$ we get
\[
Mum_{L,L'}(q)(\mathbb{C})= W' \otimes_K \mathbb{C}^\times /
\mathbf{exp}(L_{\tau}).
\]

\textbf{Analytic and algebraic $q$-expansions (see
\cite{Shimurabook1}):} Let now $f \in M_k(\Gamma)$ be a hermitian
modular form for a congruences group $\Gamma$ of $G$. As it is
explained in \cite[page 33]{Shimurabook2} we can always find a
$\mathbb{Z}$ lattice $M$ in $S$ such that $\left(
                                                                 \begin{array}{cc}
                                                                   1_n & \sigma \\
                                                                   0 & 1_n \\
                                                                 \end{array}
                                                               \right)
                                                               \in
                                                               \Gamma$,
for all $\sigma \in M$. Then we have $f(z + \sigma)=f(z)$ for all
$\sigma \in M$ and hence the hermitian modular form $f$ has a
Fourier expansion

\begin{equation}\label{analytic q-expansion}
 f(z)=\sum_{h \in
L}c(h)e_{\mathbf{a}}^n(hz),
\end{equation}
where $L:=\left\{h \in S\,\, |\,\, Tr_{F/\mathbb{Q}}(tr(hM)) \subset
\mathbb{Z} \right\}$. In particular, by Shimura \cite[page
147]{Shimurabook1}, for $\Gamma=\Gamma_0(\mathfrak{b},\mathfrak{c})$
we have that $L=\mathfrak{d}^{-1}\mathfrak{b}T$, where $T:=\left\{x
\in S \,\,|\,\, tr(S(\mathfrak{r})x) \subset \mathfrak{g} \right\}.$
Actually if we consider for an element $g \in
GL_n(\mathbb{A}_{K,\mathbf{h}})$ groups of the form
\[
\Gamma_g:=G_1 \cap \left(
                      \begin{array}{cc}
                        \hat{g} & 0 \\
                        0 & g\\
                      \end{array}
                    \right)D[\mathfrak{b},\mathfrak{c}]\left(
                      \begin{array}{cc}
                        \hat{g} & 0 \\
                        0 & g\\
                      \end{array}
                    \right)^{-1}.
                    \]
then we have that the lattice $L$ above is now equal to
$\mathfrak{d}^{-1}\mathfrak{b}gTg^{*}$.

The expression in \ref{analytic q-expansion} has an algebraic
interpretation through the use of the Mumford objects introduced
above. To this end, we consider now an hermitian modular form $f \in
M_k(\Gamma_0(N)$ defined over the ring $R$. Then we may evaluate $f$
at the Mumford data
$(\underline{Mum_{L,L'}(q)},\omega_{can})=(Mum_{L,L'}(q)),\lambda_{can},i_{can},\alpha_N,\omega_{can})$
defined over the ring $R((q,H^{\vee}_{\geq 0}))$ to obtain
\begin{equation}\label{algebraic q-expanion}
f\left(\underline{Mum_{L,L'}(q)},\omega_{can} \right)=\sum_{h \in
H^{\vee}}c(h)q^h.
\end{equation}
When $R=\mathbb{C}$ we may pick $q:=e_{\mathbf{a}}^n(z)$ and then
the expression above is the same as the one in (\ref{analytic
q-expansion}).

Finally we close this section by recalling also the $q$-expansion
for unitary automorphic forms. So we let $\phi \in
\mathcal{M}_k(D)$, with
$D=D[\mathfrak{b}^{-1},\mathfrak{b}\mathfrak{c}]$. Then the
following proposition is taken from \cite[page 148]{Shimurabook1}.
\begin{prop}\label{adelic q-expansion} For every $\sigma \in
S_{\mathbb{A}}$ and $q \in GL_n(\mathbb{A}_K)$ we have
\[
\phi \left(\left(
             \begin{array}{cc}
               q & \sigma\hat{q} \\
               0 & \hat{q} \\
             \end{array}
           \right)
\right)=\sum_{h \in S} c(h,q) e^n_{\mathbb{A}}(tr(h\sigma)),
\]
with the following properties
\begin{enumerate}
\item $det(q)^{-\rho k}c(h,q)$ depends only on $\phi$, $h$,
$q_\mathbf{h}$ and $(qq^*)_{\mathbf{a}}$,
\item $c(h,q)\neq 0$ if and only if $(q^*hq)_v \in
\mathfrak{b}_v\mathfrak{d}^{-1}_v T_v$ for all finite places $v$,
\item $c(h,q)$ may be written as
\[
c(h,q)=det(q)^{k \rho}c_0(h,q)e_{\mathbf{a}}^n(i \cdot tr(q^*hq)),
\]
where $c_0(h,q)$ depends only on $\phi$, $h$ and $q_{\mathbf{h}}$.
\end{enumerate}
\end{prop}

We now briefly explain the relation of this automorphic
$q$-expansion with the complex analytic one in (\ref{analytic
q-expansion}). We start with the remark that we may write $z \in
\mathbb{H}= S_{\mathbf{a}} + \mathbf{i}S_{\mathbf{a}+}$ as $z=x +
\mathbf{i}y$. We then define $q \in GL_n(\mathbb{A}_K)$ by $q_h=1_n$
and $q_\mathbf{a}=y^{1/2}$ so that $q_\mathbf{a}q_{\mathbf{a}}^*=y$.
Further we pick $\sigma \in S_{\mathbb{A}}$ by $\sigma_\mathbf{h}=1$
and $\sigma_{\mathbf{a}}=x$. With these choices we have that
\[
\phi\left(\left(
            \begin{array}{cc}
              q & \sigma \hat{q} \\
              0 & \hat{q} \\
            \end{array}
          \right)
\right)=det(q)^{\rho k} f_1(z).
\]

\section{The Eisenstein Measure of Harris, Li and Skinner.}

In this section our goal is to present the construction of an
Eisenstein measure due to Harris, Li and Skinner as in
\cite{HLS1,HLS2}. As we will see the key ingredients are (i) the
computation of the Fourier coefficients of Siegel type Eisenstein
series by Shimura (see for example
\cite{Shimurabook1,Shimurabook2}), (ii) the definition of particular
sections at places above $p$ as done by Harris, Li and Skinner (loc.
cit.) and (iii) the definition of sections for other ``bad'' primes
(not including those above $p$) as done by Ming-Lun Hsieh
\cite{Hsieh1,Hsieh2}. Finally we also mention here the recent work
of Eischen \cite{Eischen2} generalizing various aspects of the work
of Harris, Li and Skinner and working the relation of their
Eisenstein measure with the theory of the $p$-adic differential
operators developed in \cite{Eischen}.

\subsection*{\textbf{Siegel Eisenstein series for $U(n,n)$}}

We follow Shimura \cite{Shimurabook1}. We let $(W,\psi)$ be a
Hermitian form that decomposes as $(W,\psi)=(V,\phi) \oplus
(H_m,\eta_m)$. We write $n:=dim(V)+m$ and we define
\[
(X,\omega):= (W,\psi) \oplus (V,-\phi)
\]
a hermitian space of dimension $2n$ over $K$. We consider the
decomposition of $(H_m,\eta_m)$ to its maximal isotropic spaces $I$
and $I'$ i.e. $H_m = I \oplus I'$ and hence we can write
\[
X = W \oplus V = I' \oplus V \oplus I \oplus V
\]
We pick a basis of $W$ so that
\[
\omega = \left(
           \begin{array}{cc}
             \psi & 0 \\
             0 & -\phi \\
           \end{array}
         \right)
\]
and hence obtain an embedding $G^\psi \times G^\phi \hookrightarrow
G^\omega$ by $(\beta,\gamma) \mapsto diag[\beta,\gamma]$. If we
write the elements of $X$ in the form $(i',v,i,u)$ with $i' \in I',
i \in I$ and $u,v \in V$ with respect to the above decomposition of
$X$ we put
\[
U := \left\{(0,v,i,v)| v \in V, i \in I \right\},\,\,\,P_U^\omega:=
\{\gamma \in G^\omega| U\gamma = U\}
\]
Then $U$ is totally $\omega$-isotropic and $P^\omega$ is a parabolic
subgroup of $G^\omega$. From \cite[page 7]{Shimurabook1} we know
that if $U'$ is another totally $\omega$-isotropic subspace of $X$
with $dim(U)=dim(U')$ then there exists $\beta \in G^\omega$ such
that $P_U^\omega=\beta P_{U'}^\omega \beta^{-1}$.

As it is explained in Shimura \cite[page 176]{Shimurabook1}, we have
that $(X,\omega) \cong (H_n,\eta_n)$. In the group $G:=G^{\eta_n}$
we write $P$ for the standard Siegel parabolic given by elements $x=\left(
                                                                      \begin{array}{cc}
                                                                        a_x & b_x \\
                                                                        b_x & d_x \\
                                                                      \end{array}
                                                                    \right)
\in G$ with $c_x=0_n$. Now we are ready to define the classical
Siegel-type Eisenstein series attached to a Hecke character $\chi$.
That is we consider a Hecke character $\chi$ of infinite type
$\{k_v\}_{v \in \mathbf{a}}$ with respect to the fixed CM type
$(K,\Sigma)$ and we write $\mathfrak{c} \subseteq \mathfrak{g}$ for
the conductor of $\chi$. Moreover for our applications we are going
to assume that $\mathfrak{c} \neq \mathfrak{r}$. We now note that
the Siegel parabolic $P$ in $G$ is given by
\[
P \cong \left\{\left(
                    \begin{array}{cc}
                      A & B \\
                      0 & \hat{A} \\
                    \end{array}
                  \right)
: A \in GL_n(K),\,\,B \in S_n \right\},\] where
$\hat{A}=(\transpose{\bar{A}})^{-1}$ and $S_n$ the space of $n
\times n$ Hermitian matrices. Let $v$ be a place of $F$. We define
the modulus character $\delta_{P,v}: P(F_v) \rightarrow
\R_+^{\times}$ as
\[
\delta_{P,v}(g)=|N_{K/F} \circ det(A(g))|_v^{-1},\,\,\, g \in P(F_v)
\]
We write $\delta_{P,\mathbb{A}_F}:=\prod_{v}\delta_{P,v}$ for the
adelic modulus character and for $s \in \mathbb{C}$ and $\chi$ our
Hecke character of $K$ we define
\[
\delta_{P,\mathbb{A}}(g,\chi,s):=\chi(det(A(g))\delta_{P,\mathbb{A}}(g,s)
\]
where
\[
\delta_{P,\mathbb{A}}(g,s):=|N_{K/F} \circ
det(A(g))|_\mathbb{A}^{\frac{n}{2}+s},\,\,\, g \in P(\mathbb{A})
\]
We define $I_{\chi}$ as the parabolic induction
$Ind_{P(\mathbb{A}_F)}^{G(\mathbb{A}_F)}\delta_{P,\mathbb{A}_F}(g,\chi,s)$,
i.e.
\[
I_\chi:= \{\phi:G(\mathbb{A}_F) \rightarrow \mathbb{C}: \phi(pg) =
\delta_{P,\mathbb{A}_F}(p,\chi,s)\phi(g), p \in P(\mathbb{A}_F),
g\in G(\mathbb{A}_F)\}\] where $\phi$ is $D_{\mathbf{a}} \cong
U(n)_{\mathbf{a}} \times U(n)_{\mathbf{a}}$ finite, for some maximal
open compact subgroup $D_{\mathbf{a}}$ of $G_{\mathbf{a}}$. Given
$\phi \in I_{\chi}$ we define
\[
G(g,\phi,\chi,s):= \sum_{\gamma \in P(F) \setminus G(F)}\phi(\gamma
g)
\]
and
\[
E(g,\phi,\chi,s):=\left(\Gamma_n(s+\frac{k}{2})^g\prod_{i=0}^{n-1}L_{\mathfrak{c}}(2s-i,\chi_1
\epsilon^i)\right) G(g,\phi,\chi,s),
\]
where $\Gamma_n(s):=\pi^{n(n-1)/2}\prod_{m=0}^{n-1}\Gamma(s-m)$.
Both of them are automorphic forms of $G$ for $Re(s) \gg 0$.
Actually if we the (automorphic) type of the character is
$k1_{\Sigma}$ (that is, the arithmetic is $-k\Sigma$), then it is
known (see \cite[Theorem 19.3]{Shimurabook1} and \cite[Theorem
17.12]{Shimurabook2}) that
\begin{enumerate}
\item the Eisenstein series $G(g,\phi,\chi,\frac{k}{2})$ is an automorphic form of weight $k$ when $k
\geq n$ except when $F=\mathbb{Q}$, $k=n+1$ and the restriction
$\chi_1$ of $\chi$ to $\mathbb{Q}$ is equal to $\epsilon^{n+1}$
where $\epsilon$ the non-trivial character of $K/\mathbb{Q}$,
\item the normalized Eisenstein series
$E(g,\phi,\chi,\frac{\nu}{2})$ is an automorphic form of weight $k$
for $\nu =2n-k$ except when $0 \leq \nu < n$, $\mathfrak{c}=
\mathfrak{g}$ and $\chi_1=\epsilon^k$. In particular we have that
the normalized Eisenstein series is an automorphic form of weight
$n$ for $s=\frac{n}{2}$ when $k=n$.
\end{enumerate}

\textbf{Fourier expansion of automorphic forms of $G$}: We let
$\psi:\mathbb{A}_F/F \rightarrow \mathbb{C}^{\times}$ be the
non-trivial additive character with the property
$\psi_{\infty}(x)=exp(2\pi i \sum_{\sigma}x_{\sigma})$. Then all the
additive characters of $\mathbb{A}_F/F$ can be obtained as
$\psi_a(x):=\psi(ax)$ with $a \in F^{\times}$. For $\beta$ a
hermitian $n \times n$ matrix we define the character
\[
\psi_{\beta}:U_P(F) \setminus U_P(\mathbb{A}_F) \rightarrow
\mathbb{C},\,\,\,n(b) \mapsto \psi(tr(\beta b))
\]
where we have used the fact that there is an isomorphism
$n:S_n(\mathbb{A}_F) \cong U_P(\mathbb{A}_F)$. Here $U_P$ is the
unipotent radical of the parabolic group $P$ which is given by
\[
U_P(F):= \left\{\left(
               \begin{array}{cc}
                 1_n & X \\
                 0 & 1_n \\
               \end{array}
             \right)| X \in S_n(F)
\right\}
\]

The $\beta$-th Fourier coefficient of $G(g,\phi,\chi,s)$ is given by
\[
G_{\beta}(\phi,\chi,s)(g)=\int_{U_P(F) \setminus
U_P(\mathbb{A}_F)}G(ug,\phi,\chi,s)\psi_{-\beta}(u)du
\]

When $\beta$ is of full rank $n$ the Fourier coefficient equals up
to a normalized factor the Whittaker function
$W_{\beta}(g,\phi,s)=\prod_v W_{\beta,v}(g_v,\phi_v,s)$ with
\[
W_{\beta,v}(g_v,\phi_v,s)=\int_{U_P(F_v)}\phi_v(wn_vg,\chi_v,s)\psi_{-\beta}dn_v
\]
with $w=\left(
          \begin{array}{cc}
            0 & -1_r\\
            1_r & 0\\
          \end{array}
        \right)$.

\textbf{The local sections:} We pick a finite set $S$ of places of
$F$ that include (i) all places $v$ in $F$ above $p$, (ii) all
places of $F$ ramified in $K/F$ (iii) all places of $F$ that are
below the conductor of $\chi$ (iv) all places of $F$ such that after
localization of $(V,\phi)$ at $v$ we have $(V_v,\phi_v) \equiv
(T_v,\theta_v) \oplus (H_{r_v},\eta_{r_v})$ with $t_v:=dim(T_v)=2$.
It is known (see (\cite[Prop. 10.2 (1)]{Shimurabook1} that this set
is finite. (v) Finally $S$ contains all archimedean places of $F$.

\textit{The spherical sections:} We consider the places $v$ of $F$
that are not in $S$. For such a $v$, we pick the section $\phi_v \in
I_v(\chi,s)$ to be the normalized spherical sections for the group
$D_v=D[\mathfrak{g},\mathfrak{g}]_v$, a maximal normal subgroup of
$U(n,n)(F_v)$. Here normalized means that $\phi_v(1,\chi_v,s)=1$. We
moreover define the lattice $T$ in $S:=S_n$ by
\[
T:=\left\{x \in S | tr(S(\mathfrak{r})x) \subseteq \mathfrak{g}
\right\}.
\]
Then Shimura has the local Fourier coefficients for such $v$'s
explicitly computed for $\beta$ of full rank. We summarize his
result in the next proposition.

\begin{prop}[Shimura] Let $\phi_v$ be the spherical local
section above. Let $m=m(A) \in M(F_v)$. Then
$W_{\beta,v}(m,\phi_v,s)=0$ unless $\bar{A}^{t}\beta A \in
{\mathfrak{d}_{F_v/\Q_v}^{-1}} T_v$. In this case we have
\[
W_{\beta,v}(m,\phi_v,s)=|N\circ
detA|^{n/2-s}_v\chi_v(detA)g_{\beta,m,v}(\chi_v(\varpi_v)q_v^{-2s})\prod_{j=1}^{n}L_v(2s+j,\chi_{|F}\epsilon_{K/F}^{n-j})^{-1}
\]
with $g_{\beta,m,v}(X) \in \Z[X]$ with $g_{\beta,m,v}(0)=1$. When
$v$ is unramified in $K$ and $det(a\delta_v\bar{A}^{t}\beta A) \in
\mathfrak{g}_{_v}^{\times}$ then $g_{\beta,m,v}(X)=1$.
\end{prop}

\textit{The non-spherical sections:} Now following
\cite{HLS1,HLS2,Hsieh1,Hsieh2} we deal with the non-spherical finite
places, that is the finite places $v$ of $F$ that belong to the set
$S$. We are going to distinguish between those above $p$, we call
this set $S_p$ and the rest.

\textit{Finite places not dividing $p$:} Our choice of the local
section is the one described by Ming-Lun Hsieh \cite{Hsieh1,Hsieh2}.
We first recall the sections that are defined by Shimura \cite[page
149]{Shimurabook1}. We define a function $\tilde{\phi}_v$ on $G_v$
by
\[
\tilde{\phi}_v(x)=0,\,\,\,for\,\,\,x \not \in P_v D_v
\]
and
\[
\tilde{\phi}(pw)=\chi(det(d_p))^{-1}\chi_{\mathfrak{c}}(det(d_w))^{-1}|det
d_p \bar{d}_p|^{-s},
\]
for $p \in P_v$ and $w \in D_v=D[\mathfrak{g},\mathfrak{g}]_v$. Now
we recall that we are considering a form $(W,\psi)=(V,\theta) \oplus
(H_m,\eta_m)$ and we have defined an element
$\Sigma_{\mathbf{h}}:=diag[1_m,\sigma_{\mathbf{h}},1_m,\hat{\sigma}_{\mathbf{h}}]$
with $\sigma_v=1$ for $v \not \in S$. Then we define
\[
\phi_{v}(g):=\tilde{\phi}_{v}(gw'\Sigma^{-1})
\]
where $w':=\left(
             \begin{array}{cccc}
                &  & -1_m &  \\
                &  &  & -\frac{\theta}{2} \\
               1_m &  &  &  \\
                & -\frac{\theta^{-1}}{2} &  &  \\
             \end{array}
           \right)$. We now define $u := \left(
                  \begin{array}{cc}
                    1_m &  \\
                     & \frac{\theta}{2}\sigma^*\\
                  \end{array}
                \right)$ and
the lattice $L_v := Her_n(F_v) \cap (u M_n(2\mathfrak{c}_v)u^*)$.
Then $\phi_v$ is the unique section such that
\[
supp(\phi_v) = P(F_v)wU_P(L_v),\,\,\and\,\,\phi_v(w\left(
                                                       \begin{array}{cc}
                                                         1_n & \ell \\
                                                         0 & 1_n \\
                                                       \end{array}
                                                     \right)=
                                                     \chi^{-1}_v(det
                                                     u)|det(u
                                                     \bar{u})|^{-s},
                                                     \ell \in L_v,
\]
where we recall $w=\left(
                     \begin{array}{cc}
                       0 & -1_n \\
                       1_n & 0 \\
                     \end{array}
                   \right)$. Then as it is explained in \cite{Hsieh1,Hsieh2} we have that
\[
W_{\beta,v}(m(A),\phi_v,s)=\mathbb{I}_{L^{\vee}_v}(\bar{A}^t \beta
A) |detA|_v^{n/2-s}\chi_v(detA) vol(L_v)\chi^{-1}(det u)|det(u
\bar{u})|_v^{-2s}.
\]

\textit{The sections at infinite:} We follow again
\cite{Hsieh1,Hsieh2}. In the Hermitian symmetric space
$\mathbb{H}_{n}^\mathbf{a}$ we pick a CM point $\mathbf{i}$ and
write $D_{\mathbf{a}}:= \left\{g \in G_\mathbf{a}(\mathbb{R}) | g
\mathbf{i} = g \right\} \subseteq G_{\mathbf{a}}$ for the stabilizer
of it. For example we may take
$\mathbf{i}:=(\mathbf{i}_\sigma)_{\sigma \in \Sigma}$ defined by
\[
\mathbf{i}_\sigma = \left(
                      \begin{array}{cc}
                        i 1_{s} & 0 \\
                        0 & -\frac{1}{2}\sigma(\theta) \\
                      \end{array}
                    \right) \in \mathbb{H}_n.
\]
We use the group $D_\mathbf{a}$ to identify
$G_\mathbf{a}(\mathbb{R})/D_{\mathbf{a}}$ with the symmetric space
$\mathbb{H}_n^\mathbf{a}$. That is for every element $z \in
\mathbb{H}_n^\mathbf{a}$ we find $p \in P_\mathbf{a}(\mathbb{R})$
such that $p(\mathbf{i})=z$. For $g \in G(\mathbb{R})$ and $z \in
\mathbb{H}_n^\mathbf{a}$ we define the automorphy factors attached
to $D_{\mathbf{a}}$ by
\[
J(g,z):=det(c_g z +d_g),\,\,and\,\, J'(g,z):= det (\bar{c}_g z +
\bar{d}_g)= det(g)^{-1}J(g,z)\nu(g)^2
\]
Then we have
\[
\delta(g)=|J(g,\mathbf{i})|^{-2}|det(g)|,\,\,\,g\in G(\mathbb{R})
\]
For a given integer $k$ we define the local section at infinity as
\[
\phi_{\infty,s}(g):=J(g,\mathbf{i})^{-k}\delta^{\frac{k}{2}-s}(g)
\]
We set
\[
\Gamma_n(s):=\pi^{\frac{n(n-1)}{2}} \prod_{j=0}^{n-1}\Gamma(s-j)
\]
and
\[
L_{n,\infty}(k,s):=i^{nk}2^{-n(k-n+1)}\pi^{n(s+k)}\Gamma_n(s+k)
\]
Then as it is explained in \cite{Hsieh1,Hsieh2} we have for every
$\sigma \in \Sigma$ and $p_{\sigma} \in P(\mathbb{R})$
\[
J(p_\sigma,\mathbf{i}_\sigma)^k
W_{\beta,\sigma}(p_{\sigma},\phi_{\infty,s})|_{s=\frac{k}{2}}=\left\{
                                                      \begin{array}{ll}
                                                        L_{n,\infty}(k,0)^{-1}det(\sigma(\beta))^{k-n}e^{2\pi i tr(\sigma(\beta)z_{\sigma})}, & \hbox{$\beta \geq 0$;} \\
                                                        0, & \hbox{otherwise.}
                                                      \end{array}
                                                    \right.
\]
where $z_{\sigma}:=p_{\sigma}(\mathbf{i}_{\sigma})$.

\textit{Finite places dividing $p$, the sections of Harris, Li and
Skinner, after \cite{HLS1,HLS2}:} Now we turn to the sections at
finite places above $p$ as defined by Harris, Li and Skinner. The
interested reader should also see the work of Eischen
\cite{Eischen2} for a more detailed study of these sections. We are
assuming that we are given the following data: A character $\chi$ of
$F^\times_v \times F^\times_v$, a partition $n = n_1 + n_2 + \ldots
+n_\ell$ with $\ell \in \mathbb{N}$ and an $\ell$-tuple
$(\nu_1,\ldots,\nu_{\ell})$ of characters $\nu_j$ of $F_v^\times$.
Moreover we assume that ordinary condition that is all the primes
$v$ above $p$ in $F$ split in $K$.

We identify $U(n,n)(F_v)$ with $GL(2n,F_v)$ and the character $\chi$
used in the parabolic induction of the Siegel-Eisenstein series with
the character
\[
\left(
  \begin{array}{cc}
    A & * \\
    0 & B \\
  \end{array}
\right) \mapsto \chi_1(det(B))^{-1}\chi_2(det(A))|det(AB^{-1})|^{s},
\]
where we write $\chi=(\chi_1,\chi_2)$ for $\chi$ as a character on
$F_v \times F_v$.

We are given a partition $n = n_1 + n_2 + \ldots +n_\ell$. We write
$P$ for the standard parabolic subgroup of $GL(n)$ corresponding to
the fixed partition and we let $I$ to denote the corresponding
Iwahori subgroup. That is $I$ consists of matrices $I_{ij}$ written
in blocks with respect the fixed partition of $n$ such that (i)
$Z_{jj} \in GL(n_j, \mathfrak{g}_v)$ for $1 \leq j \leq \ell$, (ii)
$Z_{ij}$ has entries in $\mathfrak{g}_v$ for $1 \leq i < j \leq
\ell$ and (iii) $Z_{ij}$ has entries in $\mathfrak{p}_v$ for $i >
j$.

We consider $\ell$ characters $(\nu_1,\ldots \nu_\ell)$ of
$F^\times_v$ and define the Schwartz function $\Phi_\nu$ on
$M(n,F_v)$ by
\[
\Phi_\nu(Z)= \left\{
             \begin{array}{ll}
               \nu_1(Z_{11})\ldots \nu_\ell(Z_{\ell \ell}), & \hbox{$Z \in I$;} \\
               0, & \hbox{otherwise.}
             \end{array}
           \right.
\]

Now we define the $\ell$-tuple $(\mu_1,\ldots, \mu_{\ell})$ where
$\mu_j:=\nu^{-1}_j \chi_2^{-1}$. Using this $\ell$-tuple we define
$\Phi_{\mu}$ in the same way that we have defined $\Phi_{\nu}$. We
pick an integer $t \in \mathbb{N}$ that is bigger than all the
conductors of the characters $\mu_j$, $j=1,\ldots,\ell$. Writing
$\mathfrak{p}$ for the maximal ideal of $\mathfrak{g}_v$ we define
the group $\Gamma(\mathfrak{p}^t)$ to be the subgroup of
$GL(\mathfrak{g}_v)$ consisting of matrices whose off diagonal
blocks are divisible by $\mathfrak{p}^t$. Then we define the
function
\[
\tilde{\Phi}_{\mu}(x)=\left\{
                     \begin{array}{ll}
                       vol(\Gamma(\mathfrak{p}^t))^{-1} \Phi_{\mu}(x), & \hbox{$x \in \Gamma(\mathfrak{p}^t)$;} \\
                       0, & \hbox{otherwise.}
                     \end{array}
                   \right.,
\]
where $ vol(\Gamma(\mathfrak{p}^t))$ is defined as in \cite{HLS2}
page 59. We now define the Bruhat-Schwartz function on $M(n,2n,F)$
by
\[
\Phi(x,y)=\tilde{\Phi}_{\mu}(x)\widehat{\Phi}_{\nu^{-1}\chi_{1}}(y)
\]
and then the section
\[
\phi_v(h;\chi,s):=f_{\Phi}(h,s):=\chi_2(deth)|deth|^{s}\int_{GL_n(F_v)}\Phi((0,Z)h)\chi_2\chi_1(detZ)|detZ|^{2s}d^\times
Z,
\]
where $\widehat{\Phi}_{\nu^{-1}\chi_{1}}$ is the Fourier transform
of $\Phi_{\nu^{-1}\chi_{1}}$ as for example defined in \cite{HLS2}.

 \textbf{Remark:} This section is slightly different from the
one of Harris, Li and Skinner. This choice will be justified by our
computations in the proof of Lemma \ref{local p integral} in the
next section. We remark here that similar modifications occur in the
works of Eischen \cite{Eischen2} and of Ming-Lun Hsieh
\cite{Hsieh1,Hsieh2}.

We now compute the local Fourier coefficient at $v$. By definition
we have that
\[
W_{\beta,v}(g,\phi_v,s)= \int_{S_n} \phi_v(w_n n(S) g)
\psi_{-\beta}(S) d S,
\]
where $w_n=\left(
             \begin{array}{cc}
               0 & -1_n \\
               1_n & 0 \\
             \end{array}
           \right)$ and $n(S):=\left(
                                 \begin{array}{cc}
                                   1_n & S \\
                                   0 & 1_n \\
                                 \end{array}
                               \right)$. Note that since we are in
the split case we have that $S_n = M_n(F_v)$. Hence putting also the
definition of our section we obtain
\[
W_{\beta,v}(1,\phi_v,s)=\int_{M_n(F_v)} \int_{GL_n(F_v)}
\Phi\left((0,Z) \left(
                  \begin{array}{cc}
                    0 & -1_n \\
                    1_n & X \\
                  \end{array}
                \right)
\right)\chi_2\chi_1(detZ)|detZ|^{2s}  \psi_{-\beta}(X) d^\times Z dX
\]
\[
=\int_{M_n(F_v)} \int_{GL_n(F_v)} \Phi(Z,ZX)
\chi_2\chi_1(detZ)|detZ|^{2s}  \psi_{-\beta}(X) d^\times Z dX
\]
But
$\Phi(Z,ZX)=\tilde{\Phi}_{\mu}(Z)\widehat{\Phi}_{\nu^{-1}\chi_1}(ZX)$
and hence the integral above reads
\[
 \int_{GL_n(F_v)}\tilde{\Phi}_{\mu}(Z)\left(\int_{M_n(F_v)}\widehat{\Phi}_{{\nu}^{-1}\chi_1}(ZX)\psi_{-\beta}(X) dX\right)\chi_2\chi_1(detZ)|detZ|^{2s}d^\times Z
\]
But by the Fourier inversion formula, after setting $X\mapsto
Z^{-1}X$, we have that
\[
\int_{M_n(F_v)}\widehat{\Phi}_{{\nu}^{-1}\chi_1}(ZX)\psi_{-\beta}(X)
dX=|detZ|^{-n} \Phi_{\nu^{-1}\chi_1}(\transpose{Z}^{-1}
\transpose{\beta}),
\]
and hence we have that
\[
W_{\beta,v}(1,\phi_v,s)=
\int_{GL_n(F_v)}\tilde{\Phi}_{\mu}(Z)|detZ|^{-n}
\Phi_{\nu^{-1}\chi_{1}}(\transpose{Z}^{-1}
\transpose{\beta})\chi_2\chi_1(detZ)|detZ|^{2s}d^\times Z
\]
By the definition of $\tilde{\Phi}_{\mu}$ we have that the integral
is over $\Gamma_0(\mathfrak{p}^t)$ and by the support of
$\Phi_{\nu^{-1}\chi_1}$ it is non zero if and only if $det(\beta)
\neq 0$. In this case, after making the change of variables $Z
\mapsto Z \beta$, we get
\[
W_{\beta,v}(1,\phi_v,s)=\chi_2\chi_1(det(\beta))|det(\beta)|^{2s-n}\int_{\Gamma_0(\mathfrak{p}^t)}\tilde{\Phi}_{\mu}(Z\beta)
\Phi_{\nu^{-1}\chi_{1}}(\transpose{Z}^{-1})\chi_2\chi_1(detZ)|detZ|^{2s-n}d^\times
Z.
\]
\[
=\chi_2\chi_1(det(\beta))|det(\beta)|^{2s-n}\Phi_{\mu}(\beta)\int_{\Gamma_0(\mathfrak{p}^t)}\tilde{\Phi}_{\mu}(Z)
\Phi_{\nu^{-1}\chi_{1}}(\transpose{Z}^{-1})\chi_2\chi_1(detZ)d^\times
Z.
\]
\[
=\chi_2\chi_1(det(\beta))|det(\beta)|^{2s-n}\Phi_{\mu}(\beta),
\]
because of the normalization of $\tilde{\Phi}_{\mu}(Z)$. We
summarize the computations in the following lemma
\begin{lem} With the choices as above we have
\[
W_{\beta,v}(1,\phi_v,s)=\chi_2\chi_1(det(\beta))|det(\beta)|^{2s-n}\Phi_{\mu}(\beta).
\]
\end{lem}

We note also the following lemma, a proof of which can be found in
\cite[page 52]{HLS2}.
\begin{lem} Let $m(A)=\left(
                        \begin{array}{cc}
                          A & 0 \\
                          0 & \transpose{\bar{A}}^{-1} \\
                        \end{array}
                      \right) \in U(n,n)(F)$ for $A \in GL_n(K)$. Then
\[
W_{\beta,v}(m(A)g_v,\phi_v,s) |N \circ
det(A)|^{\frac{n}{2}-s}\chi_v(det(A)) W_{\transpose{\bar{A}}\beta
A,v}(g_v,\phi_v,s).
\]
\end{lem}

 \textbf{Normalization:} We now normalize the Eisenstein
series $G(g,\phi,\chi,s)$. We introduce the quantity
\[
c(n,F,K):=2^{n(n-1)[F:\mathbb{Q}/2]}|\delta(F)|^{-n/2}|\delta(K)|^{-n(n-1)/4}
\]
Then as it is explained in \cite{Shimurabook1} page 153 we have $dx
= c(n,F,K) \prod_{v}dx_v$. Here $dx$ is the Haar measure on
$S_{\mathbb{A}}$ normalized such that $\int_{S_{\mathbb{A}}/S}dx=1$.
For $v \in \mathbf{h}$ the $dx_v$ are measures on $S_v$ that are
normalized to give volume 1 to the maximal compact subgroup
$\L_v:=S(O_F)_v$. For $v \in \mathbf{a}$ we refer to Shimura. Then
we define
\[
C^S(n,K,s):=c(n,K^+,K)\prod_{j=0}^{n-1}L^{S}(n+j+s,\chi\epsilon^j)^{-1}L_{n,\infty}(k,s)^{-1}
\]
and normalize
\[
E(g,\phi,\chi,s):=C^{S}(n,K,s)^{-1}G(g,\phi,\chi,s)
\]
We now pick $g_f:=m(A) \in M(\mathbb{A}_f)$, $A \in
GL_n(\mathbb{A}_{f,K})$ and consider the hermitian form
\[
E(z,m(A),\phi,\chi):=J_k(g_\infty,\mathbf{i})E(g,\phi,\chi,s)_{|s=0}
\]
where $g=g_f g_\infty$ with $g_\infty \mathbf{i}=z \in
\mathbb{H}_{n}^\mathbf{a}$. As we have seen the definition of the
section $\phi$ depends on the characters $\chi$ and $\nu$. Hence we
may sometimes write $E(z,m(A),\nu,\chi)$ instead of
$E(z,m(A),\phi,\chi)$.

As in \cite{HLS2} we write $m(A)=m(A^{(p)}\prod_{v|p}h_v$ where we
identify $h_v$ with elements in $GL_n(F_v)$ and in particular write
$h_v=diag[A(h_v),B(h_v)]$. Moreover we recall that $\chi_v
=(\chi_{1,v},\chi_{2,v})$. From the remarks above we have that the
global Fourier coefficient of $E(z,m(A),\phi,\chi)$ at $\beta \in S$
is given by
\[
E_\beta(m(A),\phi,\chi)=T^0(\beta,m(A))N_{F/\mathbb{Q}}(det(\beta))^{(k-n)}\chi(det(A))|det(A)|^n
\times
\]
\[ \prod_{v \in
\Sigma_p}\chi^{-1}_{1,v}(det(B(h_v)))\chi_{2,v}(det(A(h_v)))\Phi_{\mu}(\transpose{A(h_v)}\beta
B(h_v)^{-1}) \times \]
\[
\prod_{v \in S^{(p)}}\mathbb{I}_{L^{\vee}_v}(\transpose{\bar{A}}
\beta A)
 vol(L_v)\chi_v^{-1}(det u)
\]
where
\[
T^0(\beta,m(A)):=\prod_{v \not  \in
S}\mathbb{I}_{\mathfrak{d}^{-1}_{F_v/\mathbb{Q}_v}T_v}(\transpose{\bar{A}}
\beta A)g_{\beta,m(A),v}(\chi_v(\varpi_v))
\]
We define
\[
Q(\beta,A,k,\nu):=N_{F/\mathbb{Q}}(det(\beta))^{(k-n)}|det(A)|^n\prod_{v
\not \in
S}\mathbb{I}_{\mathfrak{d}^{-1}_{F_v/\mathbb{Q}_v}T_v}(\transpose{\bar{A}}
\beta A)\prod_{v \in
S^{(p)}}\mathbb{I}_{L^{\vee}_v}(\transpose{\bar{A}} \beta A)
 vol(L_v)
\]
\[
P_{\beta,m(A)}:=\prod_{v \not  \in
S}g_{\beta,m(A),v}(\varpi_v)=\sum_{(\mathfrak{a},S)=1}n_{\mathfrak{a}}(\beta,m(A))\mathfrak{a},\,\,\,n_\mathfrak{a}(\beta,m(A))
\in \mathbb{Z},
\]
where $n_{\mathfrak{a}}(\beta,m(A))=0$ for almost all $\mathfrak{a}$
and
$P^{(S)}_{\beta,m(A)}(\chi):=\chi(P^{(S)}_{\beta,m(A)}):=\sum_{(\mathfrak{a},S)=1}n_{\mathfrak{a}}(\beta)\chi^{(S)}(\mathfrak{a})$.
Then
\[
E_\beta(m(A),\phi,\chi)=Q(\beta,A,k,\nu)\chi(det(A))\left(\prod_{v
\in
\Sigma_p}\chi_{1,v}^{-1}(det(B(h_v)))\chi_{2,v}(det(A(h_v)))\Phi_{\mu}(A(h_v)^t\beta
B(h_v)^{-1})\right)\times
\]
\[
\left(\prod_{v \in S^{(p)}}\chi_v(det
u^{-1})\right)P^{(S)}_{\beta,m(A)}(\chi)=
\]
\[
Q(\beta,A,k,\nu)\sum_{(\mathfrak{a},S)=1}n_{\mathfrak{a}}(\beta,m(A))\chi(det(A))\left(\prod_{v
\in
\Sigma_p}\chi^{-1}_{1,v}(det(B(h_v)))\chi_{2,v}(det(A(h_v)))\Phi_{\mu}(\transpose{A(h_v)}\beta
B(h_v)^{-1})\right)\times
\]
\[
\left(\prod_{v \in S^{(p)}}\chi_v(det
u^{-1})\right)\chi^{(S)}(\mathfrak{a})
\]

Let now $\epsilon:=\sum_{j}c_j\chi_j$ be a locally constant function
of $G^{ab}_K$ written in a unique way as a finite sum of finite
characters. Moreover we let $\psi$ be a character of infinite type
$k\Sigma$ as above. Then we define
\[
E(z,m(A),\phi,\epsilon\psi):=\sum_j c_j E(z,m(A),\phi,\chi_j\psi)
\]
and then
\[
E_\beta(m(A),\phi,\epsilon\psi)=Q(\beta,A,k)
\sum_{(\mathfrak{a},S)=1}n_{\mathfrak{a}}(\beta,m(A)) \sum_j c_j
\chi_j\psi(det(A)) \times
\]
\[
\left(\prod_{v \in
\Sigma_p}\chi_{1,v,j}\psi_{1,v}^{-1}(det(B(h_v)))\chi_{2,v,j}\psi_{2,v}(det(A(h_v)))\Phi_{\mu}(\transpose{A(h_v)}\beta
B(h_v)^{-1})\right)\times
\]
\[
\left(\prod_{v \in
S^{(p)}}\chi_{v,j}\psi_v(det
u^{-1})\right)\chi_j^{(S)}(\mathfrak{a}).
\]

Moreover it follows easily from the above description that
$E_\beta(m(A),\phi,\epsilon\psi) \in \mathbb{Z}[\epsilon\psi,\nu]$.
Indeed one needs to observe that the values of $\Phi_\mu$ are also
given by the local characters at $v |p$.

\textbf{Notation:} We will write $E^\nu_{\epsilon\psi}(\cdot)$ for
the algebraic counterpart of the Eisenstein series
$E(z,m(A),\phi,\epsilon\psi)$ as well as for its $p$-adic avatar.

\textbf{The Eisenstein measure of Harris, Li and Skinner:} In the
definition of the $p$-adic sections we have fixed an integer $\ell
\in \mathbb{N}$. We consider the torus
\[
T(\ell):= T(\ell)^0 \times
Res_{\mathfrak{r}/\mathbb{Z}}\mathbb{G}_{m}/\mathfrak{r}
\]
where
$T(\ell)^0:=(Res_{\mathfrak{g}/\mathbb{Z}}\mathbb{G}_{m}/\mathfrak{g})^\ell$.
Note that the set of $\mathbb{Z}_p$ points of $T(\ell)^0$ is
canonically isomorphic to $\prod_{v|p}(\mathfrak{g}_v^\times)^\ell$
which the ordinary assumption allows us to identify with $\prod_{v
\in\Sigma_p}(\mathfrak{r}_v^\times)^\ell$. We write
$X_{fin}(T(\ell))$ for the set of finite characters of $T(\ell)$.
This set can be parameterized by the $(\ell +1)-$tuples
$(\nu_1,\ldots,\nu_\ell,\chi)$ where $\nu_j$ are characters of
$\prod_{v|p}\mathfrak{g}_v^\times$ and $\chi$ is a character of
$\prod_{v | p}(\mathfrak{r}_v^\times)$. Note that the decomposition
\[
\prod_{v | p}\mathfrak{r}_v^\times=\prod_{v
\in\Sigma_p}\mathfrak{r}_v^\times \times \prod_{v \in
c\Sigma_p}\mathfrak{r}_v^\times
\]
allow us to decompose $\chi$ as a pair $(\chi_1,\chi_2)$. Recall
that to the tuple $(\nu_1,\ldots,\nu_\ell,\chi)$ we have attached
another tuple $(\mu_1,\ldots,\mu_\ell)$. Then Harris, Li and Skinner
in \cite[page 67]{HLS2} have obtained:
\begin{thm} There is a measure $\mu^{HLS}_{Eis}$ on $T(\ell)$ with
the property that, for any $(\ell+1)$-tuple
$(\nu_1,\ldots,\nu_\ell,\chi)$ we have
\[
\int_{T(\ell)} (\nu_1,\ldots,\nu_\ell,\chi) d\mu^{HLS}_{Eis}=
E(\cdot,\phi,\chi),
\]
where $\phi$ is the section described above.
\end{thm}

For the applications that we have in mind, we are going to keep the
tuple $(\nu_1,\ldots,\nu_\ell)$ fixed and vary $\chi$, which in
particular is going to be the $p$-part of a Hecke character with
fixed infinite type of the form $k\Sigma$, for $\Sigma$ a selected
CM-type. Let us explain now how we are going to fix the $\ell$-tuple
$(\nu_1,\ldots,\nu_\ell)$ for the applications that we have in mind.

\begin{defn} (see \cite[pages 14-15]{HLS1} Let $\pi$ be an irreducible cuspidal automorphic representation of
$U(V)(\mathbb{A}_F)$ with $dim_K(V)=n$. Let $v$ be place of $F$ that
splits in $K$. Then $\pi$ is called of type $\nu
=(\nu_1,\ldots,\nu_\ell)$ at $v$ if $\pi_v$ is a principal series of
$U(V)(F_v)\cong GL_n(F_v)$ and if it is an eigenvector for
$P(\mathfrak{g}_v) \subset GL_n(F_v)$ with eigenvalues given by the
$\ell$-tuple $(\nu_1,\ldots,\nu_\ell)$. Of course here $P$ is the
parabolic of $GL_n$ corresponding to the fixed partition
$n=\sum_{j=1}^\ell n_j$.
\end{defn}

For the construction of the Eisenstein measure we need that the
representation $\pi_v$ needs to be $P$-ordinary at $v$. Following
\cite[page 14]{HLS1}, we explain briefly what that means. Let us
write $B$ for the standard Borel subgroup of $GL_n$. We assume that
$\pi_p$ is induced from the character $\alpha=\alpha_1 \otimes
\ldots \otimes \alpha_n :B \rightarrow \bar{\mathbb{Q}}^\times$
where $\alpha_j:F_v^\times \rightarrow \bar{\mathbb{Q}}^\times$ are
finite order characters. If $\pi_v$ is of type
$(\nu_1,\ldots,\nu_\ell)$ then we can order the characters
$\alpha_j$ in such a way so that their restriction
$\tilde{\alpha}_j$ to $\mathfrak{g}_v^\times$ satisfy
\[
\tilde{\alpha}_j=\nu_j,\,\,\sum_{l=1}^{i-1}n_l < j \leq
\sum_{l=1}^{i}n_l
\]
Fixing a $p$-adic valuation $|\cdot|_p$ of $\bar{\mathbb{Q}}$ we
define $m_j:=|\alpha_j(v)|_p$. Then $P$-ordinary means that
\[
\sum_{j=1}^{n_1 + \ldots + n_i} m_j = \sum_{j=1}^{n_1+ \ldots +
n_i}(j-1+k),\,\,\,i=1,\ldots\ell
\]
for some fixed constant $k$.

\section{Construction of the $p$-adic measure
$\mu^{HLS}_{\pi,\chi}$.}

In this section we are going to use the Eisenstein measure of
Harris, Li and Skinner to obtain a measure that interpolates
critical values (and their twists) of the $L$-functions that we are
interested in. The path is well known, we will evaluate the
Eisenstein measure at CM points and then use the doubling method (in
this setting the analogue of Damerell's formula) to prove the
interpolation properties. We will need to compute some zeta
integrals in order to prove that our measure has the right
interpolation properties. As we mentioned in the introduction a more
general study of such measures is the subject of the work in
preparation of Eischen, Harris, Li and Skinner \cite{EHLS}. Here we
are going to restrict ourselves in the cases that we need for our
work. We consider a motive $M(\pi)/F \otimes M(\chi)/F$ as in the
introduction (here we write $\chi$ instead of $\psi$ there). The
first goal is to prove the following theorem.

\begin{thm}\label{interpolation properties} There exists a measure $\mu^{HLS}_{\pi,\chi}$ on
$G_{\mathfrak{c}}:=Gal(K(p^\infty\mathfrak{c})/K)$ such that for
every finite character $\psi$ of $G_{\mathfrak{c}}$ we have
\[
\frac{1}{\Omega_p(Y,\Sigma)}\int_{G_{\mathfrak{c}}}\psi
d\mu^{HLS}_{\pi,\chi} =
\frac{L_{S}(BC(\pi),\chi\psi,0)}{\Omega_\infty(Y,\Sigma)}\prod_{v |
p} Z_v(\pi,\check{\pi},\chi\psi,f_{\Phi}),
\]
where $BC(\pi)$ is the base changed automorphic representation to
$GL_n(K)$ and $\Omega_\infty(Y,\Sigma)$ $($resp.
$\Omega_p(Y,\Sigma)$$)$ is the archimedean $($resp. $p$-adic$)$
period corresponding to the CM pair $(Y,\Sigma)$ for a CM-algebra
$Y$ and they will be defined below. Moreover we have the following
explicit description of the factors
$Z_v(s,\pi,\check{\pi},\chi\psi,f_{\Phi})$ for the places of $v$
above $p$.
\begin{enumerate}
\item If $n=1$ and we take $\pi$ the trivial representation then for
the place $\mathfrak{p} \in \Sigma_p$ above $v$ we have
\[
Z_v(\pi,\check{\pi},\chi\psi,f_{\Phi}) = \alpha(\chi,\psi) \times
\frac{L_{\mathfrak{p}}(0,\chi\psi)}{e_{\mathfrak{p}}(0,\chi\psi)L_{\mathfrak{p}}(1,\chi^{-1}\psi^{-1})},
\]
for $\alpha(\chi,\psi)$ a factor that we make explicit in lemma
\ref{local p integral}.

\item If $n=2$ and after the identification $U(F_v) \cong GL_2(F_v)$ for all $v|p$ we have $\pi_v=\pi(\nu_1,\nu_2)$ with
$\nu_1,\nu_2$ unramified then
\[
Z_v(\pi,\check{\pi},\chi\psi,f_{\Phi})= \alpha(\chi,\psi) \times
\frac{L_{v}(0,\nu^{-1}_1\phi_1)}{e_{v}(0,\nu^{-1}_1\phi_1)L_{v}(1,\nu_1\phi_1^{-1})}\frac{L_{v}(1,\nu^{-1}_2\phi_1)}{e_{v}(1,\nu^{-1}_2\phi_1)L_{v}(0,\nu_2\phi_1^{-1})},
\]
where $\phi:=\chi\psi$ and $\phi=(\phi_1,\phi_2)$ corresponding to
the split $v\mathfrak{r} = \mathfrak{p} \bar{\mathfrak{p}}$ with
$\mathfrak{p} \in \Sigma_p$. Here, as above, the factor
$\alpha(\chi,\psi)$ will be made explicit in Lemma \ref{local p
integral}.
\end{enumerate}
\end{thm}

\textbf{Remark:} Here we should remark that according to
\cite[Proposition 4.3]{HLS1} one has a similar form for the local
integrals $Z_v(\pi,\check{\pi},\chi\psi,f_{\Phi})$ for any $n$. We
will compute these local integrals in the cases $n=1,2$ that we need
in our applications. Moreover, our theorem should be a special case
of the theorem 4.4 announced in \cite{HLS1} whose proof should
appear in \cite{EHLS}.

\textbf{The doubling method of Shimura, Piateski-Shapiro and
Rallis.} We start with an exposition to the doubling method as was
developed by Shimura, Piateski-Shapiro and Rallis. Our references
are \cite{Harris1,HLS1,HLS2}and \cite{GPR}. We write $(V,\theta)$
for an $n$-dimensional hermitian vector space over $K$ and we write
$G^\theta$ for the corresponding unitary group. As before we define
the hermitian space $(V,-\theta)$ and the $2n$-dimensional hemritian
space $(2W = W \oplus -W, \theta \oplus -\theta)$. Then we have
$U(\theta \oplus -\theta)=U(2W) \cong U(n,n)$. Later we will discuss
this isomorphism a little bit more explicit (see also \cite[page
176]{Shimurabook1}. Fixing such an isomorphism, we have an embedding
\[
G:=G^\theta \times G^\theta \hookrightarrow U(n,n)
\]
We pick a Haar measure $dg = \otimes_v dg_v$ on $G(\mathbb{A}_F)$
such that for almost places $v$ of $F$ we have $dg_v$ assigns volume
1 to a fixed hyperspecial maximal compact subgroup $K_v  \subset
G(E_v)$. We consider an irreducible cuspidal automorphic
representation $(V_\pi,\pi)$ of $G^\theta$ and we write $(
V_{\check{\pi}},\check{\pi})$ for its dual representation. We fix
now decompositions
\[
\pi \cong \bigotimes_v \pi_v,\,\,\,\,\check{\pi} \cong \bigotimes_v
\check{\pi}_v
\]
We now pick $\phi \in V_{\pi}$ and $\check{\phi} \in
V_{\check{\pi}}$ such that
\begin{enumerate}
\item The vectors $\phi$ and $\check{\phi}$ are pure tensors. That
is $\phi=\otimes_{v}\phi_v$ and $\check{\phi}=\otimes_v
\check{\phi}_v$ with $\phi_v \in V_{\pi_v}$ and $\check{\phi}_v \in
V_{\check{\pi}_v}$

\item They are normalized, that is for all $v$'s that $\pi_v$ is
spherical we pick these so that $(\phi_v,\check{\phi}_v)=1$.

\end{enumerate}

Let $\chi$ be a Hecke character of $K $ and cosnider a section
$\mathcal{F} \in I_{\chi}$. We consider the integral
\[
Z(s,\phi,\check{\phi},\chi,\mathcal{F})= \int_{G(F)\setminus
G(\mathbb{A}_F)}E((g,g'),\mathcal{F},\chi,s)\phi(g)\check{\phi}(g')\chi^{-1}(detg')dgdg'
\]
Let us now write $\mathcal{F} = \bigotimes_v \mathcal{F}_v$ with
respect to the decompositions $I_{\chi} \cong \bigotimes_v
I_{\chi,v}$. We now state the following important results:

\begin{thm}[Key Identity of Piatetski-Shapiro, Rallis and Shimura]

\[
Z(s,\phi,\check{\phi},\chi,\mathcal{F}) =
\int_{G^\theta(\mathbb{A}_F)} \mathcal{F}((g,1))
(\pi(g)\phi,\check{\phi})dg
\]
\end{thm}
This formula implies by computations of Li in \cite{Li} the formula,

\begin{thm}[Li's computations of the spherical integrals]\label{key identity}
\[
d_n(s,\chi)Z(s,\phi,\check{\phi},\chi,\mathcal{F})=
(\phi,\check{\phi}) \left(\prod_{v \in
S}Z_v(s,\phi,\check{\phi},\chi,\mathcal{F})\right)L_S(BC(\pi),\chi,s)
\]
\end{thm}

Let us explain the various notations arising in the formula
\begin{enumerate}
\item  $S$ is a finite set of places including the archimedean ones, the
places where $\pi$ and $\chi$ are ramified and the places of $F$
ramified at $K$. The local factors for $v \in S$ are given by
\[
Z_v(s,\phi,\check{\phi},\chi,\mathcal{F}):=
\int_{G^\theta(F_v)}\mathcal{F}_v((g_v,1);\chi,s)(\pi_v(g_v)\phi_v,\check{\phi}_v)
dg_v
\]
\item The inner product $(\phi,\check{\phi})$ is defined as
\[
(\phi,\check{\phi})=\int_{G^\theta(F)\setminus
G^\theta(\mathbb{A}_F)}\phi(g)\check{\phi}(g)dg
\]
\item The factor $d_n(s,\chi)$ is a product of Dirichlet $L$ functions
\[
d_n(s,\chi):= \prod_{r=0}^{n-1}L_S(2s+n-r,\chi_1
\epsilon_K^{n-1+r}),
\]
where $\chi_1$ is the restriction of the Hecke character from
$\mathbb{A}_K^{\times}$ to $\mathbb{A}_F^\times$.

\item Finally the $L$-function $L_S(BC(\pi),\chi,s))$ is the standard $L$ function associated to the automorphic
representation $BC(\pi) \otimes \chi$ of $GL_n(K)$, where $BC(\pi)$
is the base change from $G^\theta$ to $Res_{K/F}G^\theta=GL_n/K$. As
usual the $S$ subscript indicate that we have removed the Euler
factors at the places that are in $S$.
\end{enumerate}

\textbf{The complex analytic point of view:} Before we start with
the proof of the theorem we need to recall a few things that we have
already seen. We consider a hermitian space $(V,\theta)$ over $K$
with $\theta$ positive definite and we have fixed an embedding
\[
\gamma_n : G_{\theta,\theta} \hookrightarrow
G:=U(n,n),\,\,\,diag(a,b) \mapsto S^{-1}diag(a,b)S,
\]
where $S=\left(
     \begin{array}{cc}
       1_n & -\lambda \\
       -1_n & \lambda^* \\
     \end{array}
   \right).$ Let us write $M$ for the $\mathfrak{g}$-maximal $\mathfrak{r}$-lattice in $V$ used to define
the congruences subgroup $D^\theta \subset G^\theta_{\mathbf{h}}$.
That is, for an ideal $\mathfrak{c}$ of $\mathfrak{r}$ we define
\[
C:=\{\alpha \in G^\theta_{\mathbf{h}} \,\,|\,\, M\alpha
=M\},\,\,\,\tilde{M}:=\{x \in V\,\,|\,\, \theta(x,M) \subset
\mathfrak{d}^{-1}_{K/F}\}
\]
and then
\[
D^\theta(\mathfrak{c}):=\{\gamma \in C \,\,|\,\, \tilde{M}_v
(\gamma_v-1) \subset \mathfrak{c}_v M_v,\,\,\forall v |
\mathfrak{c}\}
\]

Following Shimura \cite[page 87]{Shimurabook1} we have defined an
element $\sigma \in GL_n(K)_{\mathbf{h}}$ such that $M'\sigma=M$
where $M'=\sum_{i=1}^{n}\mathfrak{r}e_i$ for some fixed basis
$\{e_i\}$ of $V$. We define the element $\Sigma_{\mathbf{h}} \in
G_\mathbf{h}$ by $\Sigma_{\mathbf{h}} :=diag[\sigma,\hat{\sigma}]$.

Let now $D$ be an open compact subgroup of $G_{\mathbf{h}}$ we fix a
set $\mathcal{C}$ of representatives of the double coset $G=G(F)
\setminus G(\mathbb{A}_{F})/D D_\infty$, where $K_\infty \cong
U(n)(F_\mathbf{a}) \times U(n)(F_\mathbf{a})$, the standard compact
subgroup in $GU(n,n)$. It is known \cite[page 73]{Shimurabook1} that
we can pick the elements in $\mathcal{C}$ in the form
$diag(r,\hat{r})$ with $r \in GL(K)_\mathbf{h}$ and $r_v=1$ for
every $v$ in a selected finite set of places $v$ of $F$. We have
already seen that an automorphic form $\phi$ of $G$ with respect to
$D$ is equivalent to tuple of hermitian modular forms $(f_r)_{r \in
\mathcal{C}}$, where we have abused the notation and wrote $r$ for
$diag(r,\hat{r})$. As it is explained in \cite[page
181]{Shimurabook1} there exist an element $U \in G_{\mathbf{a}}$
such that if we consider $2^{-1}i\theta \in
\mathbf{i}S_{\mathbf{a},+} \subset \mathbb{H}_\mathbf{a}$ and define
$z_{CM,\theta}:=U^{-1} \cdot (2^{-1} i \theta) \in
\mathbb{H}_\mathbf{a}$ then we have that
\[
\gamma_n(a,b)(z_{CM,\theta}) = z_{CM,\theta},\,\,\,(a,b) \in
G_{\theta,\theta}.
\]

We now fix a set $\mathcal{B}$ of $G^\theta(F) \setminus
G^\theta_{\mathbf{h}}/ D^\theta(\mathfrak{c})$ such that for all $b
\in \mathcal{B}$ we have $b_v = 1$ for all $v$ in the set $S$. For
an element $(b_1,b_2) \in G_{\theta,\theta,\mathbf{h}}$, where
$b_1,b_2 \in \mathcal{B}$ we write
\[
\gamma_n(b_1,b_2)=\alpha(b_1,b_2) \, r(b_1,b_2)\, k(b_1,b_2)
\]
where $\alpha(b_1,b_2) \in G$, $ r(b_1,b_2) \in \mathcal{C}$ and
$k(b_1,b_2) \in D[\mathfrak{g},\mathfrak{c}]$. By the key identity
of the doubling method we have that
\[
\int_{G_{\theta,\theta}(F) \setminus
G_{\theta,\theta}(\mathbb{A}_F)}E(\gamma_n(g,g'),\mathcal{F},\chi)\phi(g)\check{\phi}(g')\chi^{-1}(detg')dgdg'
=(\phi,\check{\phi}) \left(\prod_{v \in
S}Z_v(s,\phi,\check{\phi},\chi,\mathcal{F})\right)L_S(BC(\pi),\chi,0)
\]
If we write $z_{b_1,b_2}:=\alpha(b_1,b_2)\cdot z_{CM,\theta} \in
\mathbb{H}_{\mathbf{a}}$, then the integral on the left hand side
can be rewritten as
\[
\int_{G_{\theta,\theta}(F) \setminus
G_{\theta,\theta}(\mathbb{A}_F)}E(\gamma_n(g,g'),\mathcal{F},\chi,s)\phi(g)\check{\phi}(g')\chi^{-1}(detg')dgdg'=
\]
\[
C(s) \int_{G_{\theta,\theta}(F) \setminus
G_{\theta,\theta}(\mathbb{A}_F)/(D^\theta(\mathfrak{c}) \times
D^\theta(\mathfrak{c}))}E(\gamma_n(g,g'),\mathcal{F},\chi,s)\phi(g)\check{\phi}(g')\chi^{-1}(detg')dgdg'=
\]
\[
C(s) \sum_{b_1,b_2 \in
\mathcal{B}}E(\gamma(b_1,b_2),\mathcal{F},\chi,s)\chi^{-1}(det(b_2))\phi(b_1)\check{\phi}(b_2),
\]
where
\[
C(s):=\chi_{\mathbf{h}}(det(\sigma^*))^{-1}|det(\sigma_{\mathbf{h}})|^{-s}|2^{-n}
det(\theta)^n det(\sigma_{\mathbf{a}})^{-2}|^{s-k/2}
vol(D^\theta(\mathfrak{c}) \times D^\theta(\mathfrak{c})).
\]
Note that this plus our considerations over the doubling method are
equivalent to the formula of Shimura \cite[equation
22.11.3]{Shimurabook1} after one also multiplies the formula by
$\check{f}_b$ and takes an extra summation over $b$ (with the
notation as in Shimura formulas).

Now we explain the first equation, the second being trivial. We have
to study the integrand
\[
E(\gamma_n(g,g'),\mathcal{F},\chi,s)\phi(g)\check{\phi}(g')\chi^{-1}(detg')
\]
with respect to the left translation by elements of
$\gamma_n(D^\theta(\mathfrak{c}) \times D^\theta(\mathfrak{c}))$.
The first remark is that for $k \in
D^\theta(\mathfrak{c})_\mathbf{h}$ we have that $det(k) \in 1 +
\mathfrak{c}$ and hence $\chi^{-1}(k)=1$ as the conductor of $\chi$
divides $\mathfrak{c}$. Indeed from \cite[page 88]{Shimurabook1} we
have for every finite place $v$ of $F$ that
\[
k_v \in D^\theta(\mathfrak{c})_v \Leftrightarrow {\theta'}_v^{-1}
(\sigma k_v \sigma^{-1}-1)_v \prec \mathfrak{c}_v
\mathfrak{d}_{K/F,v}.
\]
In particular since $\theta' \prec \mathfrak{d}^{-1}_{K/F}$ we
conclude that for $k \in D^\theta(\mathfrak{c})$ we have that
$det(k)-1=det(\sigma k \sigma^{-1})-1 \prec \mathfrak{c}$ which
concludes our claim. Now we state the following facts that are taken
from \cite[pages 178, 179]{Shimurabook1}.

\begin{lem} Set $\mathfrak{b}:=\kappa^{-1}\mathfrak{d}_{K/F} \cap F$. Let $\epsilon := \Sigma_{\mathbf{h}} S^{-1}diag[1,\gamma]S
\Sigma_{\mathbf{h}}^{-1}$ with $\gamma \in D^\theta(\mathfrak{c})$.
Then $\epsilon \in D[\mathfrak{b}^{-1},\mathfrak{bc}]$ and
$d_{\epsilon}-1 \prec \mathfrak{c}$ for $\epsilon=\left(
                              \begin{array}{cc}
                                a_{\epsilon} & b_{\epsilon} \\
                                c_{\epsilon} & d_{\epsilon} \\
                              \end{array}
                            \right)$.
\end{lem}

Our choice of $\mathfrak{c}$ (i.e. $\mathfrak{d}_{K/F,v} \neq
\mathfrak{r}_v$ implies $v | \mathfrak{c}$) we have from \cite[page
177, lemma 21.4(iii)]{Shimurabook1} that for $\alpha \in
D^\theta(\mathfrak{c})$ we have that $x:=\Sigma_{\mathbf{h}}
\gamma_n(\alpha,1)\Sigma_{\mathbf{h}}^{-1} \in
D[\mathfrak{b}^{-1},\mathfrak{bc}]$. Moreover by \cite[equation
21.6.3, page 179]{Shimurabook1} for all finite places $v$ of $F$
with $v | \mathfrak{c}$ we have that
\[
(det(d_x)^{-1}det(\sigma^*))_v \in 1 + \mathfrak{c}_v
\]

These remarks and the modular properties of the Eisenstein series
and the automorphic forms $\phi$ and $\hat{\phi}$ allow us to
conclude that for $k,k' \in D^\theta(\mathfrak{c})_{\mathbf{h}}$ we
have
\[
E(\gamma_n(gk,g'k'),\mathcal{F},\chi,s)\phi(gk)\check{\phi}(g'k')\chi^{-1}(detg'k')=
\]
\[
\chi_{\mathbf{h}}(det(\sigma^*))^{-1}
E(\gamma_n(g,g'),\mathcal{F},\chi,s)\phi(g)\check{\phi}(g')\chi^{-1}(detg')
\].

Before we proceed to the proof of Theorem \ref{interpolation
properties} we need to define the $p$-adic and archimedean periods.
We do that next.

\textbf{The archimedean and $p$-adic periods
$\Omega_\infty(Y,\Sigma)$ and $\Omega_p(Y,\Sigma)$.} Now we need to
explain how we pick the complex and $p$-adic periods that will
appear in the interpolation formula. Our definition of these periods
is the natural extensions of that of Katz \cite[page 268]{Katz2} in
our setting.

In general we start with $(W,\theta)$, a positive definite Hermitian
space of dimension $n$ over a CM field $K$ with $d:=[F:\mathbb{Q}]$,
where $F:=K^+$. We write $G$ for $G^\theta$ and fix a maximal open
compact subgroup of $G(\mathbb{A}_{F,f})$. We note that the Shimura
variety
\[
Sh_G(U):= G(F) \setminus G(\mathbb{A}_{F,f})/U
\]
is zero dimensional and parameterizes abelian varieties with CM by
the CM-algebra $Y:=\bigoplus_{j=1}^n K$ and additional additive
structure determined by the open compact subgroup $U$. Indeed if $U$
is defined as the open compact subgroup of $G_f$ that fixes an
$\mathfrak{r}$ lattice $L$ in $W \cong K^n$, then $Sh_G(U)$ is
simply the set of classes of $L$ contained in the genus of $L$ (see
\cite[page 62]{Shimurabook1}). For our considerations we assume that
we may pick $L=\sum_{i=1}^n\mathfrak{r}e_i$ with respect to the
standard basis of $W$ over $K$.

We may now pick (see \cite[page 65]{Shimurabook1}) representatives
$\{L_i\}^h_{i=1}$ of the classes of $L$ such that $L_i=L \cdot
g_{i}$ with $g_i \in G(\mathbb{A}_{F,f})$ such that the ideal of $K$
corresponding the the idele $det(g_i)$ is relative prime to $p$. We
write $X(L_i)$ for the abelian variety corresponding to the lattice
$L_i$. We define $A:=\{a \in \overline{\mathbb{Q}} \,\,|\,\,
incl_p(a) \in D_p\}$, where $D_p$ is the ring of integers of
$\mathbb{C}_p$ and $incl_p :\overline{\mathbb{Q}} \hookrightarrow
\overline{\mathbb{Q}}_p \hookrightarrow \mathbb{C}_p$ the fixed
$p$-adic embedding. As in Katz (loc. cit.), we have that
$Lie(X(L_i)) = Lie(X(L))$ for all $1 \leq i \leq h$, where the
equality is in $L \otimes_{\mathfrak{r}} A$.

Now we let $\omega(L)$ be a nowhere vanishing differential of $X(L)$
over $A$, that is it induces through diality an isomorphism
\[
\omega(L): Lie(X(L)) \stackrel{\sim}{\rightarrow}
\mathfrak{d}_K^{-1} \otimes A.
\]
From the fixed isomorphisms $Lie(X(L_i)) = Lie(X(L))$ we obtain for
each $X(L_i)$ a nowhere vanishing differential $\omega(L_i)$ by the
composition of this isomorphism with $\omega(L)$, that is
\[
\omega(L_i):Lie(X(L_i)) = Lie(X(L))\stackrel{\omega(L)}{\rightarrow}
\mathfrak{d}_K^{-1} \otimes A.
\]

Now we write $\omega_{trans}(L_i)$ for the nowhere vanishing
differential on the complex analytic abelian variety
$X(L_i)/\mathbb{C}:=X(L_i) \times_{A} \mathbb{C}$, obtained after a
fixed embedding $incl_{\infty}:A \hookrightarrow \mathbb{C}$,
corresponding to the lattice $L_i \subset
\mathbb{C}^{n[F:\mathbb{Q}]}$. Then as in Katz \cite[page
269]{Katz2} from the very definition of the $\omega(L_i)$'s we
obtain the following lemma.

\begin{lem}\label{definition of algebraic periods} There exists an element
$\Omega_\infty(Y,\Sigma)=(\ldots,\Omega_\infty(i,\sigma),\ldots)_{i=1,\ldots,n,
\sigma \in \Sigma}$ in $(\mathbb{C}^\times)^{n|\Sigma|}$ such that
for all selected $L_i$ as above we have
\[
\omega(L_i) = \Omega_\infty(Y,\Sigma) \cdot \omega_{trans}(L_i).
\]
\end{lem}

Now we can also define in the same way as in Katz the $p$-adic
periods. Our first step is to explain how we can give a
$p^\infty$-structure to the abelian varieties $X(L_i)$. We recall
that the ordinary condition implies that all the primes above $p$ in
$F$ split in $K$. Then if we consider $L_p:= L \otimes \mathbb{Z}_p$
the splitting condition implies a splitting $L_p= L_p(\Sigma) \oplus
L_p(\rho\Sigma)$ with $L_p(\Sigma) \cong L_p(\rho\Sigma) \cong
\sum_{i=1}^n \mathfrak{g} \otimes \mathbb{Z}_p$. Then the embedding
$\sum_{i=1}^n \mathfrak{g} \otimes \mathbb{Z}_p \hookrightarrow L_p$
to the first component provides the needed $p^\infty$-structure.
Through the isomorphisms $X(L) \cong X(L_i)$ we define the
$p^\infty$-structure to the rest of the varieties.

Hence, after extension of scalars $incl_p : A \hookrightarrow D_p$,
we may consider the canonical differential $\omega_{can}(L_i)$
associated to the $p^\infty$ arithmetic structure of $X(L_i)$. Then
as in Katz (loc. cit.) we have the following lemma,

\begin{lem} There exists a unit $\Omega_p(Y,\Sigma) =
(\ldots,\Omega_p(i,\sigma),\ldots) \in (D_p^\times)^{n|\Sigma|}$
such that for the selected $L_i$'s above we have
\[
\omega(L_i)=\Omega_p(Y,\Sigma) \cdot \omega_{can}(L_i).
\]
\end{lem}
\begin{proof} The proof is exactly as in Katz \cite[lemma
5.1.47]{Katz2}; one has only to remark that over $A$ we have an
isomorphism $X(L_i)[p^\infty] \cong X(L)[p^\infty]$ induced from the
identification $L \otimes_\mathfrak{r} A = L_i
\otimes_{\mathfrak{r}} A$.
\end{proof}

Our next goal is to relate the periods that we have associated to
the abelian varieties  of the definite unitary groups $U(n)/F$ to
abelian varieties of the definite unitary group $U(1)/F$. We start
by recalling the following theorem of Shimura \cite[page
164]{ShimuraCM}.

\begin{prop}Let $(A,\imath)$ be an abelian variety of type
$(K,\Psi)$ with a CM-field $K$ and $\Psi$ of the form
$\Psi_v(a)=diag[a_v 1_{r_v}, \bar{a}_v 1_{s_v}]$ (Thus
$dim(A)=d=m[K^+:\mathbb{Q}]$ and $r_v + s_v = m$). If $\sum_{v \in
\mathbf{a}}r_vs_v=0$, then $A$ is isogenous to the product of $m$
copies of an abelian variety belonging to a CM-type $(K,\Phi)$ with
$\Phi$ such that $\Psi$ is equivalent to the sum of $m$ copies of
$\Phi$.
\end{prop}

We now write $B$ for an abelian variety with complex multiplication
by the CM field $K$ and type $\Sigma$. As in Katz \cite{Katz2} or
deShalit \cite{deShalit} we fix a pair
$(\Omega_{\infty}(\Sigma),\Omega_{p}(\Sigma)) \in
(\mathbb{C}^\Sigma,D_p^\Sigma)^{\times}$ of a complex and $p$-adic
period. As with the pair
$(\Omega_\infty(Y,\Sigma),\Omega_p(Y,\Sigma)$ the definition is
again independent of the particular $B$ and depends only on the
$(K,\Sigma)$-type. Then we have

\begin{lem} We may pick the pairs $(\Omega_\infty(Y,\Sigma),\Omega_p(Y,\Sigma)$ and $(\Omega_{\infty}(\Sigma),\Omega_{p}(\Sigma))$ so that
\[
(\Omega_\infty(Y,\Sigma),\Omega_p(Y,\Sigma)=(\Omega_{\infty}(\Sigma)^n,\Omega_{p}(\Sigma)^n)
\]
\end{lem}
\begin{proof} The only thing that we need to remark, is that with
notation as above we have that $X(L) = \oplus_{i=1}^n
X(\mathfrak{r})$, where $X(\mathfrak{r})$ is the abelian variety
with CM of type $(K,\Sigma)$ associated to the lattice
$\Sigma(\mathfrak{r}) \subset \mathbb{C}^{[F:\mathbb{Q}]}$.
\end{proof}

\textbf{Proof of Theorem \ref{interpolation properties}:} Now we
ready to proceed to the proof of Theorem \ref{interpolation
properties}.

\begin{proof}(of theorem \ref{interpolation properties}) We recall
that we have defined the sets of representatives $\mathcal{B}$ and
$\mathcal{C}$. We change now our notation and write
$\mathbf{f}=\otimes_v \mathbf{f}_v$ for the normalized automorphic
form associated to $\pi$. Similarly we write $\check{\mathbf{f}}$
for the associated to $\check{\pi}$. Moreover it is well know that
$\mathbf{f}$ is determined by a set of data $(f(a))_{a \in
\mathcal{B}}$ where $f(a):=\mathbf{f}(a)$. Moreover we have defined
an embedding
\[
\gamma_n:G^\theta \times G^\theta \hookrightarrow G.
\]
For $a,b \in \mathcal{B}$ we write $r_{a,b}$ for the fixed
representative in $\mathcal{C}$ of the element $\gamma_n(a,b)$. We
introduce the following notation, given an $a \in \mathcal{B}$, an
$r \in \mathcal{C}$ and a $r$-polarized Eisenstein series
$E_{\chi}(\cdot)$ we denote by $E_\chi^{(a)}(\cdot)$ the Eisenstein
series
\[
E_\chi^{(a)}(\cdot):=E_{\chi}(\cdot)\chi\left(\frac{det(r)}{det(a)}\right)
\]
For every finite character $\psi$ of $G_{\mathfrak{c}}$ and $a \in
\mathcal{B}$ we define the measures
\[
\mu_{a,\chi}: \psi \mapsto \sum_{b \in \mathcal{B}}
E^{(a)}_{\psi\chi}(\underline{A_a} \times \underline{A_{b}},j_1
\times j_2)f(b),
\]
and then the measure
\[
\int_{G_{\mathfrak{c}}}\psi d\mu^{HLS}_{\pi,\chi} :=
\frac{1}{(f,\check{f})}\sum_{a \in
\mathcal{B}}\mu_{a,\chi}(\psi)\check{f}(a)=
\]
\[
\frac{1}{(f,\check{f})}\sum_{a,b \in \mathcal{B}} \chi\psi(det(b))
E_{\chi\psi}(\underline{A_a} \times \underline{A_{b}},j_1 \times
j_2) f(b)\check{f}(a),
\]
The last equality follows from the fact that
$\chi\psi\left(\frac{det(r_{a,b})}{det(a)}\right)=\chi\psi(det(b))$.
Indeed we can write $\gamma_n(a,b) = \gamma r_{a,b} k$ with $\gamma
\in G(F)$ and $k \in D(\mathfrak{c})$. In particular we have that
\[
det(a)det(b)=det(\gamma_n(a,b))=det(\gamma) det(r_{a,b}) det(k)
\]
and of course $\chi\psi (det(\gamma)det(r_{a,b})
det(k))=\chi\psi(det(r_{a,b}))$. Moreover we can assume (see
\cite[page 65]{Shimurabook1}) that $b_v=1$ for all finite places $v$
of $F$ where the representation $\pi_v$ is not spherical or $v \in
S$. Then we claim that this measure has the property:
\[
\frac{1}{\Omega_p(Y,\Sigma)}\int_{G_{\mathfrak{c}}}\psi
d\mu^{HLS}_{\pi,\chi} =
\frac{L_{S}(\pi,\chi\psi,1)}{\Omega_\infty(Y,\Sigma)}\prod_{v |
p}Z^{(a)}_v(\pi,\chi\psi)Z^{(b)}_v(\pi,\chi\psi)
\]
and $(f,\check{f})\mu^{HLS}_{\pi,\chi}$ takes integral values. We
use the doubling method as developed above. We start by observing
that,
\[
\frac{1}{\Omega_p(Y,\Sigma)}\sum_{a,b \in \mathcal{B}}
\chi\psi(det(b)) E_{\chi\psi}(\underline{A_a} \times
\underline{A_{b}},j_1 \times j_2) f(b)\check{f}(a)=\]
\[
\frac{1}{\Omega_\infty(Y,\Sigma)}\sum_{a,b \in \mathcal{B}}
\chi\psi(det(b)) E_{\chi\psi}(\underline{A_a} \times
\underline{A_{b}},\omega_\infty(A_a)\times \omega_\infty(A_b))
f(b)\check{f}(a)=
\]
\begin{equation}\label{equation1}
=\frac{1}{\Omega_\infty(Y,\Sigma)}\sum_{a\in
\mathcal{B}}\left(\sum_{b \in \mathcal{B}} \chi\psi(det(b))
E_{\chi\psi}(\underline{A_a} \times
\underline{A_{b}},\omega_\infty(A_a)\times \omega_\infty(A_b))
f(b)\right)\check{f}(a).
\end{equation}
Now we use an observation of Shimura in \cite[page 88 and
186]{Shimurabook1}. Namely, the inner summation
\[
g(a):=\sum_{b \in \mathcal{B}} \chi\psi(det(b))
E_{\chi\psi}(\underline{A_a} \times
\underline{A_{b}},\omega_\infty(A_a)\times \omega_\infty(A_b)) f(b)
\]
corresponds to the value at $a \in \mathcal{B}$ of the adelic
automorphic form $\mathbf{f}| \mathfrak{T}$ where $\mathbf{f}$ the
adelic form corresponding to $(f(a))_{a \in \mathcal{B}}$ and
$\mathfrak{T}:=\prod_{v}\mathfrak{T}_v=\prod_{v \in S}\mathfrak{T}_v
\times \prod_{v \not \in S}\mathfrak{T}_v$, where $\mathfrak{T}_v$
for $v \not \in S$ is given as in Shimura (with notation as therein
but for us here normalized as for example explained in page 168 of
\cite{Shimurabook1})
\[
\mathbf{f}_v | \mathfrak{T}_v =
\prod_{r=1}^{n}L_{v}(n-r,\chi_1\psi_1\epsilon^{r})\sum_{\tau \in D_v
\setminus \mathfrak{X}_v /D_v} (\mathbf{f}_v | D_v\tau
D_v)\chi\psi(\nu^\sigma(\tau))N(\nu^\sigma(\tau)))^{-k}
\]
and for $v \in S$ we have that
\[
\mathbf{f}_v | \mathfrak{T}_v =
\prod_{r=1}^{n}L_{v}(n-r,\chi_1\psi_1\epsilon^{r})\sum_{\tau \in D_v
\setminus \mathfrak{X}_v /D_v} (\mathbf{f}_v | D_v\tau
D_v)f_{\Phi_{\chi\psi}}((\tau,1))N(\nu^\sigma(\tau)))^{-k}
\]
But we are taking $\mathbf{f}$ an eigenform and normalized. In
particular $g(b)$ is equal to $\alpha f(b)$ where $\alpha$ is the
eigenvalue of $\mathbf{f}$ with respect to the operator
$\mathfrak{T}$. Then for all $v \in S$ we have
\[
(\mathbf{f}_v | \mathfrak{T}_v,\mathbf{f}_v)_v =
Z_v(\mathbf{f}_v,\mathbf{f}_v,\chi\psi,f_{\Phi_{\chi\psi}})
(\mathbf{f}_v,\mathbf{f}_v)_v
\]
or equivalently
\[
\mathbf{f}_v |
\mathfrak{T}_v=Z_v(\mathbf{f}_v,\mathbf{f}_v,\chi\psi,f_{\Phi_{\chi\psi}})
\mathbf{f}_v.
\]
Indeed the last equations follow directly from the definition of the
integrals
\[
Z_v(\mathbf{f}_v,\mathbf{f}_v,\chi\psi,f_{\Phi_{\chi\psi}})=\int_{G(F_v)}f_{\Phi_{\chi\psi}}((g_v,1)(\pi(g_v)\mathbf{f}_v,\mathbf{f}_v)_vdg_v
\]
and the $p$-adic Cartan decomposition of $D_v \setminus G(F_v)
/D_v$. Putting the last considerations together we get that the
equation \ref{equation1} reads
\[
\frac{1}{\Omega_\infty(Y,\Sigma)}\left(\prod_{v \in
S}Z_v(1,f,\check{f},\chi\psi,\phi)\right)\int_{(G \times
G)(F)\setminus (G \times G)(\mathbb{A}_F)/D^\theta(\mathfrak{c})
\times
D^\theta(\mathfrak{c}))}E((g,g'),\phi,\chi\psi)f(g)\check{f}(g')(\psi\chi)^{-1}(detg')dgdg'=
\]
where $E(x,\phi,\chi\psi)$ is the standard normalized Eisenstein
series appearing in Shimura \cite{Shimurabook1} that is by picking
the trivial section at the places $v \in S$ and hence also removing
the Euler factors there. Then we have that the last expression is
equal to
\[
(f,\check{f}) \left(\prod_{v \in
S}Z_v(1,f,\check{f},\chi\psi,\phi)\right)\frac{L_S(BC(\pi),\chi\psi,0)}{\Omega_{\infty}(Y,\Sigma)},
\]
where we have used the fact that $\mathcal{B}= G(F) \setminus
G(\mathbb{A}_F)/K(\mathfrak{c})$ and
$G(F_\infty)=K_\infty(\mathfrak{c})$. The factor
$\prod_{v|p}Z_v(1,f,\check{f},\chi\psi,\phi)$ will be computed
below.
\end{proof}

From its very definition it is easily seen that the measure
$(f,\check{f})\mu^{HLS}_{\pi,\chi}$ is integral valued. Hence we
could for example establish that the measure $\mu^{HLS}_{\pi,\chi}$
is integral valued if we knew that the quantity $(f,\check{f})$ is a
$p$-adic unit. It is well-known (see \cite[page 2]{HLS1} that the
$p$-divisibility of this quantity corresponds to congruences modulo
$p$ between forms in $\pi$ and other cuspidal automorphic
representations $\pi'$ of $U(n)$. Hence if we assume that there are
no congruences between forms of $U(n)$ we can conclude that the
measure $\mu^{HLS}_{\pi,\chi}$ is $p$-adic integral

\textbf{Computation of the local integrals for $v$ above $p$:} We
now compute the integrals
\[
Z_p(s,\phi,\check{\phi},\chi,\mathcal{F})=\prod_{v | p
}Z_v(s,\phi_v,\check{\phi}_v,\chi_v,\mathcal{F}_v)
\]
in the special case where $n=1$ or $2$. We start with some general
remarks with respect the Fourier transform over $GL_1$ and then we
generalize to $GL_n$. Our main references are \cite{Hsieh1,Hsieh2}.
We let $F$ be a local field with ring of integers $\mathfrak{g}$ and
we fix a uniformizer $\varpi$ of $\mathfrak{g}$ and write
$\mathfrak{p}=(\varpi)$ . For a complex character $\chi:F^\times
\rightarrow \mathbb{C}^\times$ we define the Bruhat-Schwartz
function $\Phi_\chi(x)=\chi(x)I_{\mathfrak{g}^\times}(x)$. If we
define the quantity
\[
E_v(s,\chi):=\frac{L(s,\chi)}{e(s,\chi)L(1-s,\chi^{-1})},
\]
where $L(s,\chi)$ the standard $L$ factor of $\chi$ and $e(s,\chi)$
the epsilon factors of $\chi$, then it is well known that
\[
Z(s,\chi,\Phi_{\chi^{-1}}):=\int_{F}\chi(x) \Phi_{\chi^{-1}}(x)
|x|^s d^\times x = vol(\mathfrak{g})
\]
and
\[
Z(s,\chi,\widehat{\Phi}_{\chi}):= \int_F \chi(x) |x|^s
\widehat{\Phi}_{\chi}(x) d^\times x = E_v(s,\chi),
\]
where $\widehat{\Phi}_{\chi}$ the Fourier transform of $\Phi_\chi$
defined as $\widehat{\Phi}_{\chi}(x) = \int_F \Phi_\chi(y)\psi(yx)
dy$ for an additive character $\psi : F \rightarrow
\mathbb{C}^\times$. Moreover it is well known that if we write
$c(\chi)$ for the valuation of the conductor of $\chi$ and $d$ for
the valuation of the different of $F$ over $\mathbb{Q}_p$ then
\[
\widehat{\Phi}_\chi(x)=\left\{
                         \begin{array}{ll}
                           \chi^{-1}(x)I_{c^{-1}\mathfrak{g}^\times}(x)\tau(\chi), & \hbox{$c(\chi) \neq 0$;} \\
                           I_{\mathfrak{g}}(x)-|\varpi|I_{\mathfrak{p}^{-1}}(x), & \hbox{$c(\chi)=0$.}
                         \end{array}
                       \right.
\]
where $c:=\varpi^{c(\chi) + d}$ and
$\tau(\chi)=\int_{\mathfrak{g}^\times}
\chi(\frac{x}{c})\psi(\frac{x}{c})d^\times x$, a Gauss sum related
to the local-epsilon factor by $e(\chi,s)=|c|^s \tau(\chi)^{-1}$.

We now generalize these considerations to the case of $GL_n$. For a
Bruhat-Schwartz function $\Phi$ of $M_n(F)$ we define its
Fourier-transform $\widehat{\Phi}(X):=\int_{M_n(F)}\Phi(Y)
\psi(tr(\transpose{Y}X))dY$. For a partition $n:=n_1 + \ldots +
n_\ell$ of $n$ and a set of characters
$\nu:=(\nu_1,\ldots,\nu_\ell)$ of $F^\times$ we have defined the
Bruhat-Schwartz function
\[
\Phi_\nu(X):=\left\{
               \begin{array}{ll}
                 \nu_1(det(X_{11}))\ldots\nu_\ell(det(X_{\ell \ell})), & \hbox{$X \in \Gamma_0(\mathfrak{p}^t)$;} \\
                 0, & \hbox{otherwise.}
               \end{array}
             \right.
\]

Now we recall the Godement-Jacquet zeta functions as introduced in
\cite{GJ}. So we consider an automorphic representation
$(\pi,V_{\pi})$ of $GL(n)(F)$, which always we take to be a
principal series of the form $\pi=\pi(\nu_1,\ldots,\nu_\ell)$. We
write $\omega(g):= <\pi(g)v,\tilde{v}>$ for the matrix coefficient
where $v \in V_{\pi}$ and $\tilde{v} \in \tilde{V}$ in the space of
the contragredient representation $\tilde{\pi}$ and $<\cdot,\cdot>$
properly normalized so that $<v,\tilde{v}>=1$. For a Bruhat-Schwartz
function $\Phi$ of $M_n(F)$ and a character $\chi$ of
$GL_1(F)=F^\times$ we define the integrals
\[
Z(s,\Phi,\omega,\chi):=\int_{GL_n(F)}\Phi(x)\,\chi(det(x))\,\omega(x)\,|det(x)|^s\,
d^\times x.
\]
These integrals generalize the theory of Tate (in the case $n=1$) as
it is proven in \cite{GJ}. It is known  \cite[page 80]{GJ} that if
$\pi$ and $\chi$ are spherical and we pick
$\Phi:=\mathbf{1}_{M_n(\mathfrak{g})}$ then we have that
\[
Z(s,\Phi,\omega,\chi)= L(s,\pi,\chi),
\]
that is the $L-$factor of $\pi$ twisted by $\chi$. Now we take $n=2$
and $\ell=2$. Then we have the following lemmas that generalize the
case of $GL_1$.

\begin{lem}\label{euler factor I} Consider the principal series representation $\pi = \pi(\nu_1,\nu_{2})$, where $\nu_2$ is an unramified character and define
$\omega(x)=<\pi(x)v,\tilde{v}>$ with $v \in V_{\pi}$ so that
$\pi(x)v=\nu_1(a)v$ for $x = \left(
                               \begin{array}{cc}
                                 a & b \\
                                 c & d \\
                               \end{array}
                             \right) \in \Gamma_0(\mathfrak{p}^t)$. Then the Godement-Jacquet
integral $Z(s,\Phi,\pi,\chi)$ with $\Phi:=\Phi_{\chi^{-1}\nu^{-1}}$
is equal to $vol(\Gamma_0(\mathfrak{p}^t)) <v,\tilde{v}>$.
\end{lem}
\begin{proof} By definition we have that
\[
Z(s,\Phi,\pi,\chi)= \int_{GL_2(F)} \Phi(x) \,\chi(det(x))
\,\omega(x) |det(x)|^s \, d^\times x =
\int_{\Gamma_0(\mathfrak{p}^t)} \Phi(x) \,\chi(det(x))
\omega(x)\,|det(x)|^s d^\times(x).
\]
But for $x=\left(
             \begin{array}{cc}
               a & b \\
               c & d \\
             \end{array}
           \right)
 \in \Gamma_0(\mathfrak{p}^t)$ we have that $\omega(x)=
<\pi(x)v,\tilde{v}>=<\nu_1(a)v,\tilde{v}>=\nu_1(a)<v,\tilde{v}>$.
That means we have,
\[
Z(s,\Phi,\pi,\chi)=<v,\tilde{v}> \int_{\Gamma_0(\mathfrak{p}^t)}
\Phi(x) \,\chi(det(x)) \nu_1(a) |det(x)|^s d^\times(x),\,\,x=\left(
                                                  \begin{array}{cc}
                                                    a & b \\
                                                    c & d \\
                                                  \end{array}
                                                \right)
.
\]
But by the definition of $\Phi$ we have that
$\Phi(x)=\nu^{-1}_1(a)\nu_2^{-1}(d)\chi^{-1}(ad)$ and we notice that
by the choice of $t$, i.e. bigger than the conductors of $\nu_i$ and
$\chi$, we have that
$\chi(det(x))=\chi(ad-bc)=\chi(ad)\chi(1+\mathfrak{p}^t)=\chi(ad)$
as $a,d \in \mathfrak{g}^\times$ by the definition of
$\Gamma_0(\mathfrak{p}^t)$. Putting these considerations together we
have
\[
Z(s,\Phi,\pi,\chi)=<v,\tilde{v}>\int_{\Gamma_0(\mathfrak{p}^t)}\nu^{-1}_1(a)\nu_2^{-1}(d)\chi^{-1}(ad)
\,\chi(ad) \nu_1(a) |det(x)|^s d^\times(x),\,\,x=\left(
                                                  \begin{array}{cc}
                                                    a & b \\
                                                    c & d \\
                                                  \end{array}
                                                \right)
.
\]
Since $\nu_2$ is not ramified we have that
$\nu_2(\mathfrak{g}^\times)=1$ which concludes the proof of the
lemma.
\end{proof}

We now compute the integrals $Z(s,\widehat{\Phi},\pi,\chi)$. As we
mentioned above these integrals should be computed in full
generality in \cite{EHLS}. In the following lemma we will compute
them only in the case of interest, namely for $n=2$. We note also
here that a similar integral has been computed in \cite[Lemma
6.7]{Hsieh2}.

\begin{lem}\label{euler factor II}With the setting as in lemma \ref{euler factor I} but with $\pi=\pi(\nu_1,\nu_2)$ unramified and for $\Phi:=\Phi_{\nu\chi}$ we have that the Godement-Jacquet
integral $Z(s,\widehat{\Phi},\pi,\chi)$ is given by
\[
Z(s,\widehat{\Phi},\pi,\chi)= |\mathfrak{d}_F| \times
E_v(s-1,\nu_1\chi)E_v(s,\nu_2\chi),
\]
where $\mathfrak{d}_F$ is the different of $F$.
\end{lem}
\begin{proof} We start by exploring the support of $\widehat{\Phi}$.
By definition we have that $\widehat{\Phi}(X)=\int_{M_2(F)} \Phi(Y)
\psi(tr(\transpose{Y}X))$. Writing $X=\left(
                                        \begin{array}{cc}
                                          x_1 & x_2 \\
                                          x_3 & x_4 \\
                                        \end{array}
                                      \right)$ and $Y=\left(
                                                        \begin{array}{cc}
                                                          y_1 & y_2 \\
                                                          y_3 & y_4 \\
                                                        \end{array}
                                                      \right)$ we
have that
\[
\widehat{\Phi}(X)=\int_{M_2(F)} \Phi\left(\left(
                                                        \begin{array}{cc}
                                                          y_1 & y_2 \\
                                                          y_3 & y_4 \\
                                                        \end{array}
                                                      \right)\right)
\psi(x_1y_1)\psi(x_2y_2)\psi(x_3 y_3)\psi(x_4y_4)dy_1dy_2dy_3dy_4
\]
By the definition of $\Phi=\Phi_{\nu\chi}$ we have that
\[
\widehat{\Phi}(X)=\int_{\mathcal{I}} \nu_1\chi(x_1) \nu_2\chi(x_4)
\psi(x_1y_1)\psi(x_2y_2)\psi(x_3 y_3)\psi(x_4y_4)dy_1dy_2dy_3dy_4,
\]
where $\mathcal{I} \subset M_2(F)$ the support of $\Phi$. The above
integral we can write as
\[
\left(\int_{I_1} \nu_1\chi(x_1) \psi(x_1y_1) dy_1\right)
\left(\int_{I_4}\nu_2\chi(x_4) \psi(x_4y_4) dy_4\right)\left(
\int_{I_2} \psi(x_2y_2)dy_2\right)\left( \int_{I_3}\psi(x_3
y_3)dy_3\right),
\]
where we have written $\mathcal{I}=\left(
                                     \begin{array}{cc}
                                       I_1 & I_2 \\
                                       I_3 & I_4 \\
                                     \end{array}
                                   \right)$. by definition we have
that $I_1=I_4 = \mathfrak{g}^\times$ and hence we have that
\[
\int_{I_1} \nu_1\chi(x_1) \psi(x_1y_1)
dy_1=\widehat{\Phi}_{\nu_1\chi}(y_1),\,\,\,\int_{I_4}\nu_2\chi(x_4)
\psi(x_4y_4) dy_4=\widehat{\Phi}_{\nu_2\chi}(y_4).
\]
For the other two integrals we have that
\[
\int_{I_2} \psi(x_2y_2)dy_2=\left\{
                              \begin{array}{ll}
                                vol(I_2), & \hbox{$x_2 I_2 \in \mathfrak{d}^{-1}_F$;} \\
                                0, & \hbox{otherwise.}
                              \end{array}
                            \right.
\]
and similarly for $\int_{I_3}\psi(x_3 y_3)dy_3$. Now we turn to the
integral $Z(s,\widehat{\Phi},\pi,\chi)$. By definition we have
\[
Z(s,\widehat{\Phi},\pi,\chi)=\int_{GL_2(F)} \widehat{\Phi}(x)
\,\chi(det(x)) \,\omega(x) |det(x)|^s \, d^\times x.
\]
By the Iwasawa decomposition $GL_2(F)=B(F)K$ with
$K=GL_2(\mathfrak{g})$. Hence if we write $x=b_Fk$ and $b_F=\left(
                                                          \begin{array}{cc}
                                                            a & y \\
                                                            0 & b \\
                                                          \end{array}
                                                        \right)$ and
observe that $d^\times x= |a|^{-1}dyd^\times a d^\times b dk$ then
the integral above reads
\[
\int_{F^\times}\int_{F^\times}\int_{F}\int_{K}
\widehat{\Phi}\left(\left(
                                                          \begin{array}{cc}
                                                            a & y \\
                                                            0 & b \\
                                                          \end{array}
                                                        \right)k\right) \,\chi(abdet(k)) \,\omega\left(\left(
                                                          \begin{array}{cc}
                                                            a & y \\
                                                            0 & b \\
                                                          \end{array}
                                                        \right)k\right) |ab|^s \,|a|^{-1}dyd^\times a d^\times b dk
\]
By definition we have that $\omega(x)=<\pi(x)v,\tilde{v}>$ with $v$
a normalized spherical vector. That means, as $\pi=\pi(\nu_1,\nu_2)$
that we have $\omega\left(\left(
                                                          \begin{array}{cc}
                                                            a & y \\
                                                            0 & b \\
                                                          \end{array}
                                                        \right)k\right)=
\nu_1(a)\nu_2(b)<v,\tilde{v}>=\nu_1(a)\nu_{2}(b)$. Moreover we note
that by definition
$\widehat{\Phi}(xk)=\widehat{\Phi}(x)\Phi(k^{-1})$. Indeed from the
definition of the Fourier transform after the change of variable
$y\mapsto y\transpose{k}^{-1}$ and noticing that
$tr(\transpose{y}xk)=tr(k\transpose{y}x)$ we have that
\[
\widehat{\Phi}(xk)=\int_{M_2(F)}\Phi(y)\psi(tr(\transpose{y}xk))=\int_{M_2(F)}\Phi(y\transpose{k}^{-1})\psi(tr(\transpose{y}x)).
\]
But now we note that by the very definition of $\Phi$ that
$\Phi(y\transpose{k}^{-1})=\Phi(y)\Phi(\transpose{k}^{-1})=\Phi(y)\Phi(k^{-1})$
which proves our claim. Our next observation is that
$\Phi(k^{-1})=\chi^{-1}(det(k))$ since $\nu_1$ and $\nu_2$ are
unramified characters. These considerations together give us that
\[
Z(s,\widehat{\Phi},\pi,\chi)=vol(K)
\left(\int_{F^\times}\widehat{\Phi}_{\nu_1\chi}(a)\nu_1\chi(a)|a|^{s-1}
d^\times a
\right)\left(\int_{F^\times}\widehat{\Phi}_{\nu_2\chi}(b)\nu_2\chi(b)|b|^s
d^\times b \right)\int_{F} \int_{I_2}\psi(yx)dydx
\]
But we have that $\int_{F} \int_{I_2}\psi(yx)dydx=
vol(I_2)\int_{\mathfrak{d}_F^{-1}I_2^{-1}}dy=vol(I_2)vol(\mathfrak{d}_F^{-1}I_2^{-1})=|\mathfrak{d}_F|$.
\end{proof}

Let $(V,\theta)$ be an $n$ dimensional hermitian space over $K$ so
that for all archimedean places $v \in \Sigma$ we have $\theta_v$ is
positive definite. By Shimura \cite[page 171]{Shimurabook1} there
exists a matrix $\lambda \in K_n^n$ and an element $\kappa \in
K^\times$ such that (i) $\kappa^\rho = -\kappa$, (ii) $i\kappa_v
\theta_v$ has signature $(n,0)$ for all $v \in \Sigma$ and (iii)
$\kappa \,\theta = \lambda^* - \lambda$. We now define the matrix
\[
S:=\left(
     \begin{array}{cc}
       1_n & -\lambda \\
       -1_n & \lambda^* \\
     \end{array}
   \right).
\]
We define the matrix $\omega:=\left(
                                \begin{array}{cc}
                                  \theta & 0 \\
                                  0 & -\theta \\
                                \end{array}
                              \right)$ and write $G^{\omega}$ for
the corresponding unitary group. That is, $G^\omega$ corresponds to
the hermitian space $(2V:=V \oplus V, \theta \oplus (-\theta))$. As
always we write $\eta_{n}$ for the matrix $\left(
                                                                     \begin{array}{cc}
                                                                       0 & -1_n \\
                                                                       1_n & 0 \\
                                                                     \end{array}
                                                                   \right)$.
Then as it is explained in Shimura \cite[page 176]{Shimurabook1}we
have $S^{-1}G^{\omega} S=G^{\eta_n}$ and if we define $P^\omega:=
{\gamma \in G^\omega : U \gamma = U}$ with $U:=\{(v,v) \in 2V, v \in
V \}$ then $S^{-1} P^{\omega} S=P^{\eta_n}$ with $P^{\eta_n}$ the
standard Siegel parabolic $\{x \in G^{\eta_n}: c_x = 0\}$. We write
$\gamma_n : G^\omega \stackrel{\sim}{\rightarrow} G^{\eta_n},\,g
\mapsto S^{-1}\, g\, S$. Now we note that if we define the group
$G_{\theta,\theta}:=  G^\theta \times G^\theta $ then we have a
canonical embedding $G_{\theta,\theta} \hookrightarrow G^\omega$
given by $(g,g')\mapsto diag(g,g')$. Then we remark that
\[
\gamma_n^{-1}(P^{\eta_n}) \cap G_{\theta,\theta} = \{(g,g) | g \in
G^{\theta}\}.
\]

\begin{lem}\label{local p integral} Let $v$ be a place over $p$ and
let $\Phi:=\Phi_v$ be the Bruhat-Schwartz selected above. Then if we
write $f^{(S)}_\Phi$ for the corresponding local section we have
\[
Z_v(s,\pi,\tilde{\pi},f^{(S)}_{\Phi}))=\int_{GL_n(F_v)}
f^{(S)}_{\Phi}(\gamma_n(g,1)) \omega(g) d^\times g=
\]
\[
=\left\{
   \begin{array}{ll}
     \alpha(\chi,\psi,s) \times \frac{L_{\mathfrak{p}}(s-1, \chi\psi)}{e_{\mathfrak{p}}(s-1, \chi\psi)L_{\mathfrak{p}}(2-s, \chi^{-1}\psi^{-1})}, & \hbox{if $n=1$;} \\
     \alpha(\chi,\psi,s) \times \frac{L_{v}(s-1, \nu^{-1}_1\phi_1)}{e_{v}(s-1, \nu^{-1}_1\phi_1)L_{v}(2-s,\nu_1\phi_1^{-1})}\frac{L_{v}(s, \nu^{-1}_2\phi_1)}{e_{v}(s, \nu^{-1}_2\phi_1)L_{v}(1-s, \nu_2\phi_1^{-1})},
 & \hbox{if $n=2$.}
   \end{array}
 \right.
\]
where the notation is as in Theorem \ref{interpolation properties}.
\end{lem}
\begin{proof} In this proof we will write $F$ for $F_v$. Moreover we write $(\chi_1,\chi_2)$ for the pairs $(\chi_{\mathfrak{p}}\psi_{\mathfrak{p}},\chi_{\bar{\mathfrak{p}}}\psi_{\bar{\mathfrak{p}}})$ in case $n=1$ and
$(\phi_1,\phi_2)$ in case $n=2$. By definition we have
$f^{(S)}_{\Phi}(x)= f_{\Phi}(xS^{-1})$ with
\[
f_\Phi(h):=\chi_2(deth)|det(h)|^s
\int_{GL_n(F)}\Phi((0,Z)h)\chi_2\chi_1(detZ) |det(Z)|^{2s} d^\times
Z.
\]
That is the integral $Z_p(s,\pi,\tilde{\pi},f^{(S)}_{\Phi}))$ is
equal to
\[
I(s):=\chi_2^{-1}(det(S))|det(S)|^{-s} \times
\]
\[\int_{GL_n(F)}
\int_{GL_n(F)}\Phi((0,Z)\gamma_n(diag(g,1))S^{-1})
\chi_2\chi_1(det(Z))|det(Z)|^{2s} \chi_2(det(g))|det(g)|^s \omega(g)
d^\times Z d^\times g
\]
By the definition of $\gamma_n$ we have that
\[
\Phi((0,Z)\gamma_n(g,1)S^{-1})=\Phi((0,Z)S^{-1}diag(g,1)).
\]
We have that
\[
S^{-1}=\left(
         \begin{array}{cc}
           \lambda^* & \lambda \\
           1_n & 1_n \\
         \end{array}
       \right) \left(
                 \begin{array}{cc}
                   \kappa^{-1} \theta^{-1} & 0 \\
                   0 & \kappa^{-1} \theta^{-1} \\
                 \end{array}
               \right).
\]
After doing the algebra we obtain
\[
S^{-1}\left(
        \begin{array}{cc}
          g & 0 \\
          0 & 1_n \\
        \end{array}
      \right) = \left(
         \begin{array}{cc}
           \lambda^* & \lambda \\
           1_n & 1_n \\
         \end{array}
       \right)\left(
        \begin{array}{cc}
          \kappa^{-1} \theta^{-1} g & 0 \\
          0 & \kappa^{-1} \theta^{-1} \\
        \end{array}
      \right) =\left(
                 \begin{array}{cc}
                   \lambda^*\kappa^{-1} \theta^{-1} g & \lambda \kappa^{-1} \theta^{-1} \\
                   \kappa^{-1} \theta^{-1} g & \kappa^{-1} \theta^{-1}  \\
                 \end{array}
               \right).
\]
In particular
\[
\Phi\left((0,Z) \left(
                 \begin{array}{cc}
                   \lambda^*\kappa^{-1} \theta^{-1} g & \lambda \kappa^{-1} \theta^{-1} \\
                   \kappa^{-1} \theta^{-1} g & \kappa^{-1} \theta^{-1}  \\
                 \end{array}
               \right)\right)= \Phi(Z \kappa^{-1} \theta^{-1} g, Z\kappa^{-1}
\theta^{-1}).
\]
The integral now reads
\[
I(s) =\chi_2^{-1}(det(S))|det(S)|^{-s} \times
\]
\[\int_{GL_n(F)}
\int_{GL_n(F)} \Phi(Z \kappa^{-1} \theta^{-1} g, Z\kappa^{-1}
\theta^{-1}) \chi_2\chi_1(det(Z))|det(Z)|^{2s}
\chi_2(detg)|det(g)|^s \omega(g) d^\times Z  d^\times g.
\]
We make the change of variables $Z \kappa^{-1} \theta^{-1} \mapsto
Z_2$ and $Z \kappa^{-1} \theta^{-1} g \mapsto Z_1$. That is
$Z_2^{-1} Z_1 = g$ and by the translation invariance (recall that
$GL_n(F)$ is unimodular i.e. left and right Haar measures coincide)
of the measures we have
\[
I(s)=\chi_2^{-1}(det(S))|det(S)|^{-s}\times
\]
\[
\int_{GL_n(F)} \int_{GL_n(F)} \Phi(Z_1,Z_2)\chi_2\chi_1(det(\kappa
\theta))|det(\kappa \theta)|^{2s}
\chi_2\chi_1(det(Z_2))\chi_2(det(Z_2)^{-1})\]
\[
\chi_2(det(Z_1))|det(Z_2)|^{s}|det(Z_1)|^{s}\omega(Z_2^{-1} Z_1)
d^\times Z_1 d^\times Z_2.
\]
As it is explained in \cite[page 36]{GPR} this integral is equal to
\[
I(s)=\chi_2^{-1}(det(S))|det(S)|^{-s}  \chi_2\chi_1(det(\kappa
\theta))|det(\kappa \theta)|^{2s} \times
\]
\[
\int_{GL_n(F)} \int_{GL_n(F)}
\Phi(Z_1,Z_2)\chi_2(det(Z_1))\chi_1(det(Z_2))
|det(Z_2)|^{s}|det(Z_1)|^{s}\omega(Z_2^{-1}) \omega(Z_1) d^\times
Z_1 d^\times Z_2
\]
which in turn is equal to
\[
\chi_2^{-1}(det(S))|det(S)|^{-s}\chi_2\chi_1(det(\kappa
\theta))|det(\kappa \theta)|^{2s} \times
\]
\[
\int_{GL_n(F)} \Phi_1(Z_1)\chi_2(det(Z_1))\omega(Z_1) |det(Z_1)|^s
d^\times Z_1 \int_{GL_n(F)}
\Phi_2(Z_2)\chi_1(det(Z_2))\omega(Z^{-1}_2) |det(Z_2)|^s d^\times
Z_2.
\]

But $\Phi_1 = \tilde{\Phi}_{\nu^{-1}\chi_2^{-1}}$ and $\Phi_2=
\widehat{\Phi}_{\nu^{-1}\chi_1}$. These integrals we have already
computed. Finally we set
\[
\alpha(\chi,\psi,s):= |\mathfrak{d}_F| \times
\chi_2^{-1}(det(S))|det(S)|^{-s}\chi_2\chi_1(det(\kappa
\theta))|det(\kappa \theta)|^{2s}
\]
and $\alpha(\chi,\psi):=\alpha(\chi,\psi,1)$.
\end{proof}

\section{Congruences between Eisenstein Series}

We fix the following setting. We consider a CM field $K$ and a CM
extension $K'$ of degree $p$ with totally real field $F=K^+$ and
$F'=K'^{+}$. We also fix CM types $(K,\Sigma)$ and $(K',\Sigma')$
with $\Sigma'$ be the lift of $\Sigma$. We moreover make the
following assumption: The primes that ramify in $F'/F$ are
unramified in $K/F$. Our aim now is to study the natural embedding
We now wish to study the embedding
\[
U(n,n)_{/F}(F) \hookrightarrow Res_{F'/F}U(n,n)_{/F'}(F).
\]

\textbf{The diagonal embedding; algebraically and analytically:} We
start by first observing that the compatibility of the CM-types
induces an embedding of the corresponding symmetric spaces. That is,
we have
\[
\Delta : \mathbb{H}_F \hookrightarrow \mathbb{H}_{F'},
\,\,\,(z_{\sigma})_{\sigma \in \mathbf{a}} \mapsto
(z_{\sigma'})_{\sigma' \in \mathbf{a}'},
\]
with $z_{\sigma'}:=z_{\sigma}$ for $\sigma'|_{K}=\sigma$. We now
consider the congruence subgroup
$\Gamma_0(\mathfrak{b},\mathfrak{c})$ of $U(n,n)_{/F}$ for an
integral ideal $\mathfrak{c}$ of $\mathfrak{g}$ and a fractional
ideal $\mathfrak{b}$ of $F$. We moreover consider the congruences
subgroup
$\Gamma_0(\mathfrak{b}\mathfrak{g}',\mathfrak{c}\mathfrak{g}')$ of
$Res_{F'/F}U(n,n)_{/F'}$. Then we have that the embedding
$U(n,n)_{/F} \hookrightarrow Res_{F'/F}U(n,n)_{/F'}$ induces an
embedding $\Gamma_0(\mathfrak{b},\mathfrak{c}) \hookrightarrow
\Gamma_0(\mathfrak{b}\mathfrak{g}',\mathfrak{c}\mathfrak{g}')$. We
simplify our notation by setting
$\Gamma:=\Gamma_0(\mathfrak{b},\mathfrak{c})$ and
$\Gamma':=\Gamma_0(\mathfrak{b}\mathfrak{g}',\mathfrak{c}\mathfrak{g}')$.
Now we easily observe that $\Delta$ induces, by pull-back, a map
\[
\Delta^*:M_k(\Gamma') \rightarrow M_{pk}(\Gamma),\,\,\,f \mapsto f
\circ \Delta.
\]
Next we study the effect of this map on the $q$-expansion of the
hermitian modular form. Namely, if we assume that $f \in
M_k(\Gamma')$ has a $q$-expansion of the form
\[
f(z') = \sum_{h' \in L'}c(h') e^n_{\mathbf{a}'}(h'z'),
\]
where we recall that $L'={\mathfrak{d}'}^{-1}T'$, then the following
lemma provides us the $q$-expansion of the form $\Delta^*f$.
\begin{lem}\label{analytic q-expansion of pull-back } For $f \in M_k(\Gamma')$ as above we have
\[
\Delta^*f(z)=\sum_{h \in L} \left(\sum_{h' \in L', Tr_{K'/K}(h')=h}
c(h')\right)e^n_{\mathbf{a}}(hz).
\]
\end{lem}
\begin{proof} We first observe that $Tr_{F'/F} : L' \rightarrow L$,
where $Tr_{F'/F}(h'):=\sum_{\sigma \in Gal(K'/K)} (h')^\sigma$.
Indeed, we recall that $L'=\mathfrak{d'}^{-1}T'$ and
$L=\mathfrak{d}^{-1}T$. To see this, consider an element $d't' \in
L'$ with $d' \in \mathfrak{d'}^{-1}$ and $t' \in T'$. For an element
$y \in S(\mathfrak{r})$ we compute
\[
tr(yTr_{K'/K}(d't')) = tr(y \sum_{\sigma}(d')^\sigma
(t')^\sigma))=\sum_{\sigma}(d'tr(y t'))^\sigma.
\]
But $S(\mathfrak{r}) \subset S(\mathfrak{r}')$ hence we have that
$g' :=tr(yt') \in \mathfrak{g}'$. That is we have shown
$tr(yTr_{K'/K}(d't'))=\sum_{\sigma}(d'g')^\sigma$ and $d'g' \in
{\mathfrak{d}'}^{-1}$. But $Tr_{K'/K}({\mathfrak{d}'}^{-1}) \subset
\mathfrak{d}^{-1}$ hence we have shown that
$tr(S(\mathfrak{r})Tr_{K'/K}(L')) \subseteq \mathfrak{d}^{-1}$ or
equivalently $\mathfrak{d} Tr_{K'/K}(L') \subset T$ which concludes
our claim.

Now we consider what happens to the ${h'}^{th}$ component of $f$
after setting $t':=\Delta(z)$. We have
\[
e^n_{\mathbf{a}'}(h' \Delta(z))= e^{2\pi i \sum_{\sigma' \in
\Sigma'}(\sum_{i,j}h'_{i,j}z_{j,i})^{\sigma'}}=e^{2\pi i
\sum_{\sigma \in \Sigma'}\sum_{i,j}{h'_{i,j}}^{\sigma'}
z_{j,i}^{\sigma'}}=
\]
\[
e^{2\pi i \sum_{\sigma \in \Sigma}\sum_{\sigma' \in
\Sigma',\sigma'_{|K}=\sigma}\sum_{i,j}{h'_{i,j}}^{\sigma'}
z_{j,i}^{\sigma'}}=e^{2\pi i \sum_{i,j} \sum_{\sigma \in
\Sigma}z_{j,i}^{\sigma}\sum_{\sigma' \in
\Sigma',\sigma'_{|K}=\sigma}{h'_{i,j}}^{\sigma'} }=
\]
\[
e^{2\pi i \sum_{i,j} \sum_{\sigma \in
\Sigma}z_{j,i}^{\sigma}(Tr_{K'/K}({h'_{i,j}})^{\sigma}}=e^{2\pi i
\sum_{\sigma \in \Sigma} (\sum_{i,j}
Tr_{K'/K}(h'_{i,j})z_{j,i})^{\sigma}}
=e^n_{\mathbf{a}}(Tr_{K'/K}(h')z).
\]
These calculations allow us to conclude the proof of the lemma.
\end{proof}

It is easily seen that the above considerations can be generalized
to more general congruent subgroups. Namely, for an element $g \in
GL_n(\mathbb{A}_{K,\mathbf{h}})$ we consider groups of the form
\[
\Gamma_g:=G_1 \cap \left(
                      \begin{array}{cc}
                        \hat{g} & 0 \\
                        0 & g\\
                      \end{array}
                    \right)D[\mathfrak{b}^{-1},\mathfrak{b}\mathfrak{c}]\left(
                      \begin{array}{cc}
                        \hat{g} & 0 \\
                        0 & g\\
                      \end{array}
                    \right)^{-1}.
                    \]
We note that for $g=1_n$ we have
$\Gamma_{1_n}=\Gamma_{0}(\mathfrak{b},\mathfrak{c})$. Similarly for
the same $g \in GL_n(\mathbb{A}_{K,f}) \subset
GL_n(\mathbb{A}_{K',f})$ we define
\[
\Gamma'_g:=G'_1 \cap \left(
                      \begin{array}{cc}
                        \hat{g} & 0 \\
                        0 & g\\
                      \end{array}
                    \right)D[\mathfrak{b}^{-1}\mathfrak{g}',\mathfrak{b}\mathfrak{c}\mathfrak{g}']\left(
                      \begin{array}{cc}
                        \hat{g} & 0 \\
                        0 & g\\
                      \end{array}
                    \right)^{-1}.
                    \]

We now observe that the embedding $U(n,n)_{F}(\mathbb{A}_F)
\hookrightarrow Res_{F'/F}U(n,n)_{/F'}(\mathbb{A}_{F})$ induces the
embeddings $D[\mathfrak{b}^{-1},\mathfrak{b}\mathfrak{c}]
\hookrightarrow
D[\mathfrak{b}^{-1}\mathfrak{g}',\mathfrak{b}\mathfrak{c}\mathfrak{g}']$
and $\Gamma_{g} \hookrightarrow \Gamma'_g$. In particular, as
before, we have that the map $\Delta$ induces as before a map
$M_{k}(\Gamma'_g) \rightarrow M_{pk}(\Gamma_g)$.

Our next goal is to understand the above analytic considerations
algebraically. We start by recalling the moduli interpretation of
the Shimura varieties $\Gamma_g \setminus \mathbb{H}$. We fix a
$K$-basis $\{e_{i}\}_{i=1}^{2n}$ of $V$ so that the group $U(n,n)/F$
is represented as
\[
U(n,n)_{/F}(F):=\left\{\alpha \in GL_{2n}(K)\,\, |\,\, \alpha^*
\eta_n \alpha =  \eta_n \right\},
\]
with $\eta_n:=\left(
                \begin{array}{cc}
                  0 & -1_n \\
                  1_n & 0 \\
                \end{array}
              \right)$. We consider the $\mathfrak{g}$-maximal $\mathfrak{r}$-lattice
              $L:=\sum_{i=1}^{n} \mathfrak{r}e_i \oplus \sum_{j=n+1}^{2n}\mathfrak{d}^{-1}_{K/F}e_j \subset V$.
Then we note that for all $\gamma \in G_1 \cap
D[\mathfrak{b}^{-1},\mathfrak{b}]$ with
$\mathfrak{b}:=\mathfrak{d}^{-1}_{K/F} \cap F$ we have $L\gamma =
L$. Moreover, for $g \in G(\mathbb{A}_{F,f})$ and $L_g:= (L
\otimes_{\mathfrak{r}} \hat{\mathfrak{r}})g^{-1} \cap V$ we have
$L_g \gamma = L_g$ for all $\gamma \in G_1 \cap g
D[\mathfrak{b}^{-1},\mathfrak{b}]g^{-1}$. In particular the groups
$\Gamma_g$ above respect the lattices $L_g$. Following now Shimura
\cite[page 26]{Shimurabook2} we recall that the space $\Gamma_g
\setminus \mathbb{H}_F$ parameterizes for every $z \in \mathbb{H}_F$
families of polarized abelian varieties
\[
\mathcal{P}_z = (A_z,\mathcal{C}_z,i_z,\alpha_{\mathfrak{c}}),
\]
where $A_z := (\mathbb{C}^{2n})^{\mathbf{a}}/p_{z}(L_g)$ with $p_z$
defined by
\[
p_z:(K_{\mathbf{a}})^1_{2n} \rightarrow
(\mathbb{C}^{2n})^{\mathbf{a}},\,\,x \mapsto ([z_v \,\, 1_n]\cdot
x_v^*,[\transpose{z}_v\,\, 1_n]\cdot \transpose{x}_v)_{v \in
\mathbf{a}}.
\]
Moreover $\mathcal{C}_z$ is the polarization of $A_z$ defined by the
Riemann form $E_z(p_z(x),p_z(y)):=Tr_{K_{\mathbf{a}}/\mathbb{R}}(x
\eta_n y^*)$. The map $i_z : K \hookrightarrow
End_{\mathbb{Q}}(A_z)$ is by the map $\Psi:K \rightarrow
End((\mathbb{C}^{2n})^\mathbf{a})$ defined for $a\in K$ as
$\Psi(a):=diag[\Psi_v(a)]_{v \in \mathbf{a}}$ and
$\Psi_v(a):=diag[\overline{a_v}1_n,\,a_v1_n]$. Finally the
arithmetic structure $\alpha_{c}$ is induced from the embedding
$\mathfrak{c}^{-1}L \hookrightarrow K^{2n}$. As is explained in
Shimura (loc. cit. page 26) we have that two such data
$\mathcal{P}_z$ and $\mathcal{P}_w$ with $z,w \in \mathbb{H}_F$ are
isomorphic if and only if there exists $\gamma \in \Gamma_g$ so that
$w=\gamma z$.

Now we observe that the diagonal map $\Delta: \mathbb{H}_F
\hookrightarrow \mathbb{H}_{F'}$ introduced above induces a map
$\Gamma_{g} \setminus \mathbb{H}_F \rightarrow \Gamma'_{g} \setminus
\mathbb{H}_{F'}$ by $\mathcal{P}_z \mapsto \mathcal{P}'_{\Delta(z)}$
where the structures for the group $GU(n,n)/F'$ are with respect to
the $\mathfrak{r}'$-lattice $L'_g:= \mathfrak{r}'
\otimes_{\mathfrak{r}} L_g$. We note here the crucial assumption
that the ramification of $F/K$ and $F'/F$ is disjoint. In particular
we have that $\mathfrak{d}_{K'/F'}=\mathfrak{d}_{K/F}\mathfrak{r}'$.

Now we extend the above analytic considerations to their algebraic
counterparts. We consider the scheme $\mathcal{M}(\Gamma_g)/R$ over
some ring $R$, associated to the congruence subgroup $\Gamma_g$,
that represents the functor $S \mapsto
(A,\lambda,i,\alpha_{\mathfrak{c}})/S$ discussed in the
introduction. Then the algebraic counterpart of the map above is a
map $\mathcal{M}(\Gamma_g)/R \rightarrow \mathcal{M}'(\Gamma'_g)/R$
given by $(A,\lambda,i,\alpha_{\mathfrak{c}}) \mapsto
(A,\lambda,i,\alpha_{\mathfrak{c}}) \otimes_{\mathfrak{r}}
\mathfrak{r}'$. When $R=\mathbb{C}$ this map is the previously
defined map. In particular we see that we can define the map
$\Delta^*:M_{k}(\Gamma'_g) \rightarrow M_{pk}(\Gamma_g)$
algebraically by $\Delta^{*}f(\underline{A},\omega):=f(\underline{A}
\otimes_{\mathfrak{r}} \mathfrak{r}',\omega \otimes_{\mathfrak{r}}
\mathfrak{r}')$.

Before going further and providing also the algebraic counterpart of
lemma \ref{analytic q-expansion of pull-back } we remark that if
write $A':=A \otimes_{\mathfrak{r}} \mathfrak{r}'$, the image of the
abelian variety $A$, then this is isogenous to $[K':K]$ many copies
of $A$. Indeed, this follows by writing the $\mathfrak{r}$-module
$\mathfrak{r}'$ as a direct sum $\mathfrak{b} \oplus
\bigoplus_{i=1}^{[K':K]-1}\mathfrak{r}$ for some ideal
$\mathfrak{b}$ of $\mathfrak{r}$.

We now consider the Mumford object associated to the standard
$0$-genus cusp associated to the group $\Gamma_g$. We decompose the
lattice $L=\sum_{i=1}^{n}\mathfrak{r}e_i +
\sum_{j=n+1}^{2n}\mathfrak{d}^{-1}_{K/F}e_j$ to $L^1:=\sum_{i=1}^n
\mathfrak{r} e_i$ and
$L^2:=\sum_{j=n+1}^{2n}\mathfrak{d}^{-1}_{K/F}e_j$. Then we see that
since we consider elements of the form $\left(
                                               \begin{array}{cc}
                                                \hat{g} \ & 0 \\
                                                 0 & g \\
                                               \end{array}
                                             \right)$ for the
components of $GU(n,n)$, we have that $L_g= L_g^1 \oplus L_g^2$ with
$L^1_{g}:=L^1 g^*$ and $L^2:=L^2g^{-1}$. The Mumford data associated
to the standard $0$-cusp of the $g$'s component is then
$(Mum_{L^1_g,L^2_g}(q),\lambda_{can},i_{can},\alpha^{can}_{\mathfrak{c}})$,
defined over the ring $R((q,H^{\vee}))$, where
$H^{\vee}=\frak{d}^{-1}gTg^* \subset S$ with $T:=\{ x \in S\,\,
|\,\,tr(S(\mathfrak{r})x) \subset \mathfrak{g} \}$. We now consider
the ring homomorphism $Tr:R((q',{H'}^{\vee})) \rightarrow
R((q,{H}^{\vee}))$ defined by ${q'}^{h'} \mapsto q^{Tr_{K'/K}(h')}$.
Then it can be shown that
\[
(Mum_{L^1_g,L^2_g}(q),\lambda_{can},i_{can},\alpha^{can}_{\mathfrak{c}})
\otimes_{\mathfrak{r}} \mathfrak{r}' =
\]
\[
(Mum_{{L'}^1_g,{L'}^2_g}(q'),\lambda'_{can},i'_{can},{\alpha'}^{can}_{\mathfrak{c}})
\otimes_{R((q',{H'}^{\vee})),Tr} R((q,{H}^{\vee}))
\]
In particular if $f \in M_k(\Gamma'_g)$ is an algebraic hermitian
form with $q$-expansion given by
\[
f(Mum_{{L'}^1_g,{L'}^2_g}(q'),\lambda'_{can},i'_{can},{\alpha'}^{can}_{\mathfrak{c}},\omega'_{can})=\sum_{h'\in
{H'}^{\vee}}c(h')q^{h'}
\]
then its pull-back form $g:=\Delta^*(f)$ has $q$-expansion
\[
g\left(Mum_{L^1_g,L^2_g}(q),\lambda_{can},i_{can},\alpha^{can}_{\mathfrak{c}},\omega_{can}\right)=\]
\[
f\left((Mum_{L^1_g,L^2_g}(q),\lambda_{can},i_{can},\alpha^{can}_{\mathfrak{c}},\omega_{can})\otimes_{\mathfrak{r}}
\mathfrak{r}'\right)
\]
\[
=f\left((Mum_{{L'}^1_g,{L'}^2_g}(q'),\lambda'_{can},i'_{can},{\alpha'}^{can}_{\mathfrak{c}},\omega'_{can})
\otimes_{R((q',{H'}^{\vee})),Tr} R((q,{H}^{\vee}))\right)
\]
\[
f\left((Mum_{{L'}^1_g,{L'}^2_g}(q'),\lambda'_{can},i'_{can},{\alpha'}^{can}_{\mathfrak{c}},\omega'_{can}\right)
\otimes_{R((q',{H'}^{\vee})),Tr} R((q,{H}^{\vee})).
\]
Again with $R=\mathbb{C}$ we have the algebraic counterpart of lemma
\ref{analytic q-expansion of pull-back }. We summarize this
discussion in the following lemma.
\begin{lem}\label{algebraic q-expansion of pull-back } Let $f\in
M_k(\Gamma'_g,R)$ be an algebraic hermitian modular form defined
over $R$. Then the $q$-expansion of $g:=\Delta^*(f) \in
M_{pk}(\Gamma_g,R)$ is given by
\[
g(q):=\sum_{h} \left(\sum_{h' \in L', Tr_{K'/K}(h')=h}
c(h')\right)q^h,
\]
when the $q$-expansion of $f$ is given by
$f(q)=\sum_{h'}c(h')q^{h'}$.
\end{lem}

For a function $\epsilon:=\sum_{j}c_j\chi_j$ of $G^{ab}_{K'}$ where
$\chi_j$ are characters of the form and $\chi'_j\psi$, where $\psi$
a fixed Hecke character of infinite type $k\Sigma$ and $\chi'_j$
finite order characters we have
\[
E_\beta(m(A),\phi,\epsilon)=Q(\beta,A,k)
\sum_{(\mathfrak{a},S)=1}n_{\mathfrak{a}}(\beta,m(A)) \sum_j c_j
\chi_j(det(A))
\]
\[
\left(\prod_{v \in
\Sigma_p}\chi^{-1}_{1,v,j}(det(B(h_v)))\chi_{2,v,j}(det(A(h_v)))\Phi_{\mu}(\transpose{A(h_v)}\beta
B(h_v)^{-1})\right)\left(\prod_{v \in S^{(p)}}\chi_{v,j}(det
u^{-1})\right)\chi_j^{(S)}(\mathfrak{a}).
\]
Now we assume that $\epsilon^\gamma=\epsilon$ for all $\gamma \in
\Gamma$. Since we assume that $m(A) \in GL_n(\mathbb{A}_K)$ we have
that
\[
\chi_j(det(A))\left(\prod_{v \in
\Sigma_p}\chi^{-1}_{1,v,j}(det(B(h_v)))\chi_{2,v,j}(det(A(h_v)))\Phi_{\mu}(\transpose{A(h_v)}\beta
B(h_v)^{-1})\right)\left(\prod_{v \in S^{(p)}}\chi_{v,j}^{-1}(det
u)\right)=\]
\[
\chi^\gamma_j(det(A))\left(\prod_{v \in
\Sigma_p}\chi^{-\gamma}_{1,v,j}(det(B(h_v)))\chi^{\gamma}_{2,v,j}(det(A(h_v)))\Phi_{\mu}(\transpose{A(h_v)}\beta
B(h_v)^{-1})\right)\left(\prod_{v \in
S^{(p)}}\chi_{v,j}^{-\gamma}(det u)\right)
\]
\textbf{Claim:} For the function on $G_{K'}(\mathfrak{c}p^\infty)
\times Her_n(F)$
\[
\epsilon^{(S)}(\mathfrak{a},\beta):=\sum_j c_j
\chi_j(det(A))\left(\prod_{v \in
\Sigma_p}\chi^{-1}_{1,v,j}(det(B(h_v)))\chi_{2,v,j}(det(A(h_v)))\Phi_{\mu}(\transpose{A(h_v)}\beta
B(h_v)^{-1})\right)\times
\]
\[
\left(\prod_{v \in S^{(p)}}\chi_{v,j}^{-1}(det
u)\right)\chi_j^{(S)}(\mathfrak{a})
\]
we have
\[
\epsilon^{(S)}(\mathfrak{a},\beta)=\epsilon^{(S)}(\mathfrak{a}^\gamma,\beta^\gamma)
\]

Indeed, we notice first that $\epsilon^\gamma=\epsilon$ if an only
if $c_j=c_{\gamma(j)}$ where $c_{\gamma(j)}$ denotes the coefficient
of the character $\chi_j^\gamma=:\chi_{\gamma(j)}$ in the sum
$\sum_j c_j \chi_j$. In particular that means that we may decompose
the locally constant function $\epsilon$ as follows
\[
\epsilon= \sum_{i}c_i \chi_i + \sum_{k}c_k \sum_{\gamma \in
\Gamma}\chi_k^\gamma,
\]
where for the characters $\chi_i$ that appear in the first sum we
have that $\chi_i^\gamma =\chi_i$ for all $\gamma \in \Gamma$. Then
the claim follows from the observation above, the definition of
$\Phi_\mu$ and the fact that $\nu^\gamma=\nu$.

For the applications that we have in mind we need to understand how
the polynomials $g_{\beta,m,v}$ depend on $\beta,m$ and $v$. We
explain this by following closely Shimura's book
\cite{Shimurabook1}. We start with some definitions.

First of all we need to introduce the notion of the denominator of a
matrix as is defined in \cite{Shimurabook1} p18. Let $\mathfrak{r}$
be a principal ideal domain and let $K$ denote its field of
quotients. We assume that $K \neq \mathfrak{r}$. We set
$E=E^{n}=GL_n(\mathfrak{r})$. Given any matrix $X \in K_n^m$ of rank
$r$, there exist $A\in E^m$ and $B \in E^n$ and elements
$e_1,\ldots,e_r \in K^\times$ such that $e_{i+1}/e_i \in
\mathfrak{r}$ for all $i < r$ and
\[
AXB = \left(
        \begin{array}{cc}
          D & 0^r_{n-r} \\
          0_r^{m-r} & 0^{m-r}_{n-r} \\
        \end{array}
      \right),\,\,\,D=diag[e_1,\ldots,e_r]
\]
The ideals $e_i\mathfrak{r}$ are uniquely determined by $X$ and we
call them the elementary divisors of $X$. We call an element $X \in
\mathfrak{r}^{m}_{n}$ primitive if $rank(X) = min(m,n)$ and the
elementary divisors are all equal to $\mathfrak{r}$. Shimura shows
that for any given $x \in K_n^{m}$ there exist $c \in
\mathfrak{r}^{m}_n$ and $d \in \mathfrak{r}^{m}_{m}\cap GL_m(K)$
such that the matrix $[c\,\,d]$ is primitive and $x=d^{-1}c$ and the
integral ideal $\nu_0(x):=det(d)\mathfrak{r}$ is well-defined and
called the denominator ideal of $x$.

Now we fix a local field $F$, a finite extention of $\Q_p$ for a
prime $p$. We write $\mathfrak{g}$ for its ring of integers and
$\mathfrak{p}$ for its maximal ideal. We put
$q:=[\mathfrak{g}:\mathfrak{p}]$. We pick an additive character
$\chi:F \rightarrow S^{^1}$ such that
\[
\mathfrak{g}=\{a\in F : \chi(a\mathfrak{g})=1\}
\]

Following Shimura as in \cite{Shimurabook1} (only for the case that
we will consider) we fix symbols
$K$,$\mathfrak{r}$,$\mathfrak{q}$,$\mathfrak{d}$,$\delta$,$\rho$ and
$\epsilon$ as follows
\begin{enumerate}
\item if $K$ is a quadratic extension of $F$ then $\mathfrak{r}$ is the integral closure of $\mathfrak{g}$ in $K$,$\mathfrak{q}$ its maximal ideal,$\mathfrak{d}$ the different of $K$
relative to $F$,$\delta \in K$ that generates the different,$\rho$
the non-trivial element of $Gal(K/F)$ and $\epsilon=1$,
\item If $K=F \times F$, then $\mathfrak{r}=\mathfrak{d}=\mathfrak{g} \times \mathfrak{g}$,$\mathfrak{q}=\mathfrak{p} \times \mathfrak{p}$,,$\delta=1$,
$\rho$ is the automorphism of $K$ defined as $(x,y)^{\rho}=(y,x)$
and $\epsilon=1$.
\end{enumerate}

We introduce the following notations
\[S=S^n(\epsilon)=\{h \in K_n^n | h^*=\epsilon h
\},\,\,\,S(\mathfrak{a}):=S \cap (\mathfrak{r}\mathfrak{a})_n^n
\]
where $\mathfrak{a}$ is an $\mathfrak{r}$ or $\mathfrak{g}$ ideal.
Also we introduce the set of matrices
\[
T=T^n(\epsilon)=\{ x \in S^n(\epsilon)| tr(S(\mathfrak{r})x) \subset
\mathfrak{g})\}
\]

Now we extend the definition of the denominator of a matrix $x \in
K_n^n$ defined above for the case where $K$ is not a field as
follows. If $K=F \times F$ then for $x =(y,z) \in F_n^n$ we define
$\nu_0(x)= \nu_0(y)e + \nu_0(z)e'$ where $e=(1,0)$ and $e'=(0,1)$.
Then $\nu_0(x)$ is an $\mathfrak{r}$-ideal. Shimura shows in
\cite{Shimurabook1} that in both cases ($K$ being a field or not) we
have that $\nu_0(\sigma) = (\mathfrak{g} \cap
\nu_0(\sigma))\mathfrak{r}$ for $ \sigma \in S$. We then define for
$\sigma \in S$ the quantity $\nu[\sigma]:= N(\mathfrak{g} \cap
\nu_0(\sigma))$. Given an element $\zeta \in T^n$ we consider the
formal Dirichlet series
\[
\alpha_\zeta(s):= \sum_{\sigma \in S/S(\mathfrak{r})}\chi(tr(\zeta
\sigma))\nu[\sigma]^{-s}
\]
For $\sigma \in S$ we define the non-negative integer $k(\sigma)$ by
$\nu[\sigma]=q^{k(\sigma)}$ and introducing the indeterminant $t$ we
consider the formal series
\[
A_{\zeta}(t)=\sum_{\sigma \in S/S(\mathfrak{r})}\chi(tr(\zeta
\sigma))t^{k(\sigma)}
\]
such that $A_{\zeta}(q^{-s})=\alpha_{\zeta}(s)$. As Shimura explains
we have for $\gamma \in GL_n(\mathfrak{r})$ that
\[
A_{(\gamma\zeta \gamma^*)}(t) = A_{\zeta}(t)
\]
hence we may assume that $\zeta$ is equal to $0$ or equal to $diag
[\xi\,\,\, 0]$ for $\xi \in T^r \cap GL_r(K)$ where $r$ the rank of
$\zeta$. The following theorem is proved in \cite[page
104]{Shimurabook1},

\begin{thm} Let $\zeta \in T^n$ and let $r$ be the rank of
$\zeta$. Suppose that $\zeta =0$ or $\zeta = diag[\xi\,\,0]$ with
$\xi \in T^r \cap GL_r(K)$. Then $A_{\zeta}(t)=f_\zeta(t)
g_\zeta(t)$ where $g_\zeta \in \Z[t]$ with $g_\zeta(0)=1$ and
$f_\zeta$ a rational function given as follows:
\[
f_{\zeta}(t)=\frac{\prod_{i=1}^n(1-\tau^{i-1}q^{i-1}t)}{\prod_{i=1}^{n-r}(1-\tau^{n+1}q^{n+i-1}t)}
\]
where
\[
\tau^i:=\left\{
          \begin{array}{ll}
            1, & \hbox{if $i$ is even or $K=F \times F$;} \\
            -1, & \hbox{if $i$ is odd, $\mathfrak{d}=\mathfrak{r}$, and $K \neq F \times F$;} \\
            0, & \hbox{if $i$ is odd and $\mathfrak{d}\neq \mathfrak{r}$.}
          \end{array}
        \right.
\]
\end{thm}

\begin{prop}We consider the polynomial
$g_{\beta,m(A),v}(t) \in \mathbb{Z}[t]$ in the $K'$-setting i.e.
$\beta \in S^n(K')$, $A \in GL_n(\mathbb{A}_{f,K'})$ and $v$ a
finite place of $K'$. Let now $\gamma \in \Gamma = Gal(K'/K)$. Then
we have
$g_{\beta,m(A),v}(t)=g_{\beta^{\gamma},m(A^\gamma),v^{\gamma}}(t)$.
\end{prop}

\begin{proof}We first note that it follows from \cite[page
156]{Shimurabook1} that
\[
g_{\beta,m(A),v}(t)=g_{\zeta}(t)
\]
with $\zeta:=\omega_v A^*_v \beta A_v$ with $\omega_v$ a generator
of $\mathfrak{d}(F'/\mathbb{Q})_v$ and $g_\zeta(t)$ is defined as
above for $K'_v$. We consider the following two cases,
\begin{enumerate}
\item \textit{The element $\gamma$ fixes $v$:} In this case we have to show that
$g_{\zeta^\gamma}(t)=g_{\zeta}(t)$. Since the ranks of $\zeta$ and
$\zeta^\gamma$ are the same we have that
$f_{\zeta}(t)=f_{\zeta^\gamma}(t)$. So in order to conclude our
claim it is enough to show that $A_{\zeta}(t)=A_{\zeta^\gamma}(t)$.
By definition
\[
A_{\zeta^\gamma}(t)=\sum_{\sigma \in
S/S(\mathfrak{r})}\mathbf{e}_v(d^{-1}_{F_v}tr(\zeta^\gamma
\sigma))t^{k(\sigma)}
\]
That implies
\[
A_{\zeta^\gamma}(t)=\sum_{\sigma \in
S/S(\mathfrak{r})}\mathbf{e}_v((d^{-1}_{F_v})^\gamma tr(\zeta
\sigma^{\gamma^{-1}}))t^{k(\sigma)}
\]
But we have $k(\sigma)=k(\sigma^{\gamma^{-1}})$ as
$\nu[\sigma^{\gamma^{-1}}]=N(\mathfrak{g} \cap
\nu_0(\sigma^{-1}))=N(\mathfrak{g} \cap \nu_0(\sigma))=\nu[\sigma]$.
That means
\[
A_{\zeta^\gamma}(t)=\sum_{\sigma \in
S/S(\mathfrak{r})}\mathbf{e}_v((d^{-1}_{F_v})^\gamma tr(\zeta
\sigma^{\gamma^{-1}}))t^{k(\sigma^{\gamma^{-1}})}=A_{\zeta}(t)
\]
which allows us to conclude the proof in this case as $\gamma^{-1}$
permutes the set $S/S(\mathfrak{r})$ and $A_{\zeta}(t)$ is
independent of the additive character $\textbf{e}_v$ picked (i.e.
makes no difference whether we pick $\textbf{e}_v(d^{-1}_F \cdot)$
or $\textbf{e}_v((d^{-1}_{F_v})^\gamma \cdot)$.

\item \textit{The element $\gamma$ does not fix $v$:} We fix an identification of $K'_v$ and $K'_{v^\gamma}$ and write $x^\gamma$ for the image in
$K'_{v^\gamma}$ of an element $x\in K'_v$ with respect to this
identification. We write
\[
g_{\beta,m(A),v}(t)=g_{\zeta_v,v}(t)
\]
with $\zeta_v:=\omega_v A^*_v \beta A_v$ with $\omega_v$ a generator
of $\mathfrak{d}(F'/\mathbb{Q})_v$ and $g_{\zeta_v,v}(t)$ is defined
for $K'_v$. Similarly we have
\[
g_{\beta^\gamma,m(A^\gamma),v^\gamma}(t)=g_{\zeta_{v^\gamma,v^\gamma}}(t)
\]
with $\zeta_{v^\gamma}:=\omega_{v^\gamma} {A_v^\gamma}^*
\beta^\gamma A^\gamma_v$ with $\omega_{v^\gamma}$ a generator of
$\mathfrak{d}(F'/\mathbb{Q})_{v^\gamma}$ and
$g_{\zeta_{v^\gamma,v^\gamma}}(t)$ is defined for $K'_{v^\gamma}.$
We need to show that
\[
g_{\zeta_v,v}(t)=g_{\zeta_{v^\gamma,v^\gamma}}(t)
\]
But the rank of $\zeta_v$ is equal to the rank of $\zeta_{v^\gamma}$
and hence $f_{\zeta_v,v}(t)=f_{\zeta_{v^\gamma},v^\gamma}(t)$. So it
is enough to show that $A_{\zeta_v}(t)=A_{\zeta_{v^\gamma}}(t)$,
which follows from the identification of $K_v'$ with
$K'_{v^\gamma}$.
\end{enumerate}
\end{proof}

Now we can prove the following lemma.

\begin{prop}\label{coefficients2}Let
$m=m(A)$ with $A$ in $GL_n(\mathbb{A}^{(p)}_{K'}) \times GL_n(O_{K'}
\otimes \Z_p)$ be an element in the Levi component of $P$. Then for
all $\gamma \in \Gamma$ we have,
\[
n_{\mathfrak{a}^\gamma}(\beta^\gamma,m(A^\gamma))=
n_{\mathfrak{a}}(\beta,m(A))
\]
\end{prop}

\begin{proof}
Let as write $\mathfrak{q}$ for the prime ideal of $K'$ that
corresponds to the place $v$ of $K'$. Then we note that
$n_{\mathfrak{q}^m}(\beta,m(A))$ is the $m$-th power coefficient of
the polynomial $g_{\zeta,v}(t)$ with $\zeta:=\omega_v A^*_v \beta
A_v$ with $\omega_v$ a generator of $\mathfrak{d}(F'/\mathbb{Q})_v$.
Moreover from its very definition we have that
$n_{\mathfrak{a}}(\beta,m(A))=\prod_{j}n_{\mathfrak{q}_j^{m_j}}(\beta,m(A))$
for $\mathfrak{a}=\prod_j\mathfrak{q}_j^{m_j}$. But that means that
we need to prove the statement for $\mathfrak{a}$ a powers of a
prime ideal $\mathfrak{q}$, that is to show
$n_{\mathfrak{q}^m}(\beta,m(A))=n_{{\mathfrak{q}^\gamma}^m}(\beta^\gamma,m(A^\gamma))$.
But this follows directly from the previous proposition.
\end{proof}

We need in addition to understand how the coefficients of the
polynomials $g_{\beta,m(A),v}$ behave with respect to $K'$ and $K$.
We have the following proposition

\begin{prop}\label{coefficients} Let $\beta \in S_K$ be positive definite and $A \in GL_n(\mathbb{A}_K)$.
Then we have the congruences
\[
n_{\mathfrak{q}^j}(\beta,m(A)) \equiv
n'_{{\mathfrak{q}'}^j}(\beta,m(A)) \mod{p}
\]
where $\mathfrak{q}':=\mathfrak{q}O_{F'}$ for every prime ideal
$\mathfrak{q}$ of $F$ that is not in $S$.
\end{prop}
\begin{proof}We consider the case where $\mathfrak{q}$ splits in $F'$ and where
it inerts.

\textbf{$\mathfrak{q}$ splits in $F'$:} We start with the splitting
case. We write $\mathfrak{q}'=\prod_{i=1}^p\mathfrak{q}_i$ for the
ideals above $\mathfrak{q}$ and $v'_i$ for the corresponding places.
Then by the considerations above (as $\beta,A$ are coming from $K$)
we have that $g_{\beta,m(A),v'_i}(t)=g_{\beta,m(A),v}(t)$. Hence in
particular we conclude that $n_{\mathfrak{q}^j}(\beta,m(A)) =
n_{\mathfrak{q}_i^j}(\beta,m(A))$ for all $i$. But then
\[
n_{{\mathfrak{q}'}^j}(\beta,m(A))=\prod_{i}n_{\mathfrak{q}_i^j}(\beta,m(A))=n_{\mathfrak{q}_i^j}(\beta,m(A))^p
\]
\[
=n_{\mathfrak{q}^j}(\beta,m(A))^p\equiv
n_{\mathfrak{q}^j}(\beta,m(A)) \mod{p}
\]
Hence we conclude the congruences in this case.

\textbf{$\mathfrak{q}$ inerts in $F'$:} As before we write $v$ for
the place of $F$ that corresponds to $\mathfrak{q}$ and $v'$ for the
one that corresponds to $\mathfrak{q}'$, Moreover only for this
proof we set $F:=F_v$ and $F':=F'_{v'}$ and hence $[F':F]=p$ and it
is an unramified extension, since we assume that $v$ is not a bad
place. We first show that
\[
f_\zeta(t) \equiv f'_{\zeta}(t) \mod{p},
\]
for $\zeta \in S_K \subset S_{K'}$ and of full rank $n$. We note
that for the cases that we consider this is always the case as
$\beta$ is always a positive definite hermitian matrix. Indeed in
this case we have that
\[
f'_{\zeta}(t)=\prod_{i=1}^n(1-\tau^{i-1}{q'}^{i-1}t)=\prod_{i=1}^n(1-\tau^{i-1}{q}^{p(i-1)}t)\equiv
\prod_{i=1}^n(1-\tau^{i-1}{q}^{i-1}t)=f_\zeta(t) \mod{p}
\]
as $q'=q^p$. That is $\tilde{f}_\zeta(t)=\tilde{f}'_\zeta(t) \in
\mathbb{F}_p[t]$ where tilde indicates reduction modulo $p$. Now we
claim that also
\[
A_\zeta(t)\equiv A'_\zeta(t) \mod{p}
\]
We note that in the case that we consider with $\zeta$ of full rank
$n$ these are polynomials in $\mathbb{Z}[t]$. Recalling the
definitions we have
\[
A'_{\zeta}(t)=\sum_{\sigma' \in S'/S'(\mathfrak{r}')}
\mathbf{e}_{v'}(d^{-1}_{F'}tr(\zeta\sigma'))t^{k[\sigma']}
\]
But then as $\zeta \in S_{K}$ and since we may pick $d_{F'}=d_{F}
\in F \subset F'$ we have that
\[
\mathbf{e}_{v'}(d^{-1}_{F'}tr(\zeta\sigma'))=\mathbf{e}_{v'}(d^{-1}_{F'}\sum_{i,j}(\zeta_{ij}\sigma'_{ji}))=\mathbf{e}_{v}(d^{-1}_{F}\sum_{i,j}(\zeta_{ij}Tr_{K'/K}(\sigma'_{ji})))=
\mathbf{e}_{v}(d^{-1}_F tr(\zeta Tr_{K'/K}(\sigma')))
\]
Hence we have
\[
A'_{\zeta}(t)=\sum_{\sigma' \in S'/S'(\mathfrak{r}')}
\mathbf{e}_{v}(d^{-1}_F tr(\zeta Tr_{K'/K}(\sigma')))t^{k[\sigma']}
\]
But then since $Tr_{K'/K}(\sigma') = Tr_{K'/K}({\sigma'}^\gamma)$
and $k[\sigma']=k[{\sigma'}^\gamma]$ for $\gamma \in Gal(K'/K)$ we
have that
\[
A'_{\zeta}(t) \equiv \sum_{\sigma \in S/S(\mathfrak{r})}
\mathbf{e}_{v}(d^{-1}_F tr(\zeta Tr_{K'/K}(\sigma)))t^{k[\sigma]}
\mod{p}
\]
after collecting the $\gamma$ orbits of order $p$. The last sum is
equal now to
\[
\sum_{\sigma \in S/S(\mathfrak{r})} \mathbf{e}_{v}(p d^{-1}_F
tr(\zeta \sigma))t^{k[\sigma]}=A_{\zeta}(t)
\]
where the last equality follows from the fact that $p$ is a unit in
$\mathfrak{r}$ (recall that we consider places not in $S$ and $p$ is
in $S$) and moreover the $A_{\zeta}(t)$ is independent of the
character $\mathbf{e}_v$ used \cite[page 104]{Shimurabook1}. That is
we conclude that $\tilde{A}'_\zeta(t)=\tilde{A}_\zeta(t)$ as
polynomials in $\mathbb{F}_p[t]$. Hence we obtain also that
\[
g_{\zeta}(t)\equiv g'_{\zeta}(t) \mod{p}
\]
which concludes the proposition also in that case.
\end{proof}

\begin{thm}[Congruences of Eisenstein series of $U(n,n)$]Let
$m=m(A)$ with $A$ in $GL_n(\mathbb{A}^{(p)}_K) \times GL_n(O_K
\otimes \Z_p)$ be an element in the Levi component of $P$. Let
$\epsilon$ be a locally constant $\mathbb{Z}_p$-valued function on
$G_{K'}(\mathfrak{c}p^\infty)$ with $\epsilon^\gamma=\epsilon$ for
all $\gamma \in \Gamma$. Let also $\psi$ be a Hecke character of $K$
of infinite type $k\Sigma$ and let $\psi':=\psi \circ N_{K'/K}$.
Then we have the congruences of Eisenstein series:
\[
Res^{K'}_K(E'(z,m(A),\epsilon\psi',\nu')) \equiv
Frob_p(E(z,m(A),\epsilon \circ ver \psi^p,\nu^p)) \mod{p}
\]
\end{thm}

\begin{proof} If we write the Fourier expansion of $E_{k}'(z,m(A),\psi,\nu')$
as
\[
E(z,m(A),\epsilon\psi',\nu')=\sum_{\beta' \in
S_+'(A)}E'_{k,\beta'}(m(A),\epsilon\psi',\nu')q^{\beta'}
\]
then we have seen that
\[
Res^{K'}_K(E'(z,m(A),\epsilon\psi',\nu'))=\sum_{\beta \in
S(A)}\left(\sum_{Tr_{K'/K}(\beta')=\beta}E'_{\beta'}(m(A),\epsilon\psi',\nu')\right)q^\beta
\]
The group $\Gamma =Gal(K'/K)$ operates on the inner sum. In
particular we recall that if the write the function $\epsilon\psi'$
as a finite sum of characters $\epsilon\psi'=\sum_jc_j\chi_j$ then
\[
E'_{\beta'}(m(A),\epsilon\psi',\nu')=Q(\beta',A,k,\nu')
\sum_{(\mathfrak{a},S)=1}n'_{\mathfrak{a}}(\beta',m(A))\epsilon{\psi'}^{(S)}(\mathfrak{a},\beta')
\] with
\[\epsilon{\psi'}^{(S)}(\mathfrak{a},\beta')= \sum_j c_j
\chi_j(det(A))\left(\prod_{v \in
\Sigma_p}\chi_{1,v,j}(det(B(h_v)))\chi^{-1}_{2,v,j}(det(A(h_v)))\Phi_{\mu}(\transpose{A(h_v)}\beta
B(h_v)^{-1})\right)\times
\]
\[
\left(\prod_{v \in S^{(p)}}\chi_{v,j}^{-1}(det
u)\right)\chi_j^{(S)}(\mathfrak{a})
\]
The group $\Gamma$ operates on the pairs $(\beta',\mathfrak{a})$.
From Proposition \ref{coefficients} we have that
$n'_{\mathfrak{a}}(\beta',m(A))=n'_{\mathfrak{a}^\gamma}({\beta'}^\gamma,m(A))$.
In particular since we assume that $\psi^\gamma=\psi$ we have seen
that it implies that
$\epsilon{\psi'}^{(S)}(\mathfrak{a}^\gamma,{\beta'}^\gamma)=\epsilon{\psi'}^{(S)}(\mathfrak{a},\beta')$
and as it easily seen that
$Q({\beta'}^\gamma,A,k,\nu')=Q({\beta'},A,k,\nu')$ we have that if
the pair $(\beta',\mathfrak{a})$ is not fixed by $\gamma \in \Gamma$
then
\[
\sum_{\gamma \in
\Gamma}Q({\beta'}^\gamma,A,k,\nu')n'_{\mathfrak{a}^\gamma}({\beta'}^\gamma,m(A))\epsilon{\psi'}^{(S)}(\mathfrak{a}^\gamma,{\beta'}^\gamma)\equiv
0 \mod{p}
\]
In particular that means that
\[
Res^{K'}_K(E'(z,m(A),\epsilon\psi',\nu'))\equiv \sum_{\beta \in
S_+(A)}\left(Q(\beta,A,k,\nu') \sum_{(\mathfrak{a},S)=1, \mathfrak
{a} \subset
K}n'_{\mathfrak{a}}(\beta,m(A))\epsilon{\psi'}^{(S)}(\mathfrak{a},\beta)\right)q^{p\beta}
\mod{p}
\]
On the other hand we have
\[
Frob_p(E(z,m(A),\epsilon \circ ver \psi^p,\nu^p))=\sum_{\beta \in
S_+(A)}E_{pk,\beta}(m(A),\epsilon \circ ver \psi^p,\nu^p)q^{p\beta}
\]
Hence to conclude the congruences we have to show that
\[
Q(\beta,A,k,\nu') \sum_{(\mathfrak{a},S)=1, \mathfrak {a} \subset
K}n'_{\mathfrak{a}}(\beta,m(A))\epsilon{\psi'}^{(S)}(\mathfrak{a},\beta)
\equiv E_{\beta}(m(A),\epsilon \circ ver \psi^p,\nu^p) \mod{p}
\]
We recall that (note that $\psi^p=\psi' \circ ver$),
\[
E_{\beta}(m(A),\epsilon \circ ver \psi^p,\nu^p)=Q(\beta,A,pk,\nu^p)
\sum_{(\mathfrak{a},S)=1}n_{\mathfrak{a}}(\beta,m(A))\epsilon{\psi'}^{(S)}(
ver(\mathfrak{a}),\beta)
\]
But $\epsilon {\psi'}^{(S)}(ver(\mathfrak{a}),\beta) = \epsilon
{\psi'}^{(S)}(\mathfrak{a}O_{K'},\beta)$ and by Proposition
\ref{coefficients2} we have that
$n_{\mathfrak{a}}(\beta,m(A))=n'_{\mathfrak{a}}(\beta,m(A))$ for
$\mathfrak{a}$ an ideal of $K$. Finally we observe that
\[
Q(\beta,A,pk,\nu^p) \equiv Q(\beta,A,k,\nu') \mod{p}
\]
which allows us to conclude the proof of the theorem.
\end{proof}

\begin{cor} With the assumptions and notations as in the theorem
above we have for every $a \in G^\theta(\mathbb{A}_{F,\mathbf{h}})$
that
\[
Res^{K'}_K({E'}^{(a')}(z,m(A),\epsilon\psi',\nu')) \equiv
Frob_p(E^{(a)}(z,m(A),\epsilon \circ ver \psi^p,\nu^p)) \mod{p}
\]
where $a':=\imath(a)$ under the natural embedding
$\imath:G^{\theta}(\mathbb{A}_{F,\mathbf{h}}) \hookrightarrow
G^{\theta}(\mathbb{A}_{F',\mathbf{h}})$.
\end{cor}
\begin{proof}The corollary follows directly from the theorem and the
definition of the twisted Eisenstein series.
\end{proof}

\section{The Theory of Complex Multiplication.}

\subsection{The formalism of CM points for unitary groups
and the reciprocity law:}

We start by recalling the notion of CM points on the symmetric space
associated to the unitary group $G:=U(n,n)/F$. We will follow the
books of Shimura \cite{Shimurabook2,ShimuraCM}. Let us write
$r:=2n$. We consider the CM algebra $Y:=K_1 \oplus \ldots \oplus
K_t$ with CM fields $K_i$ such that $K \subseteq K_i$ (later we just
pick $K_i=K)$. Let us denote by $F_i$ the maximal real subfield of
$K_i$ and by $\rho$ the automorphism of $Y$ which induces the
non-trivial element of $Gal(K_i/F_i)$ for every $i$. Let us assume
that we can find a $K$-linear ring injection $h:Y \rightarrow K_r^r$
such that,
\[
h(a^\rho)= \eta_n h(a)^\star \eta_n^{-1},\,\,\,a\in Y,
\]
where we recall $\eta_n=\left(
                          \begin{array}{cc}
                            0 & -1_n \\
                            1_n & 0 \\
                          \end{array}
                        \right)$. We put $Y^u=\{a \in Y | aa^\rho =1\}$. Then
we have $h(Y^u) \subset G(F)$. But $Y^u$ is contained in a compact
subgroup of $(Y \otimes_\Q \mathbb{R})^\times$ hence the projection
of $h(Y^u)$ to $G_\textbf{a}$ is contained in a compact subgroup of
$G_\textbf{a}$ hence $h(Y^u)$ has a common fixed point in
$\mathbb{H}_n$, and it can be shown that actually there is a unique
one. A point of $\mathbb{H}_n$ obtained in this way is called a CM
point. The case that we are mostly interested in is when $Y = K
\oplus \ldots \oplus K$, $r$ copies of $K$. The CM points obtained
from this CM algebra correspond to an abelian varieties with
multiplication by $Y$ and of dimension exactly $[Y:\Q]$. We note
also here that if $(A,\imath)$ is an abelian variety $A$ with
multiplication by $Y$ i.e. $i : Y \hookrightarrow End(A)_\Q$ and
$2dim A = [Y:\Q]= r [K : \Q]$, then $A$ is isogenous to a product
$A_1 \times \ldots \times A_r$ with $\imath_i : K \hookrightarrow
End(A_i)_\Q$ and $[K:\Q]= 2 dim A_i$.

\textbf{Shimura's Reciprocity Law (for CM algebras):} We consider
the CM-algebra  $Y=K_1 \oplus \cdots \oplus K_t$, where the $K_i$'s
are CM fields. We consider an abelian variety $(A,\lambda)$ with CM
by $Y$. As it is explained in Shimura \cite[page 129]{ShimuraCM} we
have that $A$ is isogenous to $A_1 \times \ldots \times A_t$ where
$A_i$ is an abelian variety with CM by $K_i$ and $2dim(A_i)=[K_i :
\mathbb{Q}]$. Let us write the type of the $A_i$ variety as
$(K_i,\Sigma_i)$. Then we have that the type $\Psi$ of $A$ is the
direct sum of the $\Phi_i$'s in the way explained in Shimura (loc.
cit.). Let $(K^*_i,\Phi^*_i)$ be the reflex field of $(K_i,\Phi_i)$
and let $K^*$ be the composite of the $K_i^*$'s. As is explained in
Shimura (loc. cit.) we have a map
\[
g:  (K^*)^\times \rightarrow Y^\times,
\]
which extends to a map
\[
g:(K^*)_{\mathbb{A}}^\times \rightarrow Y_{\mathbb{A}}^\times
\]
In Shimura \cite[page 125 and page 130]{ShimuraCM} the following
theorem is proved.

\begin{thm} Let $\mathcal{P}=(A,\lambda,\imath)$ be a structure of
type $\Omega=(Y,\Psi,\mathfrak{a},\zeta)$ and let $K^*$ as above.
Further let $\sigma$ be an element of $Aut(\mathbb{C}/K^*)$, and $s$
an element of $(K^*)_{\mathbb{A}}^\times$ such that
$\sigma_{|_{K^*_{ab}}}=[s,K^*]$ Then there exists an exact sequence
\[
0 \rightarrow q(g(s)^{-1}\mathfrak{a}) \rightarrow \mathbb{C}^n
\stackrel{\xi'}{\rightarrow} A^{\sigma}\rightarrow 0
\]
with the following properties
\begin{enumerate}
\item $\mathcal{P}^\sigma$ is of type
$(Y,\Psi,g(s)^{-1}\mathfrak{a},\zeta')$ with
$\zeta'=N(s\mathfrak{r})\zeta$ with respect to $\xi'$, where
$\mathfrak{r}$ is the maximal order of $K^*$.
\item $\xi(q(w))^\sigma = \xi' (q(g(s)^{-1}w))$, where $\xi$ is such that
\[
0 \rightarrow q(\mathfrak{a}) \rightarrow \mathbb{C}^n
\stackrel{\xi}{\rightarrow} A \rightarrow 0
\].
\end{enumerate}
\end{thm}

\textbf{Using the Theory of Complex Multiplication:} Now we explain
how we can use the theory of complex multiplication to understand
how Frobenious operates on values of Eisenstein series at CM points.
In this section we prove the following proposition, which is just a
reformulation of what is done in \cite{Katz1} (page 539) in the case
of quadratic imaginary fields and the group $GL_2$. This proposition
has also been proved by Ellen Eischen in \cite[section
5.2]{Eischen}.

We first recall some of the assumptions that we have made. Recall
that we consider a CM type $(K,\Sigma)$ such that (i)  $p$ is
unramified in $F$, where $F$ the totally real field $K^+$, (ii) the
ordinary condition is that all primes above $p$ in $F$ are split in
$K$ and (iii) that for $\mathfrak{p}$ in $K^*$ above $p$ we have
that $N\mathfrak{p}=p$. We write $\Phi_{\mathfrak{p}}$ for the
Frobenious element in $Gal(K_{ab}^*/K^*)$ corresponding to the prime
ideal $\mathfrak{p}$ of $K^*$ through Artin's reciprocity law. We

\begin{prop}(Reciprocity law on CM points)\label{reciprocityP} Consider the $\mathfrak{g}$-lattice $\mathfrak{U}$ of the CM algebra
$Y$ and the tuple
$(\underline{X(\mathfrak{U})},\omega(\mathfrak{U}))$ defined over
$K_{ab}^*$. Let $E$ be a hermitian form defined over
$\mathbb{Q}_{ab}$. Then we have the reciprocity law:
\begin{equation}
Frob_{p}(E)(\underline{X(\mathfrak{U})},\omega(\mathfrak{U})) =
(E^{\Phi_{\mathfrak{p}}^{-1}}(\underline{X(\mathfrak{U})},\omega(\mathfrak{U})))^{\Phi_{\mathfrak{p}}}.
\end{equation}
In particular if $E$ is a hermitian form defined over $K^*$ then we
have
\begin{equation}
Frob_{p}(E)(\underline{X(\mathfrak{U})},\omega(\mathfrak{U})) =
(E(\underline{X(\mathfrak{U})},\omega(\mathfrak{U})))^{\Phi_{\mathfrak{p}}}.
\end{equation}
\end{prop}
\begin{proof}
From the compatibility of hermitian modular forms with base
extensions we have that
\begin{equation}
(E(\underline{X(\mathfrak{U})},\omega(\mathfrak{U})))^{\Phi_{\mathfrak{p}}}=
E^{\Phi_{\mathfrak{p}}}((\underline{X(\mathfrak{U})},\omega(\mathfrak{U}))
\otimes_{K^*_{ab},\Phi_{\mathfrak{p}}} K^*_{ab})
\end{equation}
where the tensor product is with respect to the map
$\Phi_{\mathfrak{p}}: K^*_{ab} \rightarrow K^*_{ab}$, i.e. the base
change of the tuple
$(\underline{X(\mathfrak{U})},\omega(\mathfrak{U}))$ with respect to
the Frobenious map. But then, from the theory of complex
multiplication explained above and our assumptions on $\mathfrak{p}$
we have that
\begin{equation}
(\underline{X(\mathfrak{U})},\omega(\mathfrak{U}))
\otimes_{K^*_{ab},\Phi_{\mathfrak{p}}} K^*_{ab} \cong
(\underline{X'(\mathfrak{U})},\omega'(\mathfrak{U})),
\end{equation}
where $(\underline{X'(\mathfrak{U})},\omega'(\mathfrak{U}))$ is the
$g(\mathfrak{p})$-transform of
$(\underline{X(\mathfrak{U})},\omega(\mathfrak{U}))$. We notice that
$X'(\mathfrak{U})=
X(g(\mathfrak{p})^{-1}\mathfrak{U})=X(\mathfrak{U})/H_{can}$, where
$H_{can}:=i(M^0 \otimes \mathbf{\mu}_{p})$ and $i$ the $p$-numerical
structure (see also \cite[page 45]{Eischen} or \cite[page
222]{Katz2}). Moreover, we have that the Mumford object
$(\underline{Mum(q)},\omega'_{can})$ is obtained from
$(\underline{Mum(q)},\omega_{can})$ by the map $q \mapsto q^p$ (see
\cite[pages 46-47]{Eischen}, from which we conclude the proposition.
\end{proof}

\subsection{The relation of CM points with respect to the diagonal
map:} For this section we write $G$ for the unitary group
$U^\theta$, with $\theta$ a definite hermitian form and $G'$ for
$Res_{F'/F}U^\theta/F'$. Then for an integral ideal $\mathfrak{c}$
of $\mathfrak{g}$ we have defined the open compact subgroup
$D(\mathfrak{c}) \subset G(\mathbb{A}_{F,f})$ and the open compact
subgroup $D(\mathfrak{c}') \subset
G(\mathbb{A}_{F',f}=G'(\mathbb{A}_{F,f})$, where
$\mathfrak{c}':=\mathfrak{c}\mathfrak{g}$. Then we have defined the
finite sets $\mathcal{B}_K:=G(F) \setminus G(\mathbb{A}_{F,f})/
D(\mathfrak{c})$ and $\mathcal{B}_{K'}:=G(F') \setminus
G(\mathbb{A}_{F',f})/ D(\mathfrak{c}')$. We write $\Gamma$ for
$Gal(F'/F)$ and consider its action on $G(\mathbb{A}_{F',f})$. Then
we note that $D(\mathfrak{c}')^\Gamma = D(\mathfrak{c})$ and also
that this action induces an action of $\Gamma$ on
$\mathcal{B}_{K'}$. Moreover the natural inclusion $F\hookrightarrow
F'$ induces a map $\imath: \mathcal{B}_K \rightarrow
\mathcal{B}_{K'}$. We now examine the conditions under which the map
$\imath: \mathcal{B}_K \rightarrow \mathcal{B}_{K'}^\Gamma$ is a
bijection. The proof of the following proposition was inspired from
a similar proof of Hida in \cite{Hida3}.

\begin{prop}\label{bijection of CM points}Assume that there exist a prime ideal $\mathfrak{q}$ of $F$ such
that
\begin{enumerate}
\item If we write $q:=\mathfrak{q} \cap \mathbb{Q}$ for the prime
below $\mathfrak{q}$ in $q$ and $e$ for the ramification index of
$\mathfrak{q}$ over $q$, then $\mathfrak{q}^\nu | \mathfrak{c}$ for
some $\nu \geq (e+1)/(q-1)$,
\item The extension $F'/F$ is not ramified at $\mathfrak{q}$.
\end{enumerate}
Then the canonical map $\imath: \mathcal{B}_K \rightarrow
\mathcal{B}_{K'}^\Gamma$ is a bijection.
\end{prop}
\begin{proof} We recall that the sets $\mathcal{B}_K$ and
$\mathcal{B}_{K'}$ are defined as $\mathcal{B}_K=G^\theta(F)
\setminus G^\theta(\mathbb{A}_{F,f})/D(\mathfrak{c})$ and similarly
$\mathcal{B}_{K'}:=G^\theta(F') \setminus
G^\theta(\mathbb{A}_{F',f})/D(\mathfrak{c}')$, with
$\mathfrak{c}'=\mathfrak{c}\mathfrak{r}'$. The conditions above
imply (see \cite[page 201, remark 2]{Shimurabook1} that the groups
$D(\mathfrak{c})$ are sufficiently small, that is we have for every
$\alpha \in G^\theta(\mathbb{A}_{F,f})$ (resp. $\beta \in
G^\theta(\mathbb{A}_{F',f})$) that $G^\theta(F) \cap \alpha
D(\mathfrak{c}) \alpha^{-1}=\{1\}$ (resp. $G^\theta(F') \cap \beta
D(\mathfrak{c}') \beta^{-1}=\{1\}$). Now we are ready to prove the
injectivity.

Assume that $\imath(x)=\imath(x')$ for $x, x' \in \mathcal{B}_K$.
Then there exists $\gamma \in G^\theta(F')$ and $d \in
D(\mathfrak{c}')$ such that $g x = x' d$. This implies, that for all
$\gamma \in \Gamma(F'/F)$ that $g^\gamma x = x' d^\gamma$. In
particular we conclude that $g^{\gamma-1} = x D(\mathfrak{c}')x^{-1}
\cap G(F') = \{1\}$. Hence, we obtain that $g \in G(F)$ and
similarly that $d \in D(\mathfrak{c}') \cap G(\mathbb{A}_{F,f}) =
D(\mathfrak{c})$. Hence $x=x'$ in $\mathcal{B}_K$.

Next we prove the surjectivity of the map $\imath$. Let $x \in
G(\mathbb{A}_{F',f})/D(\mathfrak{c}')$. Then for $\gamma \in \Gamma$
we define $g_\gamma \in G(F')$ by $g_\gamma x = x^\gamma$. Then for
$\gamma_1,\gamma_2 \in \Gamma$ we have
\[
g_{\gamma_1 \gamma_2}x = x^{\gamma_1
\gamma_2}=(x^\gamma_1)^{\gamma_2}=(g_{\gamma_1}
x)^{\gamma_2}=g^{\gamma_2}_{\gamma_1} g_{\gamma_2} x.
\]
Under the conditions of the lemma we have that the stabilizer of
$G(\mathbb{A}_{F',f})/D(\mathfrak{c}')$ in $G(F')$ is trivial. That
is we have that
\[
g_{\gamma_1 \gamma_2}=g^{\gamma_2}_{\gamma_1} g_{\gamma_2}.
\]
That is $\gamma \mapsto g_{\gamma}$ gives an element in
$H^1(\Gamma,G(F'))$. As we will show in the next proposition, we
have that $H^1(\Gamma,G(F'))=\{1\}$, i.e. it is trivial. Granted
this, we then can find a $b \in G(F')$ so that $g_\gamma =
b^\gamma/b$. That means, $b^\gamma b^{-1}x = x^\gamma$ and hence
$b^{-1}x \in G(\mathbb{A}_{F,f})$. This in turn implies the
surjectivity of the map $\imath$.
\end{proof}

\begin{prop}\label{Hilbert 90} Let $\Delta:=Gal(F'/F)$ be the Galois group of a
totally  real field extension and assume that $(2,|\Delta|)=1$.
Consider $G$, a unitary group over $F$ (with CM field K), and write
$G'$ for the base changed to $F'$ unitary group. Then the first
non-abelian cohomology group is trivial, that is
\[
H^1(\Delta,G'(F))= H^1(\Delta,G(F')=1
\]
\end{prop}

Before we start with the proof of the above proposition we recall
the following Hasse principle for unitary groups. We introduce some
notation first. Let $K/F$ be a quadratic extension of $p$-adic
fields. As it is explained in \cite[page 30 and page
56]{Shimurabook1} for each even $n$ there exits, up to isomorphism,
exactly two $n$-dimensional hermitian spaces. The unitary group
$U(V^+)$ corresponding to the one of them is quasi-split, it is
associated to the hermitian space with maximal isotropic space of
dimension $n/2$. We write $U(V^-)$ for the other one. It corresponds
to the hermitian space with an anisotropic subspace of dimension 2
over $K$. For $n$ odd there is only one isomorphic class of unitary
groups for $K/F$. For a hermitian space $V$ we define
$\epsilon(V)=\pm 1$ if $dim_K(V)$ is even and $V \cong V^{\pm}$ and
$\epsilon(V)=1$ if $dim_K(V)$ is odd. Now we consider the
archimidean case. We pick complex hermitian space $(V,\phi)$. If
$dim_\mathbb{C}V$ is odd we set $\epsilon(V)=1$. If
$dim_\mathbb{C}V$ is even, then if we write $U(V) \cong U(p,q)$, we
set $\epsilon(V)=(-1)^{\frac{n}{2}-p}$.

We turn now to global considerations. We consider a totally real
field $F$ and totally imaginary quadratic extention $K$. Then we
have the following well known result

\begin{thm} Let $n$ be a natural number. For every place $v$ of $F$ that is not split in $K$ choose a
hermitian space $V_v$ of dimension $n$ associated to the extension
$K_v/F_v$ such that if we write $G_v$ for the corresponding unitary
group and define $\epsilon_v(G_v):=\epsilon_v(V_v)$, then
$\epsilon(G_v)=1$ for almost all $v$. Then, there exists a unitary
group $G$ over $F$ such that for each place $v$ of $F$ $G \otimes_F
F_v \cong G_v$ if and only if $\prod_{v}\epsilon_v(G_v)=1$.
\end{thm}

We note that the condition is trivial if $n$ is odd. Before we start
with the proof of the Proposition \ref{Hilbert 90} we need one more
lemma. That is,

\begin{lem}Let $F'/F$ be a finite Galois extension of $p$-adic fields such that
$(|G|,2)=1$ for $G:=G(F'/F)$. Let $(V,\phi)$ be a hermitian form
over $K/F$ and write $(V',\phi')$ for the base-changed hermitian
form to $F'$. Then we have $\epsilon(V)=\epsilon(V')$.
\end{lem}

\begin{proof} The statement is clear if $n$ is odd. So we are left
with the situation where $n$ is even. Now we can reduce everything
to the case $n=2$. Indeed, by definition, $\epsilon(V)=-1$ if $V$
has an anisotropic space of dimension two and $\epsilon(V)=1$ if
there is none. That means in order to prove the lemma, we need to
show that in our situation a (two-dimensional) anisotropic hermitian
$(V,\phi)$ over $F$ remains anisotropic after base change to $F'$.
But we can study this question by study the same question for
quaternion algebras (see \cite[pages 24-25]{Shimurabook1}, that is,
if we write $B/F$ for the corresponding to $V$ division algebra
(since $V$ is anisotropic), then the base changed quaternion algebra
$B':=B \otimes_F F'$ is a division algebra. But we know that $F'$
will split $B$, that is $B'$ is not a division algebra if and only
if there is an $F$-algebra $A$ that is similar to $B$, contains $F'$
and $[A:F']=[F':F]^2$. But that means that $A \cong M_m(B)$ for some
$m$, and hence $[A:F]=4m$. But $[A:F]=[A:F'][F':F]=[F':F]^3$. Since
we assume that $([F':F],2)=1$ we conclude the proof of the lemma.
\end{proof}

\begin{proof}(of Proposition \ref{Hilbert 90}) Let us write $(V,\phi)/K$ for the hermitian space
over $K$ that correspond to the group $G(F)$. Then the space $(V':=V
\otimes_K K', \phi':=\phi_K \otimes K')$ correspond to the group
$G(F')$. Then we know that the group $H^1(\Gamma,G(F'))$ classifies
classes over $K$ of hermitian forms $(W,\theta)/K$ that become
isomorphic to $(V',\phi')$ over $K'$. Since the signature at the
archimedean places is determined by $\phi'$ we know that also the
signatures of the forms $\theta$ at infinite is fixed. So there is
only freedom at the finite places. If $n$ is odd there is nothing
more to prove. If $n$ is even, then we can use the previous lemma to
establish that $\epsilon_{v'}(V')$ determines $\epsilon_v(W)$ for
every $v$ under $v'$. Hence there is only one class that can be base
changed to $(V',\phi')$ and hence we conclude the proof of the
proposition.
\end{proof}

\section{Proof of the ``Torsion-Congruences'': The CM Method.}

We are now ready to prove the main result of this work, namely the
``torsion congruences'' for the motives that we described in the
introduction.

\subsection{Explicit Results I; the case $n=1$.} Before we start with our approach to prove the
``torsion-congruences'' in this case we mention that in
\cite{Bouganis4} we have tackled these congruences using another
method. There we have used the Eisenstein measure of
Katz-Hida-Tilouine. We start by fixing our setting in the most
interesting case, that of an elliptic curve with CM. As it is easily
seen our results extend to the case of Hilbert modular forms with
CM.

Let $E$ be an elliptic curve defined over $\Q$ with CM by the ring
of integers $\mathfrak{R}_0$ of a quadratic imaginary field $K_0$.
We fix an isomorphism $\mathfrak{R}_0 \cong End(E)$ and we write
$\Sigma_0$ for the implicit $CM$ type of $E$. Let us write
$\psi_{K_0}$ for the Gr\"{o}ssencharacter attached to $E$. That is,
$\psi_{K_0}$ is a Hecke character of $K_0$ of (ideal) type $(1,0)$
with respect to the CM type $\Sigma_0$ and satisfies
$L(E,s)=L(\psi_{K_0},s)$. We fix an odd prime $p$ where the elliptic
curve has good ordinary reduction. We fix an embedding
$\bar{\Q}\hookrightarrow \bar{\Q}_p$ and, using the selected CM
type, we fix an embedding $K_0 \hookrightarrow \bar{\Q}$. The
ordinary assumption implies that $p$ splits in $K_0$, say to $\p$
and $\bar{\p}$ where we write $\p$ for the prime ideal that
corresponds to the $p$-adic embedding $K_0 \hookrightarrow
\bar{\Q}\hookrightarrow \bar{\Q}_p$. We write $N_E$ for the
conductor of $E$ and $\mathfrak{f}$ for the conductor of
$\psi_{K_0}$.

We consider a finite totally real extension $F$ (resp. $F' \supset
F$) of $\Q$ which we assume unramified at the primes of $\Q$ that
ramify in $K_0$ and at $p$. We write $\mathfrak{r}$ (resp.
$\mathfrak{r}'$) for its ring of integers and we fix an integral
ideal $\mathfrak{n}$ of $\mathfrak{r}$ that is relative prime to $p$
and to $N_E$. Let $K$ (resp $K'$) be the CM-field $FK_0$ (resp.
$F'K_0=F'K$) and let $\mathfrak{R}$ (resp. $\mathfrak{R}'$) be its
ring of integers. Let $F(p^\infty\mathfrak{n})$ be the ray class
field of conductor $p^\infty \mathfrak{n}$ and let $F'$ be ramified
only at primes above $p$, hence $F \subset F' \subset
F(p^\infty\mathfrak{n})$. Furthermore, assume $F'/F$ to be cyclic of
order $p$ and that the primes of $F'$ that ramify in $F'/F$ are
split in $K'$. That is if we write $\theta_{F'/F}$ for the relative
different of $F'/F$ then
$\theta_{F'/F}=\mathfrak{P}\bar{\mathfrak{P}}$ in $K'$. We write
$\Gamma$ for the Galois group $Gal(F'/F) \cong Gal(K'/K)$. Note that
in both $F$ and $F'$ all primes above $p$ split in $K$ and $K'$
respectively. Finally we write $\tau$ for the nontrivial element
(complex conjugation) of $Gal(K/F) \cong Gal(K'/F')$ and we set
$g:=[F:\Q]$.

We now consider the base changed elliptic curves $E/F$ over $F$ and
$E/F'$ over $F'$. We note that the above setting gives the following
equalities between the $L$ functions,
\begin{equation}
L(E/F,s) = L(\psi_K,s),\,\,\,L(E/F',s)=L(\psi_{K'},s)
\end{equation}
where $\psi_K := \psi_{K_0} \circ N_{K/K_0}$ and $\psi_{K'}:= \psi_K
\circ N_{K'/K} = \psi_{K_0} \circ N_{K'/K_0}$, that is the
base-changed characters of $\psi_{K_0}$ to $K$ and $K'$.

We write $G_F$ for the Galois group
$Gal(F(p^{\infty}\mathfrak{n})/F)$ and
$G_{F'}:=Gal(F'(p^{\infty}\mathfrak{n})/F')$ for the analogue for
$F'$. Note that the above setting (the ramification of $F'$ over
$F$) introduces a transfer map $ver:G_F \rightarrow G_{F'}$.
Moreover we have an action of $\Gamma=Gal(F'/F)$ on $G_{F'}$ by
conjugation. We will shortly define measures $\mu_{E/F}$ of $G_F$
and $\mu_{E/F'}$ of $G_{F'}$ that interpolate the critical value at
$s=1$ of the elliptic curve $E/F$ and $E/F'$ respectively twisted by
finite order characters of conductor dividing
$p^\infty\mathfrak{n}$. These are simply defined by
$\mu_{E/F}:=\mu^{HLS}_{\psi}$ and $\mu_{E/F'}:=\mu^{HLS}_{\psi'}$,
where we have taken the measure $\mu^{HLS}_{\pi,\chi}$ constructed
above in the case $n=1$, $\pi$ trivial and $\chi=\psi$.

\begin{thm}\label{congruences1}We have the congruences
\begin{equation}
\int_{G_F}\epsilon \circ ver \,\,\,\, d\mu_{E/F} \equiv
\int_{G_{F'}}\epsilon\,\,\,\,\, d\mu_{E/F'} \mod{p\Z_p}
\end{equation}
for all $\epsilon$ locally constant $\Z_p$-valued functions on
$G_{F'}$ such that $\epsilon^{\gamma}=\epsilon$ for all $\gamma \in
\Gamma$, where
$\epsilon^\gamma(g):=\epsilon(\tilde{\gamma}g\tilde{\gamma}^{-1})$
for all $g \in G_{F'}$ and for some lift $\tilde{\gamma} \in
Gal(F'(p^\infty\mathfrak{n})/F))$ of $\gamma$.
\end{thm}

\begin{proof}The proof of this theorem is exactly the same as the proof
of Theorem \ref{congruences2} that we prove below for the case
$n=2$. One simply needs to set there $f=1$ and $\pi$ the trivial
representation. The rest of the proof is identical, so we defer the
proof for the next section.
\end{proof}

As it is explained in appendix in Theorem \ref{RitterWeissCong} a
remark of Ritter and Weiss allow us to conclude from Theorem
\ref{congruences1} the following theorem

\begin{thm}\label{torsion-congruences1} With notation as before we have that the
torsion-congruen ces hold, i.e.
\[
ver(\mu_{E/F}) \equiv \mu_{E/F'} \mod{T},
\]
where $T$ the trace ideal.
\end{thm}

We note here that the important improvement in comparison to the
previous result in \cite{Bouganis4} is that we do not need to make
any assumption on the relation between the various class groups of
$F,F',K$ and $K'$.

\subsection{Explicit Results II; the case $n=2$.} In this section we
explain our results in the case of $n=2$.

\textbf{Quaternion algebras and unitary groups in two variables:}
Following Harris \cite{Harris2} we explain a relation between the
unitary group in two variables and quaternion algebras. Let $D$ be a
quaternion algebra over a totaly real field $F$ and assume that
there is an embedding $i:K \hookrightarrow D$ for a CM field $K$
with $K^+=F$. We consider the algebraic group over $F$
\[
GU_{K}(D):= (H_{K} \times D^{\times})/H_{F}
\]
where $H_{K}= Res_{K/\Q}G_{m}/K$ and $H_{F}=Res_{F/\Q}G_m/F$ and
$H_{F}$ is embedded diagonally into $H_{K} \times D^{\times}$. Next
we will identify the group $GU_K(D)$ with a unitary group as its
notation suggests.

Let us write $\iota :D  \rightarrow D$ for a main involution of $D$
s.t the reduced norm and trace are related by $N_{D}(d)=Tr_D(d\cdot
d^\iota)$. We can then obtain the non-degenerate inner form
$(x,y)_D:=Tr_D(x\cdot y^\iota)$ on the four dimensional $F$-vector
space $D$ and define the orthogonal group $GO(D)$ as
\[
GO(D)= \{g \in GL_E(D) | (gx,gy)_D = \nu(g) (x,y)_D\}
\]
for some homomorphism $\nu : GO(D) \rightarrow H_F$. Further there
is a map $\rho:D^\times \times D^\times \rightarrow GO(D)$, defined
by
\[
\rho(d_1,d_2)(x):= d_1xd_2^{-1}, \,\,x \in D.
\]
The map $\rho$ has kernel $H_F$ embedded diagonally in $D^\times
\times D^\times$ and $\nu(\rho(d_1,d_2))=N_D(d_1\cdot d_2^{-1})$.
Let now consider $K$ as above i.e there is an embedding $i: K
\hookrightarrow D$ and $D$ splits over $K$. We define the group
$GU_K(D)$ as the subgroup of $K$-linear elements of $GO(D)$.
Actually the group $GU_K(D)$ is a the group of unitary similitudes
of the a hermitian form $(\cdot,\cdot)_{D,K}$ characterized uniquely
from the properties that $(\cdot,\cdot)_D=
Tr_{K/F}(\cdot,\cdot)_{D,K}$ and $(x,y)_{D,K}=x\cdot \bar{y}$ for
$x,y \in K$. Now the relation with our previous considerations is
that the map $\rho$ restricted to $H_K \times D^\times$ has image in
$GL_K(D) \cap GO(D)$ and induces an isomorphism $(H_K \times
D^\times)/H_F \cong GU_K(D)$. Finally when the quaternion algebra is
unramified at all infinite places then the hermitian form
$(\cdot,\cdot)_{D,K}$ is positive definite.

\textbf{The setting:} Now we recall the setting that we are
interested in. We consider a Hilbert cuspidal form $f$ of $F$, which
is a newform. We write $N_{f}$ for its conductor. We assume that
$N_f$ is square free and relative prime to $p$. We now impose the
following assumptions on $f$.
\begin{enumerate}
\item $f$ has a trivial Nebentypus.
\item There exists a finite set $S$ of finite places of $F$ such that
we have (i) $ord_v(N_f) \neq 0$ for all $v \in S$, (ii) for $v \in
S$ we have that $v$ is inert in $K$ and finally (iii) $\sharp S +
[F:\mathbb{Q}]$ is even.
\end{enumerate}
Let us write $D/F$ for the totally definite quaternion algebra that
we can associate to the set $S$, i.e. $D$ is ramified at all finite
places $v \in S$ and also at all infinite places. Note that our
assumptions imply that there exist an embedding $K \hookrightarrow
D$. If we write $\pi_f$ for the cuspidal automorphic representation
of $GL_2(\mathbb{A}_F)$ associated to $f$ then our assumptions imply
that there exist a Jacquet-Langlands correspondence $\pi:=JL(\pi_f)$
to $D^\times(\mathbb{A}_F)$. As we explained above there exists an
isomorphism
\[
(D^\times \times K^\times)/F^\times \cong GU(\theta)(F)
\]
for some totally definite two dimensional Hermitian form
$(W,\theta)$. In particular, since the representation $\pi$ is taken
of trivial central character we can consider the representation
$\mathbf{1} \times \pi$ of $(D_{\mathbb{A}_F}^\times \times
\mathbb{A}^\times_K)/\mathbb{A}^\times_F$, where $\mathbf{1}$ the
trivial representation (character) on $\mathbb{A}^\times_K$. In
particular $\pi$ induces an automorphic representation, by abuse of
notation we denote it again with $\pi$, on $GU(\theta)$ and by
restriction to $U(\theta)$. Moreover it is known that
\[
L(\pi,s)=L(BC(\pi'),s),
\]
where $BC(\pi_f)$ is the base-change of $\pi_f$ from
$GL_2(\mathbb{A}_F)$ to $GL_2(\mathbb{A}_K)$. By abusing the
notation we keep writing $f$ for the hermitian form that corresponds
to the representation $\pi$ of the definite unitary group
$U(\theta)$. It is determined by a finite set of values. Then the
main theorem then is the following:

\begin{thm}\label{torsion-congruences2} Assume the following conditions
\begin{enumerate}
\item The $p$-adic realizations of $M(\pi)$ and $M(\psi)$ have
$\mathbb{Z}_p$-coefficients.
\item The prime $p$ is unramified in $F$ (but may ramify in $F'$).
\item If we write $K^*$ for the reflex field of $(K,\Sigma)$ (note
that this is also the reflex field of $(K',\Sigma')$), then for the
primes $\mathfrak{p}$ above $p$ in $K^*$ we have
$N_{K^*/\mathbb{Q}}(\mathfrak{p})=p$.
\item If we write $(f,f)$ for the standard normalized
Peterson inner product of $\pi$, then $(f,f)$ has trivial valuation
at $p$
\end{enumerate}
Then we have that the torsion congruences hold true for the motive
$M(\psi)/F \otimes M(f)/F$, where $\phi$ a Hecke character of $K$
infinity type $-k\Sigma$ with $k \geq 2$.
\end{thm}

We now construct measures $\mu_{F}$ (resp. $\mu_{F'}$) on
$G:=G^{ab}_F(p^\infty)$ (resp. $G':=G_{F'}(p^\infty)$) that
interpolate modified values of $L(\pi_f \times \pi_{\psi},1)$ (resp
$L(\pi'_f \times \pi_{\psi'})$) where $\pi_{\psi}$ (resp.
$\pi_{\psi'}$) is the cuspidal automorphic representation obtained
by automorphic induction of $\psi$ (resp. $\psi'$). We consider the
CM algebras $Y=K \oplus K$ and $Y'=K'\oplus K'$. Then we define the
measures by
\[
\int_{G}\chi \, \mu_F:=\frac{\int_{Gal(K(p^\infty)/K)} \tilde{\chi}
\mu^{HLS}_{\pi,\psi}}{\Omega_p(Y,\Sigma)} \,\,\,\left(resp.
\,\,\,\int_{G'}\chi\, \mu_F:=\frac{\int_{Gal(K'(p^\infty)/K')}
\tilde{\chi}'\,
\mu^{HLS}_{\pi',\psi'}}{\Omega_p(Y',\Sigma')}\right),
\]
where $\tilde{\chi}$ (resp. $\tilde{\chi}'$ is the base change of
$\chi$ (resp. $\chi'$ from $F$ (resp. $F'$) to $K$ (resp. $K'$).

We will prove the following theorem. As we explain in the appendix
(see Theorem \ref{RitterWeissCong}) the following theorem implies
Theorem \ref{torsion-congruences2}.

\begin{thm}\label{congruences2} Assume that $\pi_f$ and $\psi$ have coefficients in $\mathbb{Q}_p$. Let $\epsilon$ be a locally constant $\mathbb{Z}_p$-valued function of $G_{F'}$ with $\epsilon^\gamma=\epsilon$, for $\gamma \in
Gal(F'/F)$. Then we have the congruences
\[
(f,f)\int_{G'}\epsilon(g)\, ver(\mu_{F})(g) \equiv
(f',f')\int_{G'}\epsilon(g)\mu_{F'}(g)\,\, \mod{p}.
\]
In particular we have that if $(f,f)$ is a $p$-adic unit then the
torsion congruences hold for $M(\pi)/F \otimes M(\psi)/F$ and the
extension $F'/F$.
\end{thm}

In order to prove this theorem we have to understand how the modular
forms attached to definite quaternion algebras behave under base
change.

\textbf{Jacquet-Langlands correspondence and base change:} Our main
goal in this section is to understand the relation of the
automorphic forms that correspond to the automorphic representations
of the division algebras $D$ and $D'$, where $D'=D \otimes_F F'$.

\begin{prop} Let $\pi := \pi_D$ be an automorphic representation of $D^\times_{\mathbb{A}_F}$ obtained by the Jacquet-Langlands correspondence from $\pi_f$. Let also
$\pi':=\pi_{D'}'$ be the base-changed representation to to
${D'}^\times_{\mathbb{A}_{F'}}$. Then we have the following explicit
descriptions of the local representations,
\begin{enumerate}
\item For place $v$ of $F$ that split in $F'$, $v$ does not divide $N$ and $\pi_v = \pi(\chi_1,\chi_2)$, principal series.
Then $\pi'_v= \otimes_{w | v} \pi_w$ with $\pi_w' =
\pi_v(\chi_1,\chi_2)$ under the identification $GL_2(F_v) =
GL_2(F_w')$.
\item For place $v$ of $F$ that is inert in $F'$, $v$ does not divide
$N$ and $\pi_v = \pi(\chi_1,\chi_2)$, principal series. Then
$\pi_v'=\pi(BC(\chi_1),BC(\chi_2))$ where $BC(\chi_i):=\chi_i \circ
N_{F'/F}$.
\item For place $v$ of $F$ that ramifies in $F'$, (by assumption) $v$
does not divide $N$ and $\pi_v = \pi(\chi_1,\chi_2)$, principal
series. Then $\pi_v'=\pi(BC(\chi_1),BC(\chi_2))$.
\item For place $v$ of $F$ that splits in $F'$, divides $N$, $D$ is unramified at
$v$, and $\pi_v=\pi(\chi_1,\chi_2)$ (principal or special series)
then $\pi'_v= \otimes_{w | v} \pi_w$ with $\pi_w' =
\pi_v(\chi_1,\chi_2)$ under the identification $GL_2(F_v) =
GL_2(F_w')$.
\item For place $v$ of $F$ that is inert in $F'$, divides $N$, $D$ is unramified at
$v$ and  $\pi_v=\pi(\chi_1,\chi_2)$ (principal or special series)
then $\pi_v'=\pi(BC(\chi_1),BC(\chi_2))$ where $BC(\chi_i):=\chi_i
\circ N_{F'/F}$.
\item For place $v$ of $F$ that splits in $F'$, divides $N$, $B$ ramifies at $v$, and $\pi_v$ is the one dimensional
representation of $D_v^{\times}$ given by $\pi_v= \psi \circ
N_{D_v/F_v}$,where $\psi$ a character of $F_v^{\times}$ and
$N_{D_v/F_v}$ the reduced norm of $B$ (as our conductor is square
free we have only special representations under J-L). Then $\pi_v'=
\otimes_{w|v}\pi_w'$ where $\pi_w'=\pi_v$ under the identification
$D_w'=(D\otimes_F F')_w=D_w \otimes_{F_v} F_w' = D_w$.
\item For place $v$ of $F$ that is inert in $F'$, divides $N$, $D$ ramifies at $v$ and
$\pi_v$ is the one dimensional representation of $D_v^{\times}$
given by $\pi_v= \psi \circ N_{D/F}$. Then $\pi_v' = \psi \circ
N_{F_v'/F_v} \circ N_{D'_v/F'_v}$.
\end{enumerate}
\end{prop}

After this description of base change we pick some specific
normalized vectors of the representations $\pi_D$ and $\pi_{D'}$.
Let us write $S$ for the set of bad-places, that is places where the
representation is ramified, and the places of $F$ above $p$. This is
a finite set of places of $F$. Let us write $\pi=\pi_D = \bigotimes
\pi_{D,v}$ and we describe a normalized vector $\phi = \otimes
\phi_v=\otimes_{v \not \in S}\phi_v \otimes \otimes_{v \in S}\phi_v$
with $\phi \in \pi_D$. At places $v \not \in S$, we have $ \pi_{v} =
\pi(\chi_1,\chi_2)$ with $\chi_1,\chi_2$ unramified, we pick
$\phi_v$ as the canonical spherical of this principal series
representation. If we write the Iwasawa decomposition
$GL_2(F_v)=B(F_v) K_v$ then it is defined as follows
\[
\phi_v(bk) := \delta_B^{1/2}\chi(b), \,\,\, b \in B,\,\,\,k\in K
\]
where $\chi : B \rightarrow \mathbb{C}^\times$ is the character
induced from $\chi_1,\chi_2$ on the Borel subgroup and $\delta_B$ is
given by $\delta_B(b)=|y_2/y_1|_F$ with $y_1,y_2 \in F$ such that
\[
b=\left(
    \begin{array}{cc}
      1 & x \\
      0 & 1 \\
    \end{array}
  \right)\left(
           \begin{array}{cc}
             y_1 & 0 \\
             0 & y_2\\
           \end{array}
         \right).
\]
For the other places $v \in S$ we pick $\phi_S:=\otimes_{v \in
S}\phi_v$ such that $\phi_S(1)=1$. In the same way we pick the
normalized vector $\phi' = \otimes_v \phi_v'$ of $\pi'$, where we
note that $S'$ is the induced by $S$ set of places of $F'$. Then we
prove,

\begin{thm} \label{base-change} Let $\phi = \otimes \phi_v$ be the normalized vector of $\pi =
\bigotimes_v \pi_v$ and $\phi' = \otimes \phi'_v$ the normalized
vector of $\pi' = \bigotimes \pi'_v$, where $\pi'$ the base change
of $\pi$. Then we have that,
\begin{enumerate}
\item For every $\gamma \in \Gamma = Gal(F'/F)$ we have
$\phi'((g')^\gamma)= \phi'(g')$ for $g' \in
G(\mathbb{A}_{F',\mathbf{h}})$ with $g'_{S'}=1$.
\item If $g' \in G(\mathbb{A}_{F',\mathbf{h}})$ such that $(g')^\gamma =
g'$ for $\gamma \in \Gamma=Gal(F'/F)$ and $g'_{S'}=1$ then
$\phi'(g')=\phi(g)^p$ where $g$ is the element of
$G(\mathbb{A}_{F,\mathbf{h}})$ such that $g \mapsto g'$ under
$G(\mathbb{A}_F) \hookrightarrow G(\mathbb{A}_{F'})$.
\end{enumerate}
\end{thm}

\begin{proof}

The first assertion follows from the fact that $\phi'$ comes from
base change. Indeed if we write $g' = \otimes g'_v$ then we have
$\phi'(g') = \otimes \phi'_v(g_v')$. Now we split the above tensor
product to places that are above places $w$ of $F$ and of course we
can restrict to the case $w \not \in S$. We consider $\phi'_w(g'_w)=
\otimes_{v|w}\phi'_v(g_v')$. Hence it is enough to prove that
$\phi'_w((g'_w)^\gamma)=\phi'_w(g'_w)$ for all $w$ finite places of
$F$ not in $S$. When $w$ split we have that all $\phi'_v$ for all
$v$ above $w$ are identical hence the claim follows. When $w$ is
inert then we have that $\phi'_v(g'_v)=\delta_{B'}^{1/2}(\chi_w
\circ N_{F'_v/F_w})(b'_v)$ for some Iwasawa decomposition
$g'_v=b'_vk'_v$. Hence
$\phi'_v((g_v')^\gamma)=\delta_{B'}^{1/2}(\chi_w \circ
N_{F'_v/F_w})((b'_v)^\gamma))$. Hence the claim follows also in this
case.

Now we turn to the second statement. We argue semi-locally as
before. We pick $w$ a place of $F$ not in $S$ and then consider all
the places $v$ of $F'$ above $w$. We consider first the case where
$w$ is split. The fact that $(g')^\gamma=g'$ implies that
$g'_{v_i}=g'_{v_j}=g_w$ for all $v_i,v_j$ above $w$. In particular
we have that $\phi'_w(g'_w) = \otimes_{v|w}\phi'_v(g'_v)=
\otimes_{v|w}\phi_w(g_w)=\phi_w(g_w)^p$.

Let us now consider the case of $w$ is inert. Then also in this case
it is easy to see that $(g')^\gamma=g'$ implies that $g_v' \in
GL_2(F_w) \hookrightarrow GL_2(F'_v)$ hence by definition of the
normalized vectors we have $\phi'_v(g'_v)=\phi(g_w)^p$.

\end{proof}

\textbf{Congruences between the special values:} Now we are ready
for the proof of Theorem \ref{congruences2}. As we have indicated
above what follows can be easily seen to be also a proof of Theorem
\ref{congruences1}. One needs only to replace $\pi$ and $f$ below
with the corresponding trivial object (automorphic representation or
hermitian modular form). We recall that the measure $\mu_F$ (resp.
$\mu_{F'}$) is defined by
\[
\int_{G}\chi \, \mu_F:=\int_{Gal(K(p^\infty)/K)} \tilde{\chi}
\mu^{HLS}_{\pi,\psi} \,\,\,\left(resp. \,\,\,\int_{G'}\chi\,
\mu_F:=\int_{Gal(K'(p^\infty)/K')} \tilde{\chi}'\,
\mu^{HLS}_{\pi',\psi'}\right),
\]
where $\tilde{\chi}$ (resp. $\tilde{\chi}'$ is the base change of
$\chi$ (resp. $\chi'$ from $F$ (resp. $F'$) to $K$ (resp. $K'$). We
keep writing $f$ (resp. $f'$) for the hermitian modular form that
corresponds to $\pi$ (resp. $\pi')$. Under our assumption on the
coefficients of the motive $M(f)$ these functions can be taken to be
$\mathbb{Z}_p$-valued. Moreover we note that we may be thinking $f$
as a function over the space $X \subset (\mathbb{A}^\times_{K,f}
\times D^\times_{\mathbb{A}_{F,f}})/\mathbb{A}^\times_{F,f}$, where
$X$ is defined as the image of $U(\theta)(\mathbb{A}_{F,f}) \subset
GU(\theta)(\mathbb{A}_{F,f})$ under the isomorphism
\[
(K^\times \times D^\times)/F^\times \cong GU(\theta)(F).
\]
Moreover we note that since we assumed that $\pi$ is of trivial
central character we have that the function $f$ is trivial on the
$\mathbb{A}^\times_K$ component. There we took simply the trivial
representation (character) $\mathbf{1}$.

From the construction of the $\mu^{HLS}_{\pi,\psi}$ (resp.
$\mu^{HLS}_{\pi',\psi'}$ we know that for a locally constant
function $\phi$ of $G$ (resp. $\phi'$ of $G'$) we have
\[
\int_{G}\phi \, \mu_F= \sum_{a,b \in \mathcal{B}_K}
E^{(a),\nu}_{\phi\psi}(\underline{A_a} \times \underline{A_{b}},j_1
\times j_2) f(b)f(a),
\]
\[
\left(resp. \,\,\int_{G'}\phi' \, \mu_{F'}= \sum_{a',b' \in
\mathcal{B}_{K'}} E^{(a'),\nu'}_{\phi'\psi'}(\underline{A_{a'}}
\times \underline{A_{b'}},j_1 \times j_2) f'(b')f'(a')\right).
\]

\begin{lem} Let $b'=\imath(b) \in \mathcal{B}_{K'}$ and
$a'=\imath(a) \in \mathcal{B}_{K'}$ for $a,b \in \mathcal{B}$. Then
we have that
\[
Frob_p(E^{(a),\nu}_{\phi\psi^p})(\underline{A_a} \times
\underline{A_{b}},j_1 \times j_2) \equiv
E^{(a'),\nu'}_{\phi'\psi'}(\underline{A_{a'}} \times
\underline{A_{b'}},j'_1 \times j'_2) \mod{p}
\]
for $\phi'$ a locally constant $\mathbb{Z}_p$-valued function on
$G'$ such that ${\phi'}^\gamma = \phi'$ and $\phi:=\phi' \circ ver$.
\end{lem}

\begin{proof} We have already proved that
\[
Frob_p(E^{(a),\nu}_{\phi\psi^p}) \equiv Res^{K'}_K
E^{(a'),\nu'}_{\phi'\psi'} \mod{p}
\]
Then the lemma follows by observing that
\[
Res^{K'}_K E^{(a'),\nu'}_{\phi'\psi'}(\underline{A_a} \times
\underline{A_{b}},j_1 \times
j_2)=E^{(a'),\nu'}_{\phi'\psi'}(\underline{A_{a'}} \times
\underline{A_{b'}},j'_1 \times j'_2).
\]
\end{proof}

\begin{cor} Keep the notation of $\phi$ and $\phi'$ as before.
Then we have the congruences
\[
\left(\sum_{a,b \in \mathcal{B}_K}
E^{(a),\nu}_{\phi\psi^p}(\underline{A_a} \times
\underline{A_{b}},j_1 \times j_2)
f(b)f(a)\right)^{\Phi_{\mathfrak{p}}} \equiv
\]
\[
\sum_{a',b' \in \imath(\mathcal{B}_{K})}
E^{(a'),\nu'}_{\phi'\psi'}(\underline{A_{a'}} \times
\underline{A_{b'}},j_1 \times j_2) f'(b')f'(a') \mod{p},
\]
where $\Phi_\mathfrak{p}$ is an in Theorem \ref{reciprocityP}.
\end{cor}
\begin{proof}
From the theory of complex multiplication (Theorem
\ref{reciprocityP}) and the assumptions on $\phi$ and $f$ we have
that
\[
\sum_{a,b \in \mathcal{B}_K}
Frob_p(E^{(a),\nu}_{\phi\psi^p})(\underline{A_a} \times
\underline{A_{b}},j_1 \times j_2) f(b)f(a)=
\]
\[
\left(\sum_{a,b \in \mathcal{B}_K}
E^{(a),\nu}_{\phi\psi^p}(\underline{A_a} \times
\underline{A_{b}},j_1 \times j_2)
f(b)f(a)\right)^{\Phi_{\mathfrak{p}}}
\]
We recall that we have proved (see Theorem \ref{base-change}) that
$f'(\imath(a))\equiv f(a) \mod{p}$. Also we note the fact that
$\psi'\circ ver = \psi^p$. Then the corollary follows from the lemma
above.
\end{proof}

\begin{prop}Let $\epsilon$ be a locally constant function such that
$\epsilon^\gamma=\epsilon$ for all $\gamma \in \Gamma$. Then for all
$a,b \in \mathcal{B}$ we have that
\[
E^{(a'),\nu'}_{\epsilon\psi'}(\underline{A_{a'}} \times
\underline{A_{b'}},j'_1 \times j'_2) =
E_{\epsilon\psi'}^{{(a')}^{\gamma},\nu'}(\underline{A^\gamma_{b'}}
\times \underline{A^\gamma_{a'}},j'_1 \times j'_2),
\]
where $A^\gamma_{a'}:=A_{{a'}^\gamma}$ and similarly
$A^\gamma_{b'}:=A_{{b'}^\gamma}$.
\end{prop}
\begin{proof} We write the locally constant function $\epsilon$ as a
sum of finite characters. That is, $\epsilon=\sum_j c_j \chi_j$ with
$c_j \in \mathbb{Q}(\epsilon,\chi_j)$. Now the fact that
$\epsilon^\gamma = \epsilon$ implies that this sum is of the form
\[
\epsilon = \sum_{i}c_i \chi_i + \sum_k c_k\sum_{\gamma \in
\Gamma}\chi_k^\gamma,
\]
where for the first sum we have $\chi_i^\gamma = \chi_i$, that is
$\chi_i$ comes from base change from $K$. From the definition of the
Eisenstein series we have
\[
E_{\epsilon\psi}^{\nu'}(\underline{A^\gamma_{a'}} \times
\underline{A^\gamma_{b'}},j'_1 \times
j'_2)=\left(\frac{\Omega_p(Y,\Sigma)}{\Omega_\infty(Y,\Sigma)}\right)^kE(({a'}^\gamma,{b'}^\gamma),\epsilon\psi'),
\]
where $\phi$ is the sections that we have constructed in section 3.
But we know that
\[
E(x,\phi,\epsilon\psi) = \sum_{\alpha \in
A}\mu_{\epsilon\psi}(\alpha x) \varepsilon(\alpha
x)^{-s}|_{s=0},\,\,\,A:=P \setminus G.
\]
The invariance of $\epsilon$ with respect to the action of $\Gamma$
implies the invariance of $\mu_{\epsilon\psi'}$ with respect to the
action of $\Gamma$. Indeed we recall that for a character $\chi$ we
have defined $\mu^{(\chi)}:=\prod_{v \in \mathbf{h} \cup
\mathbf{a}}\mu_v^{(\chi_v)}$ supported on
$P(\mathbb{A}_F)D(\mathfrak{c}) \cap
P(\mathbb{A}_F)w_nP(\mathbb{A}_F)$, where for $x= pw \in
P(\mathbb{A}_F)D(\mathfrak{c}) \cap
P(\mathbb{A}_F)w_nP(\mathbb{A}_F)$
\[
\mu_v^{(\chi_v)}(x_v)=\left\{
                  \begin{array}{ll}
                    \chi_v(\lambda_0(p_v)), & \hbox{$v \in \mathbf{h}$ and $v \not | \,\,\mathfrak{c}$;} \\
                    \chi_v(\lambda_0(p_v))\chi_v(\lambda_0(w_v)), & \hbox{$v \in \mathbf{h}$ and $v | \mathfrak{c}$;} \\
                    \chi_v(\lambda_0(p_v))j_{w_v}^k(\mathbf{i})^{-1}, & \hbox{$v \in \mathbf{a}$;} \\
                    f_{\Phi_\chi}(x_v), & \hbox{$v | p$.}
                  \end{array}
                \right.
\]
Hence $\mu_{\epsilon\psi'}(x)=\sum_j c_j
\mu^{(\chi_j)\psi'}(x)=\sum_{i}c_i \mu^{(\chi_i)\psi'}(x)+ \sum_k
c_k\sum_{\gamma \in \Gamma}\mu^{(\chi_k^\gamma)\psi'}(x)$ (note that
${\psi'}^\gamma =\psi'$). It is clear that the ``trace'' part, that
is $\sum_k c_k\sum_{\gamma \in
\Gamma}\mu^{(\chi_k^\gamma)\psi'}(x)$, is invariant under the
operation of $\Gamma$. But also the other part is invariant thanks
to the fact that the characters $\chi_i$ (as well as $\psi'$) that
appear there are of the form $\tilde{\chi}_i \circ N_{K'/K}$, with
$\tilde{\chi}$ a finite order character of $G_K$. The only thing
that needs to be also remarked is that for $v | p$ that is inert in
$F'$ we have that
$f_{\Phi_{\chi_{i,v}}}(x^\gamma_v)=f_{\Phi_{\chi_{i,v}}}(x_v)$, but
this follows easily from the definition of the section and the
invariance of $\chi_i$. Indeed we recall that if we write
$\chi_v=(\chi_1,\chi_2)$
\[
f_{\Phi_{\chi_v}}(x)=\chi_2(det(x)) |det(x)|^s
\int_{GL_{2n}(F'_v)}\Phi_{\chi_v}((0,Z)x) \chi_1\chi_2(det(Z))
|det(Z)|^{2s} d^\times Z,
\]
where we recall that $\Phi_{\chi_v}(X,Y)=\Phi_\mu(X)
\widehat{\Phi_\nu}(Y)$. It is easy to see that the functions
$\Phi_\mu(X)$ and $\Phi_\nu(Y)$ are invariant with respect to
$\Gamma$ and hence also
\[
\widehat{\Phi_\nu}(Y)=\int_{M_n(F'_v)}\Phi_\nu(X)\psi(\transpose{X}Y)dX
\]
since we have (note that $\psi^\gamma=\psi$)
\[
\widehat{\Phi_\nu}(Y^\gamma)=\int_{M_n(F'_v)}\Phi_\nu(X)\psi(\transpose{X}
Y^\gamma)dX=\int_{M_n(F'_v)}\Phi_\nu(X^{\gamma^{-1}})\psi(\transpose{X}^{\gamma^{-1}}
Y)dX=\widehat{\Phi_\nu}(Y).
\]
From these observations we conclude that
$f_{\Phi_{\chi_{i,v}}}(x^\gamma_v)=f_{\Phi_{\chi_{i,v}}}(x_v)$.

Now we also remark that the invariance of $\varepsilon$ follows from
its very definition (see \cite[page 95]{Shimurabook1}). Hence we
have that
\[
E(x^\gamma,\epsilon\psi') = \sum_{\alpha \in
A}\mu_{\epsilon\psi'}(\alpha x^\gamma) \varepsilon(\alpha
x^\gamma)^{-s}=\sum_{\alpha \in
A}\mu_{\epsilon\psi'}(\alpha^{\gamma^{-1}} x)
\varepsilon(\alpha^{\gamma^{-1}} x)^{-s}.
\]
But since $\gamma$ induces an automorphism of $A$ we have that the
last summation is equal to $E(x,\epsilon\psi')$. That is, we
conclude that
\[
E(({a'}^\gamma,{b'}^\gamma),\epsilon\psi')=E(({a'},{b'}),\epsilon\psi'),
\]
and hence also the proposition.
\end{proof}

An immediate corollary of this proposition is,

\begin{cor}Assume that $\phi'$ is a locally constant function with
${\phi'}^\gamma = \phi$. Then we have that
\[
E^{(a'),\nu'}_{\phi'\psi'}(\underline{A_{a'}} \times
\underline{A_{b'}},j'_1 \times j'_2) =
E^{({a'}^{\gamma}),\nu'}_{\phi'\psi'}(\underline{A^\gamma_{a'}}
\times \underline{A^\gamma_{b'}},j'_1 \times j'_2),
\]
where we also note that $A^\gamma_{a'}=A_{{a'}^\gamma}$ and
similarly $A^\gamma_{b'}=A_{{b'}^\gamma}$. In particular, as we have
seen that $f'({a'}^\gamma)=f'(a')$ and $f'({b'}^\gamma)=f'(b')$ we
have that
\[
E^{(a'),\nu'}_{\phi'\psi'}(\underline{A_{a'}} \times
\underline{A_{b'}},j'_1 \times j'_2)f'(b')f'(a') =
E^{({a'}^{\gamma}),\nu'}_{\phi'\psi'}(\underline{A^\gamma_{a'}}
\times \underline{A^\gamma_{b'}},j'_1 \times
j'_2)f'({b'}^{\gamma})f'({a'}^{\gamma}).
\]
\end{cor}

\begin{lem} We have the congruences
\[
(f,f) \frac{\int_{G}\phi \, \mu_F}{\Omega_p(Y,\Sigma)} \equiv
(f,f)\frac{\sum_{a,b \in \mathcal{B}_K}
E^{(a),\nu}_{\phi\psi^p}(\underline{A_a} \times
\underline{A_{b}},j_1 \times j_2)
f(b)f(a)}{\Omega_p(Y,\Sigma)^p}\times
\frac{\Omega_p(Y,\Sigma)^{\Phi_{\mathfrak{p}}}}{\Omega_p(Y,\Sigma)}\mod{p}
\]
\end{lem}
\begin{proof}Since $\psi^p \equiv \psi \mod{p}$ we have that
\[
(f,f)\sum_{a,b \in \mathcal{B}_K}
E^{(a),\nu}_{\phi\psi^p}(\underline{A_a} \times
\underline{A_{b}},j_1 \times j_2) f(b)f(a) \equiv (f,f)\sum_{a,b \in
\mathcal{B}_K} E^{(a),\nu}_{\phi\psi}(\underline{A_a} \times
\underline{A_{b}},j_1 \times j_2) f(b)f(a) \mod{\mathfrak{m}},
\]
where $\mathfrak{m}$ is the maximal ideal in $J_\infty$, which of
course contains $p$. Dividing by the unit $\Omega_p(Y,\Sigma)$ and
observing that
$\frac{\Omega_p(Y,\Sigma)^{\Phi_{\mathfrak{p}}}}{\Omega_p(Y,\Sigma)}\equiv
\frac{\Omega_p(Y,\Sigma)^p}{\Omega_p(Y,\Sigma)} \mod{\mathfrak{m}}$
we get the congruences
\[
\frac{(f,f)}{\Omega_p(Y,\Sigma)}\sum_{a,b \in \mathcal{B}_K}
 E^{(a)\nu}_{\phi\psi}(\underline{A_a} \times
\underline{A_{b}},j_1 \times j_2) f(b)f(a) \equiv
\]
\[
\frac{(f,f)}{\Omega_p(Y,\Sigma)^p}\sum_{a,b \in \mathcal{B}_K}
E^{(a),\nu}_{\phi\psi^p}(\underline{A_a} \times
\underline{A_{b}},j_1 \times j_2) f(b)f(a) \times
\frac{\Omega_p(Y,\Sigma)^{\Phi_{\mathfrak{p}}}}{\Omega_p(Y,\Sigma)}
\mod{\mathfrak{m}}
\]
Since both sides belong to $\mathbb{Z}_p$ we have that the
congruences are modulo $\mathfrak{m} \cap \mathbb{Z}_p = p$.
\end{proof}

\begin{thm} For $\phi'$ a locally
constant $\mathbb{Z}_p$-valued function on $G'$ such that
${\phi'}^\gamma = \phi'$ we have that
\[
(f,\check{f}) \frac{\int_{G}\phi \, \mu_F}{\Omega_p(Y,\Sigma)}
\equiv (f,\check{f}) \frac{\int_{G'}\phi' \,
\mu_{F'}}{\Omega_p(Y,\Sigma)^p}\mod{p},
\]
where $\phi:=\phi' \circ ver$. In particular if $(f,\check{f})  \in
\mathbb{Z}^{\times}_p$ then we have that
\[
\frac{\int_{G}\phi \, \mu_F}{\Omega_p(Y,\Sigma)} \equiv
\frac{\int_{G'}\phi' \, \mu_{F'}}{\Omega_p(Y,\Sigma)^p}\mod{p}
\]
\end{thm}

\begin{proof} The fact that $\imath: \mathcal{B}_K \hookrightarrow
\mathcal{B}_{K'}^\Gamma$ is a bijection implies that
\[
(f',f')\int_{G'}\phi' \, \mu_{F'} \equiv \left((f,f) \sum_{a,b \in
\mathcal{B}_K} E^{(a),\nu}_{\phi\psi^p}(\underline{A_a} \times
\underline{A_{b}},j_1 \times j_2)
f(b)f(a)\right)^{\Phi_{\mathfrak{p}}} \mod{p}
\]
where we have used the fact that $(f,f) \in \mathbb{Z}_p$ and under
our assumptions $(f,f) \equiv (f',f') \mod{p}$. Dividing by the unit
$\Omega_p(Y,\Sigma)^p$ we get
\[
(f',f')\frac{\int_{G'}\phi' \, \mu_{F'}}{\Omega_p(Y,\Sigma)^p}
\equiv \frac{1}{\Omega_p(Y,\Sigma)^p}\left(\Omega_p(Y,\Sigma)^p
\frac{(f,f)\sum_{a,b \in \mathcal{B}_K}
E^{(a),\nu}_{\phi\psi^p}(\underline{A_a} \times
\underline{A_{b}},j_1 \times j_2)
f(b)f(a)}{\Omega_p(Y,\Sigma)^p}\right)^{\Phi_{\mathfrak{p}}} \mod{p}
\]
But
\[
\frac{(f,f)\sum_{a,b \in \mathcal{B}_K}
E^{(a),\nu}_{\phi\psi^p}(\underline{A_a} \times
\underline{A_{b}},j_1 \times j_2) f(b)f(a)}{\Omega_p(Y,\Sigma)^p}
\in \mathbb{Z}_p
\]
hence
\[
(f',f')\frac{\int_{G'}\phi' \, \mu_{F'}}{\Omega_p(Y,\Sigma)^p}
\equiv
\left(\frac{\Omega_p(Y,\Sigma)^{\Phi_{\mathfrak{p}}}}{\Omega_p(Y,\Sigma)}\right)^p
\frac{(f,f)\sum_{a,b \in \mathcal{B}_K}
E^{(a),\nu}_{\phi\psi^p}(\underline{A_a} \times
\underline{A_{b}},j_1 \times j_2) f(b)f(a)}{\Omega_p(Y,\Sigma)^p}
\mod{p}
\]
But
$\left(\frac{\Omega_p(Y,\Sigma)^{\Phi_{\mathfrak{p}}}}{\Omega_p(Y,\Sigma)}\right)^p
\equiv
\frac{\Omega_p(Y,\Sigma)^{\Phi_{\mathfrak{p}}}}{\Omega_p(Y,\Sigma)}
\mod{p}$ hence we obtain
\[
(f',f')\frac{\int_{G'}\phi' \, \mu_{F'}}{\Omega_p(Y,\Sigma)^p}
\equiv
\frac{\Omega_p(Y,\Sigma)^{\Phi_{\mathfrak{p}}}}{\Omega_p(Y,\Sigma)}\times
\frac{(f,f)\sum_{a,b \in \mathcal{B}_K}
E^{(a),\nu}_{\phi\psi^p}(\underline{A_a} \times
\underline{A_{b}},j_1 \times j_2) f(b)f(a)}{\Omega_p(Y,\Sigma)^p}
\mod{p}
\]
But we have already shown that
\[
\frac{\Omega_p(Y,\Sigma)^{\Phi_{\mathfrak{p}}}}{\Omega_p(Y,\Sigma)}\times
\frac{(f,f)\sum_{a,b \in \mathcal{B}_K}
E^{(a),\nu}_{\phi\psi^p}(\underline{A_a} \times
\underline{A_{b}},j_1 \times j_2) f(b)f(a)}{\Omega_p(Y,\Sigma)^p}
\equiv (f,f) \frac{\int_{G}\phi \, \mu_F}{\Omega_p(Y,\Sigma)}
\mod{p},
\]
which concludes the proof of the theorem.
\end{proof}

Using this last theorem and the Theorem \ref{RitterWeissCong} of the
Appendix we conclude the torsion congruences of theorems
\ref{torsion-congruences1} and \ref{torsion-congruences2}.

\section{Appendix.}

We introduce the following general setting. Let $p$ be an odd prime
number. We write $F$ for a totally real field and $F'$ for a totally
real Galois extension with $\Gamma:=Gal(F'/F)$ of order $p$. We
assume that the extension is unramified outside $p$. We write
$G_F:=Gal(F(p^\infty)/F)$ where $F(p^\infty)$ is the maximal abelian
extension of $F$ unramified outside $p$ (may be ramified at
infinity). We make the similar definition for $F'$. Our assumption
on the ramification of $F'/F$ implies that there exist a transfer
map $ver:G_F \rightarrow G_{F'}$ which induces also a map
$ver:\mathbb{Z}_p[[G_F]] \rightarrow \mathbb{Z}_p[[G_{F'}]]$,
between the Iwasawa algebras of $G_F$ and $G_{F'}$, both of them
taken with coefficients in $\mathbb{Z}_p$. Let us now consider a
motive $M/F$ (by which we really mean the usual realizations of it
and their compatibilities) defined over $F$ such that its $p$-adic
realization has coefficients in $\mathbb{Z}_p$. Then under some
assumptions on the critical values of $M$ and some ordinarity
assumptions at $p$ (to be made more specific later) it is
conjectured that there exists an element $\mu_F \in
\mathbb{Z}_p[[G_F]]$ that interpolates the critical values of $M/F$
twisted by characters of $G_F$. Similarly we write $\mu_{F'}$ for
the element in $\mathbb{Z}_p[[G_{F'}]]$ associated to $M/{F'}$, the
base change of $M/F$ to $F'$. Then the so-called torsion congruences
read
\[
ver(\mu_{F}) \equiv \mu_{F'} \mod{T},
\]
where $T$ is the trace ideal in $\mathbb{Z}_p[[G_F']]^\Gamma$
generated by the elements $\sum_{\gamma \in \Gamma} \alpha^\gamma$
with $\alpha \in \mathbb{Z}_p[[G_F']]$.

We now provide a sketch of proof of the following theorem following
the work of Ritter and Weiss \cite{RW}.

\begin{thm}\label{RitterWeissCong} A necessary and sufficient condition for the ``torsion
congruences'' to hold is the following:

For every locally constant $\mathbb{Z}_p$-valued function $\epsilon$
of $G_{F'}$ satisfying $\epsilon^\gamma = \epsilon$ for all $\gamma
\in \Gamma$ the following congruences hold
\[
\int_{G_F} \epsilon \circ ver(x)  \mu_{F}(x) \equiv \int_{G_{F'}}
\epsilon(x) \mu_{F'}(x) \mod{p}\mathbb{Z}_p
\]
\end{thm}
As we said above we will prove the theorem following Ritter and
Weiss as in \cite{RW}. We start by recalling some of their notation
in (loc. cit.).

For a coset $x$ of an open subgroup $U$ of $G_F$ we set
\[
\delta^{(x)}(g)=\left\{
                 \begin{array}{ll}
                   1, & \hbox{$g \in x$;} \\
                   0, & \hbox{otherwise.}
                 \end{array}
               \right.
\]
We also for the given motive $M$ define the partial $L$-function
\[
L(M,\delta^{(x)},s)=\sum_{j}c_j L(M,\chi_j,s),
\]
where $\chi_j$ are finite order characters of $G_F/U$ such that
$\delta^{(x)}(g)=\sum_j c_j \chi_j(g)$ and $L(M,\chi_j,s)$ is the
standard twisted $L$-function of $M$ by $\chi_j$.

We call an open subgroup of $G_F$ admissible, if $\mathcal{N}_F(U)
\subset 1 + p\mathbb{Z}_p$, and define $m_F(U) \geq 1$ by
$\mathcal{N}_F(U)= 1 + p^{m_F(U)}\mathbb{Z}_p$. Here
$\mathcal{N}_F:G_F \rightarrow \mathbb{Z}_p$ stands for the
cyclotomic character. The following lemma is proved in \cite{RW}

\begin{lem} $\mathbb{Z}_p[[G_F]]$ is the inverse limit of the system
$\mathbb{Z}_p[G_F/U]/p^{m_F(U)}\mathbb{Z}_p[G_F/U]$ with $U$ running
over the cofinal system of admissible open subgroups of $G_F$.
\end{lem}

We now assume that the motive is critical and satisfy the usual
ordinarity assumption at $p$. Moreover we assume that its $p$-adic
realization has coefficients in $\mathbb{Z}_p$. Then conjecturally
there exits a measure $\mu_{F} \in \mathbb{Z}_p[[G_F]]$ such that
for any finite order character $\chi$ of $G_F$ we have
\[
\int_{G_F}\chi(g) \mu_F(g) = L^*(M,\chi) \in \mathbb{Z}_p[\chi],
\]
where $L^*(M,\chi,)$ involves the critical value $L(M,\chi,0)$ of
$M$ twisted by the finite order character $\chi$, some archimedean
periods related to $M$, a modification of the Euler factors above
$p$, $L_p(M,\chi,s)$, and finally some epsilon factors above $p$ of
the corresponding representation $M_p \otimes \chi$.

In the same spirit as above, if $\delta^{(x)}$ is the characteristic
function of a coset of an open subgroup $U$ we define
\[
L^*(M,\delta^{(x)}):=\sum_{j}c_j L^*(M,\chi_j).
\]
Then by the very definition of the element $\mu_F$ we have that its
image in $\mathbb{Z}_p[G_F/U]/ p^{m(U)}$ is given by
\[
\sum_{x \in G_F/U} L^*(M,\delta^{(x)}) \cdot x \mod{p^{m(U)}}.
\]

We finally need the following lemma, which is the analogues to Lemma
3 (2) of \cite{RW}.

\begin{lem}\label{Gamma-invariance} Let $y$ be a coset of a
$\Gamma$-stable admissible open subgroup of $G_{F'}$. Then
\[
L^*(M/F',\delta^{(y)}_{F'}) = L^*(M/F',\delta^{(y^\gamma)}_{F'}),
\]
for all $\gamma \in \Gamma$. In particular we have that $\mu_{F'}
\in \mathbb{Z}_p[[G_{F'}]]^\Gamma$.
\end{lem}
\begin{proof}The key observation is that $M/F'$ is the base-change
of $M/F$. Obviously it suffices to show the statement for finite
order characters. That is to show
\[
L^*(M/F',\chi) = L^*(M/F',\chi^\gamma),
\]
for $\chi$ a finite order character of $G_{F'}$. We recall that
\[
L^*(M/F',\chi)=e_p(M,\chi) \mathcal{L}_p(M,\chi)
\frac{L_{(S,p)}(M/F',\chi,0)}{\Omega_\infty(M)},
\]
where $L_{(S,p)}(M/F',\chi,0)$ is the critical value at $s=0$ of the
$L$-function $L(M/F',\chi,s)$ with the Euler factors at $S$ and
those above $p$ removed, where $S$ a finite $\Gamma$-invariant set
of places of $F'$. Moreover
$\mathcal{L}_p(M,\chi):=\prod_{v|p}\mathcal{L}_v(M,\chi)$ is a
modification of the Euler factor at places above $p$ and
$e_p(M,\chi):=\prod_{v|p}e_v(M,\chi)$, the local epsilon factors
above $p$.

We now observe that we have that
$L_{(S,p)}(M/F',\chi,s)=L_{(S,p)}(M/F',\chi^\gamma,s)$ since by the
inductive properties of the $L$-functions
\[
L_{(S,p)}(M/F',\chi,s)=L(M/F,ind^{F'}_{F}\chi,s)=L_{(S,p)}(M/F,ind^{F'}_F\chi^\gamma,s)=L(M/F',\chi^\gamma,s).
\]
Similarly one shows that
$\mathcal{L}_p(M,\chi)=\mathcal{L}_p(M,\chi^\gamma)$ and
$e_p(M,\chi)=e_p(M,\chi^\gamma)$ as the right sides of the equations
are nothing more than permutations of the left sides of the equation
(again the fact that $M/F'$ is the base change of $M/F$ is needed).
\end{proof}

The following lemma has also been shown by Ritter and Weiss.

\begin{lem} If $V$ is an admissible open subgroup of $G_{F'}$ and
$U$ an admissible open subgroup of $G_F$ in $ver^{-1}(V)$, then
$m_{F}(U) \geq m_{F'}(V)-1$
\end{lem}
In particular, as it is explained in \cite{RW}, one can conclude
from this lemma that the map $ver:\mathbb{Z}_p[[G_F]] \rightarrow
\mathbb{Z}_p[[G_{F'}]]$ induces a map
\[
\varprojlim_{U} \mathbb{Z}_p[G_F/U]/p^{m_F(U)} \rightarrow
\varprojlim_{V,\Gamma-stable} \mathbb{Z}_p[G_{F'}/V]p^{m_{F'}(V)-1}.
\]

Now we are ready to prove Theorem \ref{RitterWeissCong} following
the strategy of Ritter and Weiss in \cite{RW}.

\begin{proof}(of Theorem \ref{RitterWeissCong}) We consider the components of $\mu_{F'}$ and
$ver(\mu_F)$ in $\mathbb{Z}_p[G_{F'}/V]p^{m_{F'}(V)-1}$ for a
$\Gamma$-stable admissible open subgroup $V$ of $G_{F'}$. We note
that $ver(\mu_F)$ is the image under the transfer map of the $U$-
component of $\mu_F$ where $U:=ver^{-1}(V) \subseteq G_F$ which
contains $N:=ker(ver)$. These components are the images of
\begin{enumerate}
\item $\sum_{y \in G_{F'}/V}L^*(M/F',\delta^{(y)})y$,
\item $\sum_{x \in G_F/U}L^*(M/F,\delta^{(x)})ver(x)$
\end{enumerate}
in $(\mathbb{Z}_p[G_{F'}/V]/p^{m_{F'}(V)-1})^\Gamma$. We now show
that the sums in (i) and (ii) are congruent modulo $T(V)$, the image
of the trace ideal in
$(\mathbb{Z}_p[G_{F'}/V]/p^{m_{F'}(V)-1})^\Gamma$. We consider the
following two case

\textbf{$y$ is fixed by $\Gamma$:} Then $\delta^{(y)}_{F'}$ is a
locally constant function as in the Theorem \ref{RitterWeissCong},
hence we have
\[
L^*(M/F',\delta^{(y)}_{F'}) \equiv L^*(M/F,\delta^{(y)}_F \circ ver)
\mod{p}.
\]

If $y=ver(x)$ then $\delta^{(y)}_{F'} \circ ver=\delta^{(x)}_F$.
Then the corresponding summands in (i) and (ii) cancel out modulo
$T(V)$ since $p\alpha$ is a $\Gamma$ trace whenever $\alpha$ is
$\Gamma$-invariant.

If $y \not \in im(ver)$ then $\delta^{(y)}_{F'} \circ ver =0$ and
then again by the theorem we have $L^*(M/F',\delta^{(y)}_{F'})
\equiv 0$ modulo $p$, hence the corresponding summand vanishes
modulo $T(V)$.

\textbf{$y$ is not fixed by $\Gamma$:}Then we have by Lemma
\ref{Gamma-invariance} that
\[
L^*(M/F',\delta^{(y)}_{F'})=L^*(M/F',\delta^{(y^\gamma)}_{F'}),
\]
for all $\gamma \in \Gamma$. That means that the $\Gamma$ orbit of
$y$ yields the sum
\[
L^*(M/F',\delta^{(y)}_{F'}) \sum_{\gamma \in \Gamma}y^\gamma,
\]
which is in $T(V)$.
\end{proof}

\end{document}